\documentclass[11pt,leqno]{amsart}
\usepackage{amssymb}
\usepackage{amsmath}
\usepackage{amsthm}
\usepackage{verbatim}
\usepackage[top=1in, bottom=1in, left=1in, right=1in]{geometry}
\usepackage[stable]{footmisc}% foonote setting
\usepackage{indentfirst}% indentation of the 1st line of the 1st paragraph
\numberwithin{equation}{section}

\theoremstyle{definition}
\newtheorem*{definition*}{Definition}
\newtheorem{definition}{Definition}

\theoremstyle{remark}
\newtheorem{remark}{Remark}

\newtheorem*{note*}{Note}

\newtheorem*{example*}{Example}

\newtheorem*{question*}{Question}
\newtheorem*{blank*}{}

\theoremstyle{plain}
\newtheorem{theorem}{Theorem}
\newtheorem*{theorem*}{Theorem}

\newtheorem*{corollary*}{Corollary}
\newtheorem{lemma}{Lemma}
\newtheorem*{lemma*}{Lemma}
\newtheorem{proposition}{Proposition}
\newtheorem*{proposition*}{Proposition}
\newtheorem{conjecture}{Conjecture}
\newtheorem*{conjecture*}{Conjecture}

\newtheorem*{assump*}{Assumption}

\newcommand{\ba}{\mathbb A}

\newcommand{\bc}{\mathbb C}

\newcommand{\bk}{\mathbb K}

\newcommand{\bn}{\mathbb N}

\newcommand{\br}{\mathbb R}

\newcommand{\bt}{\mathbb T}

\newcommand{\mrg}{\mathrm G}

\newcommand{\mrm}{\mathrm M}
\newcommand{\mrn}{\mathrm N}

\newcommand{\mrp}{\mathrm P}

\newcommand{\mrr}{\mathrm R}

\newcommand{\ma}{\mathcal A}

\newcommand{\mc}{\mathcal C}

\newcommand{\mf}{\mathcal F}

\newcommand{\mh}{\mathcal H}

\newcommand{\mo}{\mathcal O}
\newcommand{\mpp}{\mathcal P} %special

\newcommand{\ms}{\mathcal S}

\newcommand{\muu}{\mathcal U} %special
\newcommand{\mv}{\mathcal V}

\newcommand{\rar}{\rightarrow}

\newcommand{\lrar}{\longrightarrow}

\newcommand{\hrar}{\hookrightarrow}

\newcommand{\wtilde}{\widetilde}

\newcommand{\pprime}{{\prime\prime}}

\newcommand{\smalltwomatrix[5]}{
  \left( \begin{smallmatrix}
            {#2} & {#3}\\
            {#4} & {#5}
         \end{smallmatrix}\right)}

 \DeclareMathOperator\Hom{Hom}
 
\DeclareMathOperator\Res{Res}   \DeclareMathOperator\Ind{Ind}
     \DeclareMathOperator\ad{ad}

\DeclareMathOperator\Gal{Gal} 
 
\DeclareMathOperator\re{Re}   %\Re, \Im has other usages.
 
\DeclareMathOperator\Sym{Sym} 
 
\DeclareMathOperator\Tr{Tr} 
 \DeclareMathOperator\sgn{sgn}

\newcommand{\leftup}[2]{{^{#1}}\mspace{-2.5mu}#2}

\newcommand{\isoto}{\stackrel{\sim}{\lrar}}

\newcommand{\dd}{\mathrm{d}}

\input{xy}
\xyoption{all}

\newcommand{\SO}{\mathrm{SO}}
\newcommand{\GO}{\mathrm{GO}}
\newcommand{\GSO}{\mathrm{GSO}}
\newcommand{\PGO}{\mathrm{PGO}}
\newcommand{\Sp}{\mathrm{Sp}}
\newcommand{\GSp}{\mathrm{GSp}}
\newcommand{\PGSp}{\mathrm{PGSp}}
\newcommand{\GL}{\mathrm{GL}}
\newcommand{\SL}{\mathrm{SL}}
\newcommand{\SK}{\mathrm{SK}}
\newcommand{\PGL}{\mathrm{PGL}}
\newcommand{\OO}{\mathrm{O}}
\newcommand{\U}{\mathrm{U}}
\newcommand{\GU}{\mathrm{GU}}
\newcommand{\PGU}{\mathrm{PGU}}
\newcommand{\HPS}{\mathrm{HPS}}
\newcommand{\Nm}{\mathrm{Nm}}
\newcommand{\Pa}{\mathrm{Pa}}
\newcommand{\So}{\mathrm{So}}
\newcommand{\disc}{\mathrm{disc}}
\newcommand{\pr}{\mathrm{pr}}
\newcommand{\Vol}{\mathrm{Vol}}
\newcommand{\ord}{\mathrm{ord}}

\newcommand{\ep}{{e^\prime}}

\begin{document}

\title{The Bessel Period Functional on $\SO_5$: The Nontempered Case}
\author[ Yannan Qiu]{Yannan Qiu}
%\address{Department of Mathematics, National University of Singapore, Block S17, 10 Lower Kent Ridge Road, Singapore 119076.}
\address{Department of Mathematics, Southern University of Science and Technology, 1088 Xueyuan Avenue, Shenzhen 518055, China}
\email{yannan.qiu@gmail.com} 
%\subjclass[2000]{11F67,11F70} 
%\thanks{partially supported  by grants DMS-0302043,DMS-0354353,  NSFC-10628103, and Vilas Life Cycle Professorship (Univ. Wisconsin) }
\maketitle
\begin{abstract}
For automorphic representations in the nontempered cuspidal spectrum of $\SO_5$, we establish a precise Bessel period formula, in which the square of the global Bessel period is decomposed as an Euler  product of regularized local period integrals of matrix coefficients. This is the first answer to the refined global Gross-Prasad problem for nontempered cuspidal representations, in the case of Bessel subgroups with nontrivial unipotent part. %This shows how the global Gross-Prasad conjecture can be refined for the nontempered cuspidal spectrum.
\end{abstract}

%\setcounter{tocdepth}{1}
%\tableofcontents

The global Gan-Gross-Prasad conjecture \cite{gross_prasad92, gross_prasad94, gan_gross_prasad2012} predicts a deep relation between the non-vanishing of periods of automorphic forms and the non-vanishing of special values of automorphic $L$-functions.  It is important to refine this relation into a precise period formula for various possible arithmetic applications. Generally, it is expected that the square of the global period functional is the product of local period functionals up to a precise constant.

When the involved automorphic representations are tempered cuspidal, there has been a lot of important work in formulaing and establishing cases of period formulas, for example, \cite{wald85}, \cite{ichino08}, \cite{ii10}, \cite{wgan_ichino11}, \cite{zhangw14}, \cite{liuyf2016}.  One convenience of assuming the tempered cuspidal condition is that the local period functionals are readily available as local integrals of matrix coefficients.

On the other hand, very little work is done when the involved automorphic representations are nontempered cuspidal. This is partly because the local Gross-Prasad restriction problem is more complicated for nontempered local packets and partly because the local period functionals are not readily available due to the possible divergence of local integrals of matrix coefficients.

Our paper \cite{qiu14} is the first period formula paper dealing with nontempered cuspidal representations in the framework of Ichino-Ikeda conjecture \cite{ii10}. There, we proved a precise formula for the $\SO_4$-period of  nontempered cuspidal $\SO_5$-representations of Saito-Kurokawa type; we defined the local period functionals by using a local height function to regularize the local integrals of matrix coefficient. (Ichino \cite{ichino05} previously considered the $\SO_4$-period of Saito-Kurokawa forms of Level $1$ but his period formula is not in terms of the Saito-Kurokawa forms themselves but in terms of  their lifts to $\wtilde{\SL}_2$ and $\PGL_2$.)

The current paper considers the Bessle subgroups $R=U\rtimes \SO_2$ of $\SO_5$, the paralell case of  $\SO_4$, and proves a precise Bessel period formula for the whole nontempered cuspidal spectrum of $\SO_5$ (including inner forms). We introduce below the general setting to describe our results.\\

Let $F$ be a number field and $\ba:=\ba_F$ be the ring of $F$-adels. %The original global Gross-Prasad conjecture is about an $F$-embedding $\SO_m\hrar \SO_{m+1}$ and the $\SO_m$-period of a cuspidal automorphic $\SO_{m+1}(\ba)\times \SO_{m}(\ba)$-representation $\Pi\boxtimes \Pi^\prime$. Waldspurger \cite{wald85} first established the precise period formula for $m=2$ and his formula was later extended to $m=3$ \cite{ichino08} after many people's work. Based on these two cases, Ichino and Ikeda \cite{ii10} conjectured a period formula for general $m$ by supposing $\Pi\boxtimes \Pi^\prime$ in the tempered cuspidal spectrum. In their formulation, the square of the global period is an Euler product of normalized local integrals of matrix coefficients and one role of the tempered condition is to ensure the absolute convergence of the local integrals. When $n\geq 5$, nontempered cuspidal representations appear in the discrete spectrum of $\SO_{n}(\ba)$. To handle the nontempered cuspidal spectrum, we suggest that the local integral of matrix coefficients be regularized by meromorphic continuation via a local height function with parameter $s$.
Let $\SO_n$ denote the special orthogonal group over a non-degenerate quadratic space of dimension $n$. When the $F$-rank of $\SO_{m+2r+1}$ is greater than or equal to $r$, there are corank-$r$ Bessel subgroups $\mrr=U_r\rtimes \SO_m$, where $U_r$ is certain unipotent subgroup.  One chooses a (nonimportant) generic character $\psi_r$ of $[U_r]$ and defines the Bessel period functional $P$ on an irreducible cuspidal $\SO_{m+2r+1}(\ba)\times \SO_m(\ba)$-representation $\Pi\boxtimes \Pi^\prime$ by
\[
P(f,f^\prime)=\int_{[U_r\rtimes \SO_m]}f(uh)f^\prime(h)\psi_r(u)\dd u\dd h,\quad f\in \Pi, f^\prime\in \Pi^\prime.
\]
The global period functional $P$ belongs to the space $\Hom_{\mrr(\ba)}(\Pi\otimes (\Pi^\prime\boxtimes \psi_r),\bc)$.

When $\Pi\boxtimes \Pi^\prime$ is tempered, the global Gross-Prasad conjecture \cite[Conj. 24.1, 26.1]{gan_gross_prasad2012} states:
\begin{itemize}
\item[(i)] $P$ is nonzero if and only if $L(\frac{1}{2},\Pi\times \Pi^\prime)$ is nonzero and the local homomorphism space $\Hom_{\mrr(F_v)}(\Pi_v\otimes (\Pi_v^\prime\boxtimes \psi_{r,v}),\bc)$ is nonzero at every local place,
\item[(ii)] When $L(\frac{1}{2},\Pi\times \Pi^\prime)$ is nonzero, the local root numbers $\epsilon(\frac{1}{2},\Pi_v\times \Pi_v^\prime)$ select out a unique inner form of $\SO_{m+2r+1}\times \SO_{m}$ so that the transfer of $\Pi\boxtimes \Pi^\prime$ to this inner form has nonzero global Bessel period.
\end{itemize}
On the local side, for tempered local representations, one can define a local Bessel period functional on $\Pi_v\otimes \Pi_v\otimes \Pi_v^\prime\otimes \Pi_v^\prime$ by integrating matrix coefficients  \cite{ii10, liuyf2016}, which is an element of $\Hom_{\mrr(F_v)}(\Pi_v\otimes\overline{\Pi}_v\otimes  (\Pi_v^\prime\boxtimes \psi_{r,v})\otimes (\overline{\Pi^\prime}_v\boxtimes \psi_{r,v}^{-1}),\bc)$. It is expected that
\begin{itemize}
\item[(iii)] When $\Pi_v\boxtimes \Pi_v^\prime$ is tempered irreducible unitary, the space $\Hom_{\mrr(F_v)}(\Pi_v\otimes (\Pi_v^\prime\boxtimes \psi_{r,v}),\bc)$ is nonzero if and only the local Bessel period functional is nonzero.
\end{itemize}
In this framework, the period formula problem is to express the nonzero global Bessel period functional as a product of local Bessel period integrals with explicit proportion constant.

When $\Pi\boxtimes \Pi^\prime$ is nontempered, the phenomenon is quite different. On the qualitative aspect, according to our study on $\SO_5\times \SO_2$ in this paper, we expect the following:
\begin{itemize}
\item[(1)] The assertion in (i) holds with $L(\frac{1}{2},\Pi\times \Pi^\prime)$ replaced by $L(0,\Pi,\Pi^\prime)$, where $L(s,\Pi,\Pi^\prime):=\frac{L(s+\frac{1}{2},\Pi\times \Pi^\prime)}{L(s+1,\Pi,\ad)L(s+1,\Pi^\prime,\ad)}$.
\item[(2)] The assertion in (ii) fails. When $L(0,\Pi,\Pi^\prime)\neq 0$, there may have more than one inner forms and more than one global representations in the global packet with nonvanishing global Bessel periods.
\item[(3)] In order for $\Hom_{\mrr(F_v)}(\Pi_v\otimes (\Pi_v^\prime\boxtimes \psi_{r,v}),\bc)\neq 0$, one needs to have $\ord_{s=0}L(s,\Pi_v,\Pi_v^\prime)\leq 0$ and nonvanishing regularized local Bessel period integrals. The former condition can be interpreted as $L(0,\Pi_v,\Pi_v^\prime)\neq 0$.
\end{itemize}

On the quantitative aspect, we treat the whole nontempered cuspidal spectrum of $\SO_5$. For $\SO(3,2)\cong \PGSp_4$, the classification of the discrete spectrum is presented in \cite{arthur04} and established in the recent book \cite{arthur13}, conditional on the stabilization of the twisted trace formula for $\GL_n$. According to this classification, the nontempered cuspidal representations fall in the Saito-Kurokawa packets and the Soudry packets (which include the Howe-Piatetski-Shapiro packets); these are CAP representations and were first constructed by the method of theta lifting \cite{PS83, hps79, Soudry88}. For the inner forms $\SO(4,1)$, the classification theorem is not established yet but these nontempered Arthur packets have been directly constructed by using theta correspondence in \cite{wgan2008} and a recent preprint \cite{gs2013}. We study the Bessel period functional on these packets.

Now let $(V,q)$ be a $5$-dimensional quadratic space over $F$ with positive Witt index and of determinant $1\cdot {F^\times}^2$. Decompose $V=Fv_+\oplus U\oplus Fv_-$, where $v_+,v_-$ are isotropic with $q(v_+,v_-)=1$ and $U$ is their orthogonal complement. Decompose further $U=Fv_0\oplus V_0$, where $v_0$ is non-isotropic and orthogonal to $V_0$. The Bessel subgroup associated to $v_0$ is
\[
\mrr=U\rtimes \SO(V_0).
\]
It is embedded into $\SO(V)$ via the parabolic subgroup $P$ stabilizing $v_+$ (see Sect. ~\ref{bgroup}). There is $\SO(V_0)\cong E_1^\times$ with $E=F(\sqrt{q(v_0)})$ and $\SO(V)$ is of type $\SO(3,2)$ or $\SO(4,1)$.

Let $\Pi$ be an irreducible cuspidal automorphic representation of $\SO(V)_\ba$ and $\chi$ be a character of $[\SO(V_0)]$. Choose a nontrivial character $\psi$ of $\ba/F$ and let $\psi_{v_0}(u):=\psi\big(q(v_0,u)\big)$ be the associated character on $[U]$. The global Bessel period functional $P$ on $\Pi$ with respect to $\chi\boxtimes\psi_{v_0}$ is
\begin{equation}\label{gbp}
P(f):=\int_{[U\rtimes \SO(V_0)]}f(uh)\chi(h)\psi_{v_0}(u)\dd u\dd h,\quad f\in \Pi.
\end{equation}
%Its nonvanishing is related to $L(0,\Pi,\chi)$, where $L(s,\Pi,\chi)=\frac{L(s+\frac{1}{2},\Pi\boxtimes \chi)}{L(s+1,\Pi,\ad)L(s+1,\chi_{E/F})}$.

Locally, consider the formal integral of matrix coefficients,
\begin{equation*}\label{lbp_u}
\mpp_v(f_{1,v},f_{2,v})=\int_{U(F_v)\rtimes \SO(V_0)_{F_v}}(\Pi_v(u_vh_{v})f_{1,v}, f_{2,v})\psi_{v_0}(u_v)\chi_v(h_v)\dd u_v\dd h_{v},\quad f_{i,v}\in \Pi_v.
\end{equation*}
When $\Pi_v$ is square-integrable and $V_0(F_v)$ is nonsplit, the local integral is absolutely convergent and defines an element of $\Hom_{\mrr(F_v)}(\Pi_v\otimes \big(\psi_{v_0,v}\boxtimes \chi_v\big),\bc)$. In all other situations, it is divergent and needs a suitable regularization that respects the action of $\mrr(F_v)$. We suggest that this $\mrr(F_v)$-integral be regarded as an iterated integral, with the $\SO(V_0)_{F_v}$-integration after the $U(F_v)$-integration, and that the two integrations be regularized separately according to the different natures of their divergence. Consider an $s$-family:
\[
\mpp_v(s;f_{1,v},f_{2,v})=\int_{\SO(V_0)_{F_v}}\big[\int_{U(F_v)}(\Pi_v(u_vh_{v})f_{1,v},f_{2,v})\psi_{v_0}(u_v)\dd u_v\big]\chi_v(h_v)\Delta(h_v)^s\dd h_{v},
\]
where $\Delta(h_v)$ is a height function on $\SO(V_0)_{F_v}$ derived from the doubling method (See Sect. ~\ref{subsec_lhf}).
\begin{itemize}
\item[(a)] We use the Fourier transform of tempered distributions to regularize the inner integral over $U(F_v)$, for all irreducible unitary $\Pi_v$ and for general $V$. (See Sect. ~\ref{subsec_lbp}.)

\item[(b)] The outer $\SO(V_0)_{F_v}$-integral is absolutely convergent when $\re(s)>>0$ because of the height function $\Delta(h)$ and it has meromorphic continuation to $s\in\bc$. When $\Pi_v$ is the local component of a global representation in the Saito-Kurokawa packets or Soudry packets, we find (b1) $\ord_{s=0}L(s,\Pi_v,\chi_v)\leq 0$ is a necessary condition for $\Hom_{\mrr(F_v)}(\Pi_v\otimes \big(\psi_{v_0,v}\boxtimes \chi_v\big),\bc)$ to be nonzero, (b2) when $\ord_{s=0}L(s,\Pi_v,\chi_v)\leq 0$, there is $\ord_{s=0}\mpp_v(s;f_{1,v},f_{2,v})\geq \ord_{s=0}L(s,\Pi_v,\chi_v)$ and the regularized local Bessel period functional on $\Pi_v\otimes \Pi_v$,
\begin{equation}\label{rlbp}
\mpp^\sharp_v(f_{1,v},f_{2,v})\overset{\triangle}{=}\frac{\mpp_{v}(s;f_{1,v},f_{2,v})}{\zeta_{F_v}(2)\zeta_{F_v}(4)L(s,\Pi_v,\chi_v)}\big|_{s=0},
\end{equation}
respects the action of $\mrr(F_v)\times \mrr(F_v)$. These properties are put as a local conjecture for general corank-$1$ Bessel subgroup in Section ~\ref{ss_lc}.
\end{itemize}

Now we state the main theorem. Let $\pi$ be an irreducible cuspidal automorphic representation of $\PGL_2(\ba)$, $\mu$ be a quadratic character of $\ba^\times/F^\times$, $K/F$ be a quadratic extension with $\Gal(K/F)=\{1,\iota\}$, and $\eta$ be a character of $\ba^\times_K/K^\times$ with $\eta|_{\ba^\times}=1$. Let $\Pi$ be an irreducible automorphic representation in the discrete spectrum of $\SO(V)$. We say $\Pi$ belongs to the Saito-Kurokawa packet $\SK(\pi,\mu)$ if its partial $L$-function $L^S(s,\Pi)$ is $L^S(s,\pi)L^S(s+\frac{1}{2},\mu)L^S(s-\frac{1}{2},\mu)$. The Soudry packet $\So(\eta)$ consists of those $\Pi$ with $L^S(s,\Pi)=L^S(s+\frac{1}{2},\eta)L^S(s-\frac{1}{2},\eta)$.

%When $\eta$ is nontrivial (resp. trivial) on $\ba_{K,1}^\times$, we say $\Pi$ belongs to the Soudry packet $\So(\eta)$ (resp. Howe-Piatetski-Shapiro packet $\HPS(\eta)$) if  $L^S(s,\Pi)=L^S(s+\frac{1}{2},\eta)L^S(s-\frac{1}{2},\eta)$.

\begin{theorem*}
Suppose that $\Pi$ is a cuspidal member of $\SK(\pi,\mu)$ or $\So(\eta)$.
\begin{itemize}
\item[(1)] $\Hom_{\mrr(F_v)}\big(\Pi_v\otimes (\psi_{v,v_0}\boxtimes \chi_v),\bc\big)$ is nonzero if and only if  $\mathrm{ord}_{s=0}L(s,\Pi_v,\chi_v)\leq 0$ and $\mpp^\sharp_v$ is nonzero on $\Pi_v\otimes \Pi_v$.

\item[(2)]  $P$ is nonzero on $\Pi$ if and only if  $\Hom_{\mrr(F_v)}(\Pi_v\otimes \big(\psi_{v,v_0}\boxtimes \chi_v\big),\bc)$ is nonzero for all $v$ and $L(0,\Pi,\chi)$ is nonzero.

\item[(3)] Suppose that $P$ is nonzero on $\Pi$, then $P\otimes \overline{P}=2^\beta \zeta_F(2)\zeta_F(4)L(0,\Pi,\chi)\prod_v 2^{-\beta_v}\mpp_v^\sharp$. When $\pi\in \SK(\pi,\mu)$, the constants $\beta$ is $0$ while the constants $\beta_v$ are all $1$; when $\pi\in \So(\eta)$, the constants are
\[
\beta=
\begin{cases}
1, &\eta^\iota\neq \eta,\\
2, &\eta^\iota=\eta.
\end{cases},\quad 
\beta_v=
\begin{cases}
1, &\eta_v^{\iota_v}\neq \eta_v,\\
2, &\eta_v^{\iota_v}=\eta_v.
\end{cases}
\]
For decomposable vectors $f_1, f_2\in \Pi$, there is $\mpp_v^\sharp(f_{1,v},f_{2,v})=2^{\beta_v}$ for almost all $v$. 
\end{itemize}
\end{theorem*}

%As far as we know, the above theorem is the first theorem that provides a precise Bessel period formula for nontempered cuspidal representations. The characterization of the local homomorphism space $\Hom_{\mrr(F_v)}\big(\Pi_v\otimes (\psi_{v,v_0}\boxtimes \chi_v),\bc\big)$ in terms of regularized local Bessel period integral is also new for nontempered local unitary representations.

The proof of the theorem is based on the theta lifting construction of the packets. First, the Saito-Kurokawa packets are constructed from the Waldspurger packets on $\wtilde{\SL}_2$ by using theta correspondence for $(\wtilde{\SL}_2,\SO(V))$. In this case, we transfer both the global and (regularized) local Bessel periods to integrals on $\SL_2$ and compare them there. The local period transfer is the more difficult part because one needs to carefully change the order of integration for a multivarible iterated integral; our tool is the inner product formula and the isometry property of quadratic Fourier transform established in Section ~\ref{qft}. The global period formula in this case is reduced to the Whittaker period formula for cuspidal $\wtilde{\SL}_2(\ba)$-representations, which was proved in \cite{baruch_mao07} for totally real $F$ and extended by the author with a different method to all number fields in \cite{qiuy2019}.

Second, the Soudry packets are constructed with the theta lifting from $\GO_2$ to $\GSp_4$ when $\SO(V)=\SO(3,2)$ and from a skew-Hermitian $\GU_1(D)$ to a Hermitian $\GU_2(D)$ when $\SO(V)=\SO(4,1)$, where $D$ is a nonsplit quaternion algebra over $F$. We transfer the global and local periods to integrals over $\GO_2$, $\GU_1(D)$ and compare them there. The tool for transferring local periods is the inner product formula and the Fourier inversion formula with the help of a local measure decomposition in Section ~\ref{ss_md2}. When $\eta=\eta^{\iota}$, the packet $\So(\eta)$ is the Howe-Piatetski-Shapiro packet and contains CAP representations with respect to the Borel subgroup. When $\eta\neq \eta^{\iota}$, the packet $\So(\eta)$ is the genuine Soudry packet and contains CAP representations with respect to the Klingen parabolic subgroup. We treat them uniformly in this paper.

%For general corank-$1$ Bessel subgroups, we tend to believe that $\Hom_{\mrr(F_v)}(\Pi_v\otimes (\Pi_v^\prime\boxtimes \psi_{r,v}),\bc)$ is nonzero if and only if the local $L$-function $L(s,\Pi_v,\Pi_v^\prime)$ is nonzero (including $\infty$) at $s=0$ and the regularized local Bessel period functional $P_v^\sharp$ is nonzero. On the global aspect, we expect a period formula of the shape $P\otimes \overline{P}=2^{\alpha+\beta}\prod_v 2^{-\beta}\mpp_v^\sharp$. (See Sect. ~\ref{ss_lc}.)

%In the end, we mention that Liu obtained in the recent paper \cite{liuyf2016} the Bessel period formula when $\Pi$ is in the (tempered) endoscopic packets of $\SO(3,2)$. The remaining study for the Bessel period functional on $\SO(3,2)$ is about the (general) tempered packets whose Arthur parameters are of stable semisimple type; this may better be approached by harmonic analysis on kernel functions, say via the method of relative trace formula \cite{martin_furusawa11, martin_furusawa13, martin_furusawa14} or tools alike.

The author would like to thank Gan Wee Teck for a number of very helpful conversations concerning the topic and National University of Singapore for hosting his visit while he worked on this paper.

\section{Preliminaries}\label{pre}

In this section, let $(V,q)$ be a non-degenerate $m$-dimensional quadratic space over a field $k$ of characteristic zero. Set $q(x,y):=q(x+y)-q(x)-q(y)$ and let $\Sym_n$ denote the space of symmetric $n\times n$-matrices. We also use $q$ to denote the map
\[
q: V^n\rar \Sym_n,\quad \vec{x}=(x_1,\cdots,x_n)\rar q(\vec{x}):=(q(x_i,x_j))_{n\times n}.
\]
Set $\Sym_n^\circ:=\Sym_n\cap \GL_n$ and  $V^{n,\circ}=q^{-1}(\Sym_n^\circ)$; specifically, put $V^\circ=V^{1,\circ}$.

Let $\psi$ be a nontrivial character of $k$ when $k$ is a local field of characteristic zero and a character of $\ba/F$ when $k=F$ is a number filed. For $a\in k^\times$, let $\psi_a$ be the $a$-twist of $\psi$ and $\chi_a=<a,\cdot>$ be the associated quadratic character, where $<,>$ denotes the Hilbert symbol.

\subsection{Corank-$1$ Bessel subgroup}\label{bgroup}

When $m\geq 3$ and $V$ has positive Witt index, choose a pair of isotropic vectors $v_+, v_-$ with $q(v_+,v_-)=1$, let $U$ be their orthogonal complement, and write the decomposition $V=Fv_+\oplus U\oplus Fv_-$. For a non-isotropic vector $v_0\in U$, let $V_0$ denote its orthogonal complement in $U$ and define the Bessel group associated to $v_0$,
\[
\mrr:=U\rtimes \SO(V_0).
\]
The embedding $\mrr_{v_0}\hrar \SO(V)$ is via the parabolic subgroup $P=U\rtimes M$ stabilizing $Fv_+$, where $M\cong \GL_1\times \SO(U)$ and $U$ is identified with the unipotent radical of $P$ by sending $u\in U$ to the following element,
\[
v_+\rar v_+,\quad u^\prime\in U\rar -q(u,u^\prime)v_++u^\prime,\quad v_-\rar -q(u)v_++u+v_-.
\]

\subsection{$\SO(3,2)$ and its inner forms}

A group of type $\SO(m_1,m_2)$ ($m_1\geq m_2$) refers to the group $\SO(V)$ over any quadratic space $(V,q)$ of dimension $m_1+m_2$ and Witt index $m_2$. In this paper, we study $\SO(3,2)$, $\SO(4,1)$ and will uitilize their alternative realizations in terms of symplectic groups and quaternion unitary groups.

\subsubsection{$\SO(3,2)$}\label{so32}

The following $5$-dimensional quadratic space is of Witt index $2$ and determinant $1\cdot{k^\times}^2$,
\begin{align}\label{so32r}
&V=\left\{Y=\smalltwomatrix{X}{x^\prime J_1}{-x^{\prime\prime}J_1}{\leftup{t}{X}}\in M_{4\times 4}(k) \ |\ x^\prime, x^{\prime\prime}\in F, \mathrm{Tr}(X)=0\right\},\\
\nonumber &q(Y)=\frac{1}{4}\mathrm{Tr}(Y^2)=x^\prime x^\pprime-\det X,
\end{align}
Set $J_n=\smalltwomatrix{}{1_n}{-1_n}{}$ and define $\GSp_{2n}:=\left\{g\in \GL_{2n}: gJ_n\leftup{t}{g}=\nu(g)J_n\right\}$, then $\GSp_4$ acts on $V$ by $g\circ Y=gYg^{-1}$ and this action leads to an isomorphism $\PGSp_4\isoto \SO(V)=\SO(3,2)$.

\subsubsection{The inner forms $\SO(4,1)$}\label{so41}

Given a quaternion algebra $D$ over $k$ with the main involution $X\rar \overline{X}$, one associates a $5$-dimensional quadratic space of determinant $1\cdot{k^\times}^2$,
\begin{align*}
&V_D=\left\{Y=\smalltwomatrix{X}{x^\prime1_D}{x^{\prime\prime}1_D}{\overline{X}}\in M_{2\times 2}(D)\ |\ x^\prime, x^{\prime\prime}\in F, \mathrm{Tr}_{D/k}(X)=0\right\},\\
&q_D(Y)=\frac{1}{4}\Tr_{D/k}\big[\Tr(Y^2)\big]=x^\prime x^\pprime-\Nm_{D/k}(X).
\end{align*}
The space $V_D$ is of Witt index $2$ (resp. $1$) when $D$ is split (resp. non-split). When $D$ varies, the groups $\SO(V_D)$ are distinctive and exhaust all groups of type $\SO(3,2)$ and $\SO(4,1)$.

Define the quaternion unitary group
\[
\GU_2(D):=\left\{g\in \GL_2(D): \leftup{t}{\bar{g}}\smalltwomatrix{}{1_D}{1_D}{}g=\nu(g) \smalltwomatrix{}{1_D}{1_D}{},\, \nu(g)\in k^\times\right\}.
\]
It acts on $V_D$ by $g\circ Y=gYg^{-1}$ and there is an induced isomorphism $\PGU_2(D)\isoto \SO(V_D)$.
\begin{remark}
When $D=M_{2\times 2}(F)$, the group $\GU_2(D)$ naturally sits inside $\GL_4$ and is equal to the similitude symplectic group defined as $\big\{g\in \GL_4: g\smalltwomatrix{}{J_1}{J_1}{}\leftup{t}{g}=\nu(g)\smalltwomatrix{}{J_1}{J_1}{}\big\}$.
\end{remark}

\subsection{$\OO_2$ and Skew-Hermitian $\GU_1(D)$}

They are used lated to define Soudry packets .

\subsubsection{$\OO_2$}\label{o2group}

For a quadratic algebra $k^\prime$ of $k$, let $(V_{k^\prime},q_{k^\prime})=(k^\prime,\Nm_{k^\prime/k})$ denote the associated quadratic space. There is $\SO(V_{k^\prime})={k^\prime_1}^\times$ and $\OO_2(V_{k^\prime})=\SO(V_{k^\prime})\rtimes \Gal(k^\prime/k)$. Write $\Gal(k^\prime/k)=\{1,\iota\}$, then Hilbert's Theorem 90 yields a short exact sequence $1\rar k^\times\rar {k^\prime}^\times\overset{\pr}{\rar} {k_1^\prime}^\times\rar 1$, where $\pr(x^\prime)={x^\prime}^\iota/x^\prime$ is the projection map.

Generally, for a $2$-dimensional quadratic space $V$ over $k$, write $k^\prime=k(\sqrt{-\det V})$ and consider the associated fiberation $q:V^2\rar \Sym_2$. Set $\Sigma_T:=\{X\in V^2:q(X)=T\}$; for $T\in q(V^{2,\circ})$, there is $\Sigma_T\cong \OO(V)={k_1^\prime}^\times$. Furthermore, 
\begin{align}\label{sm}
\nonumber q_{k^\prime}(V_{k^\prime}^{2,\circ})=&\{T: \exists\, g\in \GL_2(k)\ \text{s.t.}\ gT\leftup{t}{g}=\smalltwomatrix{1}{}{}{\disc(k^\prime)}\},\\
\sqcup_{a\in k^\times /\Nm({k^\prime}^\times)}aq_{k^\prime}(V_{k^\prime}^{2,o})=&\{T:\det T\in -\disc(k^\prime)\cdot {k^\times}^2\}.
\end{align}

\subsubsection{Skew-Hermitian $\GU_1(D)$}\label{gu1dgroup}

Let $D$ be a quaternion algebra over $k$ with the main involution $x\rar \bar{x}$. An $\epsilon$-Hermitian space over $D$ is a right $D$-module $\mv$ with a pairing $\mv\times \mv\rar D$ satisfying $(y,x)=\epsilon \overline{(x,y)}$ and $(xa,yb)=\overline{a}(x,y)b$. Define
\[
\GU(\mv)=\{g\in \GL_D(\mv): (gx,gy)=\nu(g)(x,y),\, \nu(g)\in k^\times\}
\]
and set $\U(\mv)=\{g\in \GU(\mv): \nu(g)=1\}$, $\PGU(\mv)=\GU(\mv)/k^\times$.

Let $D_0$ be the set of trace zero elements in $D$. For $e\in D_0$ with nonzero norm, let $\mv_e$ denote the $1$-dimensional skew-Hermitian space over $D$ whose underlying space is $D$ and whose pairing is given by $(1,1)=e$. To describe $\GU(\mv_e)$ and $\U(\mv_e)$, we put $k^\prime=k\oplus ke$, choose $e^\prime\in D_0$ with nonzero norm and satisfying $ee^\prime=-e^\prime e$, and write
\[
D=k^\prime\oplus k^\prime\ep=k\cdot 1\oplus ke\oplus ke^\prime\oplus kee^\prime,
\]
There is 
\[
\GU(\mv_e)={k^\prime}^\times \sqcup {k^\prime}^\times e^\prime,\quad \U(\mv_e)=
\begin{cases}
{k^\prime_1}^\times, &\text{$D$ is nonsplit},\\
{k^\prime_1}^\times\rtimes \mu_2, &\text{$D$ is split}.
\end{cases}
\]
Note the $\nu(\GU(\mv_e))=\Nm({k^\prime}^\times)\cup \ep^2\Nm({k^\prime}^\times)$.

\begin{remark}
$D$ is split if and only if $\ep^2\in \Nm({k^\prime}^\times)$. When $D$ is split, choose $h_{e^\prime}\in {k^\prime}^\times$ satisfying $\Nm(h_{e^\prime})=-\ep^2$, then the group $\mu_2$ occuring above can be realized as $\mu_2=\{1,h_{e^\prime}e^\prime\}$. Note that $\U(\mv_e)$ is a non-connected $k$-group and that it has the same set of $k$-rational points as its connected component $\U(\mv_e)^c$ when $D$ is nonsplit.
\end{remark}

Consider the associated fiberation $\wtilde{q}_e: \mv_e\rar D_0$, $x\rar (x,x)=\bar{x}e x$. Put $\mv_e^\circ=\{x\in \mv_e: (x,x)\neq 0\}$, $D^\circ=\{x\in D: \Nm(x)\neq 0\}$ and $D_0^\circ=D^\circ\cap D_0$. Set $\Sigma_y:=\{x\in \mv_e: (x,x)=y\}$; for $y\in \wtilde{q}_{e}(\mv_e^\circ)$, there is $\Sigma_y\cong \U(\mv_e)$. Furthermore,
\begin{equation}\label{sm2}
\sqcup_{a\in k^\times/\nu(\GU(\mv_e))}a\wtilde{q}_{e}(\mv_e^\circ):=\{y\in D_0^\circ: \Nm(y)\in \Nm(e){k^\times}^2\}.
\end{equation}

\subsection{Tamagawa Measure}\label{s_tmeasure}

For a reductive group $\mrg$ over a number field $F$, we write $[\mrg]=\mrg(F)\backslash \mrg(\ba)$ and always use the Tamagawa measure $\dd g$ on $\mrg(\ba)$.

(i) The Tamagawa measure on $\ba$ and $\ba^\times$ are $\dd t=\prod_v \dd t_v$ and $\dd^\times t= \frac{1}{\Res_{s=1}\zeta_F(s)}\prod_v \zeta_{F_v}(1)\frac{\dd t_v}{|t_v|}$ respectively, where $\dd t_v$ is the Haar measure on $F_v$ that is self-dual with respect to $\psi_v$.

(ii) For $(V,q)$ over $F$, the Tamagawa measure on $V(\ba)$ is $\dd x=\prod_v \dd x_v$, where $\dd x_v$ is the Haar measure on $V(F_v)$ that is self-dual with respect to the character of second degree $\psi_v\circ q_v$. %When $\dim V=1$ and $q(x)=ax^2$, $\dd x_v$ is the Haar measure that is self-dual with respect to $\psi_{v,2a}$.

(iii) For a quadratic algebra $E$ over $F$, let $c_E$ be $L(1,\chi_{E/F})$ when $E$ is nonsplit and $\Res_{s=1}\zeta_F(s)$ when $E$ is split. The Tamagawa measure over $\SO(V_E)_\ba=E_{1,\ba}^\times\cong \ba_E^\times/\ba^\times$ is $$\dd h=c_E^{-1}\prod_v L(1,\chi_{E_v/F_v})\dd^\sharp h_v,$$ where $\dd^\sharp h_v=\frac{\dd e_v}{|e_v|}/\frac{\dd t_v}{|t_v|}$ is the natural quotient measure on $E_{v,1}^\times\cong E_v^\times/F_v^\times$ and $\dd e_v, \dd f_v$ are self-dual measures on $E_v, F_v$ as in (i).

Write $\OO_m=\SO_m\rtimes \mu_2$. Extend the local measure $\dd h_v$ on $\SO_m(F_v)$ to $\OO_m(F_v)$ by giving $\mu_2(F_v)$ the probability measure $1$, then the product of the extended local measures is the Tamagawa measure on $\OO_m(\ba)$. There is $\mathrm{Vol}([\OO_m])=1$ when $m\geq 2$ and $\OO_m\neq \OO_{1,1}$.

(iv) Suppose that $D$ is a nonsplit quaternion algebra over $F$ and $e\in D_0$ has nonzero norm. Let $S_D$ denote the set of nonsplit places of $D$ and write $K=F\oplus Fe$, then
\[
\U(\mv_e)_\ba=\prod_v \U(\mv_e)_{F_v}=(\prod_{v\in S_D} K_{1,v}^\times)\times \big(\prod_{v\not\in S_D} K_{1,v}^\times\rtimes \mu_2\big)
\]
We follow \cite[Sect. 10.7]{gan_yu2000} to specify the Tamagawa measure on the non-connected group $\U(\mv_e)$. Let $\dd^\sharp h_{1,v}$ be the natural quotient measure on $K_{1,v}^\times$. When $v\in S_D$, $\U(\mv_e)_{F_v}=K_{1,v}^\times$ and we set $\dd^\sharp h_v=\frac{1}{2} \dd^\sharp h_{1,v}$; when $v\not\in S_D$, $\U(\mv_e)_{F_v}=K_{1,v}^\times\rtimes \mu_2$ and we let $\dd^\sharp h_v$ be the extension of $\dd^\sharp h_{1,v}$ by giving $\mu_2$ the probability measure. The Tamagawa measure on $\U(\mv_e)_\ba$ is $\dd h=c_K^{-1}\prod_v L(1,\chi_{K_v/F_v})\dd^\sharp h_v$. There is $\mathrm{Vol}([\U(\mv_e)])=2^{1-|S_D|}$.

\subsection{Measure decomposition}

We describe three measure decompositions.

\subsubsection{The fiberation $V\rar k$}\label{s_md1}

Define the $r$-sphere $S_r=\{x\in V:q(x)=r\}$, $r\in k$.

Now let $k$ be a local field and $\dd x$, $\dd r$ be the self-dual Haar measures on $V$, $k$ respectively. The map $q:V\rar k$ induces $\OO(V)$-invariant measures $\dd\sigma_r$ on $S_r$ ($r\in k^\times$) such that
\begin{equation}\label{md1}
\int_V f(x)\dd x=\int_{k^\times} \int_{S_r}f(x)\dd\sigma_r |r|^{\frac{m-2}{2}}\dd r,\quad f\in L^1(V).
\end{equation}
There is $\int_V f(X)dX=\int_{k^\times} \wtilde{f}(r)\dd r$, where $\wtilde{f}(r):=|r|^{\frac{m-2}{2}}\int_{S_r}f(x)\dd \sigma_r$, $r\in k^\times$.

We describe the measure $\dd \sigma_r$ for $r\in q(V)\cap k^\times$:
\begin{itemize}
\item[(i)] $\dim V=1$. $S_r\cong\OO_1(k)$ and $\dd \sigma_r=2|2|_k^{-1}\dd h$, where $\dd h$ is the probability measure on $\OO_1(k)$.

\item[(ii)] $\dim V=2$. $S_r\cong \SO(V)$ and $\dd \sigma_r=\dd h^\sharp$, where $\dd^\sharp h$ denotes the natural quotient measure on $\SO(V)={k^\prime_1}^\times\cong {k^\prime}^\times/k^\times$ and $k^\prime=k(\sqrt{-\det V})$.
\end{itemize}
%\begin{remark} When $\dim V\geq 3$, for $r\in q(V)\cap k^\times$, choose $x_r$ with $q(x_r)=r$ and let $x_r^\perp$ denote the orthogonal complement of $x_r$ in $V$, then there is a homeomorphism \[ \SO(V)/\SO(x_r^\perp)\isoto S_r,\quad \dot{h}\rar h\circ x_r. \] The measure $\dd \sigma_r$ is a scalar multiple of the quotient measure on $\SO(V)/\SO(x_r^\perp)$. \end{remark}

%\begin{remark} When $k=F$ is a number field and $\dim V\geq 2$, for $x_r\in V(F)$ with $q(x_r)=r\in F^\times$, there is $S_{r,v}\cong \SO(V)_{F_v}/\SO(x_r^\perp)_{F_v}$ at each local place. If $\dim V\geq 4$ and $\ba, V(\ba)$, $\SO(V)_\ba$, $\SO(x_r^\perp)_\ba$ are given the Tamagawa measures, then the product measure $\dd \sigma_r:=\prod_v \dd \sigma_{v,r_v}$ is equal to the quotient measure on $\SO(V)_\ba/\SO(x_r^\perp)_\ba$. This assertion follows from Weil's general study of Siegel-Weil formula in \cite{weil65} when $m\geq 5$ and can be verfied directly for $m=4$.\end{remark}

\subsubsection{The fiberation $V^2\rar \Sym_2$}\label{ss_md2}

Assume $\dim V=2$. Write $k^\prime=k(\sqrt{-\det V})$ and consider the fiberation $q:V^2\rar \Sym_2(k)$ in Section ~\ref{o2group}.

Suppose that $k$ is a local field. Let $\dd t$ be the self-dual Haar measure on $k$ and $\dd T=\dd t_1 \dd t_2 \dd t_3$ be the according measure at $T=\smalltwomatrix{t_1}{t_3}{t_3}{t_2}\in \Sym_2(k)$. Let $\dd x$ be the self-dual Haar measure on $V$ and $\dd X=\dd x_1\dd x_2$ be the according measure at $X=(x_1,x_2)\in V^2$. Recall that $\dd^\sharp h_1$ denotes the quotient measure $\frac{\dd x^\prime}{|x^\prime|}\big/\frac{\dd t}{|t|}$ on $\SO(V)\cong {k^\prime_1}^\times\cong {k^\prime}^\times/k^\times$. Write $\OO(V)={k^\prime}_1^\times\rtimes \Gal(k^\prime/k)$ and extend $\dd h_1^\sharp$ to a measure $\dd h^\sharp$ on $\OO(V)$ by giving $\Gal(k^\prime/k)$ the probability measure $1$.

\begin{lemma}\label{l_md2}
There exist $\OO(V)$-invariant measures $\dd \sigma_T$ on $\Sigma_T$ such that
\begin{equation}\label{md2}
\int_{V^2}F(X)\dd X=\int_{\Sym_2^\circ(k)} \int_{\Sigma_T}F(X)\dd \sigma_T |\det T|^{-\frac{1}{2}}\dd T,\quad F\in L^1(V^2).
\end{equation}
For $T\in q(V^{2,\circ})$, there is $\dd \sigma_T=2\dd^\sharp h$ with respect to the identification $\Sigma_T\cong \OO(V)$.
\end{lemma}
\begin{proof}
The lemma means that with regard to the fiberation $\OO(V)\rar V^{2,\circ}\rar q(V^{2,\circ})$, there is
\begin{equation}\label{md2p}
\dd X=2|\det T|^{-\frac{1}{2}} \dd^\sharp h \dd T.
\end{equation}
Because Equality (\ref{md2p}) is compatible with the action of $\GL_2(k)$ on $V^{2,\circ}$ and $q(V^{2,\circ})$, 
%in other words, if (\ref{md2p}) holds over $q^{-1}(\md)$ for an open subset $\md\subset q(V^{2,\circ})$, then it also holds over $g\circ q^{-1}(\md)$ for all $g\in \GL_2(k)$. 
it suffices that we prove (\ref{md2p}) over $\muu=\{(x_1,x_2)\in V^{2,\circ}: q(x_1)\neq 0\}$. Equality (\ref{md2p}) is also compatible with the scaling of the quadratic form $q$, whence one may suppose $(V,q)=(V_{k^\prime},q_{k^\prime})$.

Write $k^\prime=k\oplus ke$ with $\Tr(e)=0$ and $\Nm(e)=d\in \disc(k^\prime){k^\times}^2$. For $x=a+be\in k^\prime$, there is $\dd x=|2||d|^{1/2}\dd a\dd b$, where $\dd a, \dd b$ are the self-dual Haar measures on $k$. Consider the map
\[
p: \muu\rar V^\circ\times k\times k,\quad X=(x_1,x_2)\rar (x_1, t_2,t_3)=(x_1,q(x_2,x_2),q(x_1,x_2)) 
\]
Write $x_1=a_1+b_1e$, $x_2=a_2+b_2 e$. Set $y_1=x_1 e=-b_1d+a_1e$, then $x_1,y_1$ are orthongal and form a frame of $V$. It is easy to see that for $X\in p^{-1}(x_1,t_2,t_3)$, there is $x_2=a^\prime x_1+b^\prime y_1$ with
\[
a^\prime=\frac{t_3}{q(x_1,x_1)},\quad {b^\prime}^2=\frac{t_3^2-t_2q(x_1,x_1)}{dq(x_1,x_1)^2}.
\]
By changing the coordinates $(x_1,x_2)$ on $\muu$ to $(x_1,a^\prime,b^\prime)$, we have
\[
\dd X=\dd x_1\dd x_2=|2||d|^{1/2}\dd x_1 \dd a_2\dd b_2=|2||d|^{1/2}|\Nm(x_1)|\dd x_1\dd a^\prime \dd b^\prime.
\]
Set $r^\prime={b^\prime}^2$. By Section ~\ref{s_md1} (i), there is $\dd b^\prime=2 |2|^{-1}|r^\prime|^{-1/2}\dd h_0\dd r^\prime$, with $\dd h_0$ being the probability measure on $\OO(x_1^\perp)$. Hence
\[
\dd X=2|d|^{1/2}|\Nm(x_1)||r^\prime|^{-1/2}\dd x_1\dd a^\prime  \dd h_0 \dd r^\prime=2|t_3^2-t_2q(x_1,x_1)|^{-1/2}\dd x_1 \dd h_0 \dd t_2\dd t_3.
\]

Consider the fiberation $V^\circ\rar 2q(V^\circ)\subset k$, $x_1\rar t_1=q(x_1,x_1)$, for which the fibers are homeomorphic to $\SO(V)$. By Section ~\ref{s_md1} (ii), there is $\dd x_1=\dd^\sharp h_1 \dd t_1$. Therefore,
\[
\dd X=2|t_3^2-t_1t_2|^{-1/2}\dd^\sharp h_1\dd t_1 \dd h_0\dd t_2\dd t_3=2|\det T|^{-1/2}\dd^\sharp h \dd T.
\]
Here the measure $\dd^\sharp h_1$ on $\SO(V)$ is combined with $\dd h_0$ to get the measure $\dd^\sharp h$ on $\OO(V)$.
\end{proof}

Now let $k=F$ be a number field and write $K=F(\sqrt{-\det V})$. For $T\in q(V^{2,\circ})$, there is $\Sigma_{v,T}\cong \OO(V)_{F_v}$ at each local place; let $\dd h=\prod_v \dd h_v$ be the Tamagawa measure on $\OO(V)_\ba$ and write $\dd\sigma_{v,T}=2c_{K_v} \dd h_v$. The following lemma is a direct consequence of Lemma ~\ref{l_md2}.

\begin{lemma}\label{md2_r}
$c_{K_v}=L(1,\chi_{K_v/F_v})^{-1}$ for almost all $v$ and $\prod_v c_{K_v}L(1,\chi_{K_v/K_v})=c_K$
\end{lemma}

\subsubsection{The fiberation $\mv_e\rar D_0$}\label{fiberation2}

Consider the fiberation $\wtilde{q}_e: \mv_e\rar D_0$, $x\rar \bar{x}ex$ in Section ~\ref{gu1dgroup}. Recall the decomposition $D=k^\prime \oplus k^\prime e^\prime$ with $k^\prime=k\oplus ke$ and $ee^\prime=-e^\prime e$.

Suppose the $k$ is a local field. Let $\dd x$ be the self-dual Haar measure on the quadratic space $(\mv_e, \Nm)$ with respect to $\psi$ and $\dd y$ be the restricted measure on the subspace $D_0$. For $y\in \wtilde{q}_e(\mv_e^\circ)$, there is identification $\Sigma_y\cong \U(\mv_e)$ and we let $\dd h^\sharp$ be the local measure on $\U(\mv_e)$ as chosen in Section ~\ref{s_tmeasure} (iv).

\begin{lemma}\label{md2p_d}
There exist $\U(\mv_e)$-invariant measures $\dd \sigma_y$ on $\Sigma_y$ such that
\[
\int_{\mv_e}F(x)\dd x=\int_{D_0^\circ}\int_{\Sigma_y} f(x)\dd \sigma_y |\Nm(y)|^{-1/2}\dd y,\quad F\in L^1(\mv_e).
\]
For $y\in \wtilde{q}_e(\mv_e^\circ)$, there is $\dd \sigma_y=2|4\Nm(e)|^{-1}\dd h^\sharp$ with respect to the identification $\Sigma_y\cong \U(\mv_e)$.
\end{lemma}
\begin{proof}
Put $e^2=-\delta$ and $\ep^2=\delta^\prime$. Write elements in $\mv_e$ as $x=z_1+z_2e^\prime$ with $z_1, z_2\in k^\prime$, then $\dd x= |\delta^\prime|\dd z_1 \dd z_2$; write elements in $D_0$ as $y=te+zee^\prime$ with $t\in k$ and $z\in k^\prime$, then $\dd y=|\delta|^{3/2}|\delta^\prime| \dd t\dd z$. Accordingly, the map $\wtilde{q}_e(x)=\bar{x}ex$ is expressed as
\[
\wtilde{q}_e: \mv_e\rar D_0,\quad x=(z_1,z_2)\rar y=(t,z)=\big(\Nm(z_1)+\delta^\prime \Nm(z_2), 2\bar{z}_1z_2\big).
\]
It suffices that we verify the measure decomposition $\dd x=2\dd h^\sharp \dd y$ over the open subset $\muu=\{z_1+z_2e^\prime\in \mv_e^\circ: z_1,z_2\in{k^\prime}^\times\}$.

Over $\muu$, we change coordinates and write $x=z_1+z_1z_2e^\prime$; ffurthermoe, consider the fiberation ${k^\prime}^\times \rar \Nm({k^\prime}^\times)\subset k$, $z_1\rar r_1=\Nm(z_1)$, for which the fibers are homeomorphic to ${k_1^\prime}^\times$, then $\dd x=|\delta^\prime||\Nm(z_1)|\dd z_1 \dd z_2=|\delta^\prime||r_1|\dd^\sharp h_1\dd r_1 \dd z_2$, where $\dd^\sharp h_1$ is the natural quotient measure on ${k_1^\prime}^\times$. On the other hand, the smooth map $\wtilde{q}_e|_{U}$ is given by
\[
x=(h_1, r_1, z_2)\rar y=(t,z)=(r_1(1+\delta^\prime \Nm(z_2)),2r_1z_2).
\]
The differential form $|\Nm(y)|^{-1/2}\dd y$ is pulled back to $|4\delta\delta^\prime||r_1|\dd r_1\dd z_2$.

$\wtilde{q}_e|_{U}$ factors through $\wtilde{p}: \Nm({k^\prime}^\times)\times {k^\prime}^\times\rar D_0^\circ$, $(r_1,z_2)\rar y=(r_1(1+\delta^\prime \Nm(z_2)),2r_1z_2)$. When $\delta\not\in \Nm({k^\prime}^\times)$, $\wtilde{p}$ is injective and $D$ is nonsplit, whence $(h_1,y)$ can be used as coordinates of $x$ and $\dd x=|4\delta|^{-1}\dd^\sharp h_1 |\Nm(y)|^{-1/2}\dd y$. Therefore, $\dd\sigma_y=|4\delta|^{-1}\dd^\sharp h_1=2|4\delta|^{-1}\dd^\sharp h$. When $\delta\in \Nm({k^\prime}^\times)$, $D$ is split and each nonempty fiber of $\wtilde{p}$ is a two-point set: if $(r_1,z_2)$ is on a fiber, then $(r_1\delta^\prime\Nm(z_2),\frac{1}{\delta^\prime \bar{z}_2})$ is also on the fiber. We identify the nontempty fibers with $\mu_2$, then $|r_1|\dd r_1 \dd z_2=2|4\delta\delta^\prime|^{-1}\dd h_0\dd y$, where $\dd h_0$ is the probability measure on $\mu_2$. It follows that $\dd x=2|4\delta|^{-1}\dd^\sharp h_1\dd h_0 |\Nm(y)|^{-1/2}\dd y$, whence $\dd\sigma_y=2|4\delta|^{-1}\dd^\sharp h_1\dd h_0=2|4\delta|^{-1}\dd^\sharp h$.
\end{proof}

\begin{remark} 
Let $k=F$ be a number field and write $K=F\oplus Fe$. For $y\in \wtilde{q}_e(\mv_e^\circ)$, there is $\Sigma_{v,y}\cong \U(\mv_e)$ at every place; let $\dd h=\prod_v \dd h_v$ be the Tamagawa measure on $\U(\mv_e)_\ba$ as chosen in Section ~\ref{s_tmeasure} and write $\dd \sigma_{v,y}=2c_{K_v}\dd h_v$. Lemma ~\ref{md2p_d} implies that the constants $c_{K_v}$ satisfy the assertions in Lemma ~\ref{md2_r}.
\end{remark}

\subsection{Theta correspondence}

We review the notion of theta correspondence for the dual pair $(\Sp_{2n},\OO(V))$.

\subsubsection{Local theta correspondence}\label{ltc}

Suppose that $k$ is a local field of characteristic zero. We briefly recall the notation of abstract Howe duality (c.f. \cite{roberts96}). Let $G$ and $H$ be reductive groups over $k$ and $\rho$ be a smooth representation of $G\times H$. For any smooth irreducible representation $\tau$ of $G$ (resp. $H$), the maximal $\tau$-isotypic component of $\rho$ has the form $\pi\otimes \theta_0(\pi)$ for some smooth representation $\theta_0(\pi)$ of $H$ (resp. $G$). We say that the Howe duality holds for the triplet $(\rho, G,H)$ if the maximal semisimple quotient $\theta(\pi)$ of $\theta_0(\pi)$ is irreducible whenever $\theta_0(\pi)$ is nonzero. %(We may just say the Howe duality for the pair $(G,H)$ when $\rho$ is obvious.)

Classical Howe duality concerns the situation that $\rho$ is the Weil representation $\omega$. Let $\wtilde{\Sp}_{2n}(k)$ be the $2$-fold metaplectic cover of $\Sp_{2n}(k)$. Let $(V,q)$ be a quadratic space over $F$ of dimension $m$ and choose a non-trivial additive character $\psi$ of $k$. There is a Weil representation $\omega:=\omega_{\psi,V}$ of $\OO(V)\times \wtilde{\Sp_n}(F)$ on $\ms(V^n)$, the space of Bruhat-Schwartz functions on $V^n$:
\begin{align*}
&\omega(h)\phi(X)=\phi(h^{-1}X),\quad h\in \OO(V),\\
&\omega\big(\smalltwomatrix{1_n}{N}{}{1_n},\epsilon\big)\phi(X)=\epsilon^m \psi\big(2^{-1}\Tr(Nq(X))\big)\phi(X),\quad N\in \Sym_{n\times n}(F),\\
&\omega\big(\smalltwomatrix{A}{}{}{\leftup{t}{A^{-1}}},\epsilon\big)\phi(X)=\epsilon^m\chi_{\psi,V}(\det A)|\det A|^{\frac{m}{2}}\phi(XA),\quad A\in \GL_n(F),\\
&\omega(\smalltwomatrix{}{1_n}{-1_n}{},\epsilon)\phi(X)=\epsilon^m \gamma(\psi, V)\int_{V^n(F)}\phi(Y)\phi(q(X,Y))\dd Y.
\end{align*}
Here $X\in V^n$ is written as a row vector $(x_1,\cdots, x_n)$ with $x_i\in V$, $\gamma(\psi,V)$ is a constant of norm $1$ and $\chi_{\psi,V}:F^\times \rar S^1$ is a function satisfying $\chi_{\psi,V}(a_1a_2)=\chi_{\psi,V}(a_1)\chi_{\psi,V}(a_2)<a_1,a_2>^m$. We use the notions $\omega_\psi, \omega_V, \omega_{\psi,V}$ to denote the Weil representations, depending on the context.

The Howe duality for the triplet $(\omega,\wtilde{\Sp}_{2n},\OO(V))$ holds in general when $k$ is archimedian and when $k$ is nonarchimedean but not dyadic. For $n=1$ with $m\leq 5$ and for $n=2$ with $m\leq 2$, the Howe duality is also known for dyadic fields. 

\begin{remark}
By Conservation Principle, the Howe duality for the pair $(\wtilde{\SL}_2,\OO_m)$ is actually for the pair $(\wtilde{\Sp}_2,\SO_m)$ when $m\geq 3$. 
\end{remark}

When $m$ is even, the action of $\wtilde{\Sp}_{2n}$ via $\omega$ factors through $\Sp_{2n}$ and the Howe duality is actually between $(\Sp_{2n}, \OO(V))$. Furthermore, let $\nu$ denote the similitude characters of $\GSp_{2n}$, $\GO(V)$ and define $\GSp_{2n}^+=\{g\in \GSp_{2n}: \nu(g)\in \nu(\GO(V))\}$. The Weil representation can naturally extend to the group $R:=\{(g,h)\in \GSp_{2n}^+\times \GO(V): \nu(h)=\nu(g)\}$,
\begin{equation}\label{wr_s}
\omega(g,h)\phi(X)=|\nu(h)|^{-\frac{mn}{4}}\omega_{\psi,V}\big(g\smalltwomatrix{1}{}{}{\nu(g)^{-1}}\big)\phi(h^{-1}X),\quad (g,h)\in R.
\end{equation}
Put $\Omega^+=\Ind_R^{\GSp_{2n}^+\times \GO(V)}\omega$. By \cite{roberts96}, the Howe duality for $(\Omega^+, \GSp_{2n}^+,\GO(V))$ follows from the Howe duality for $(\omega, \GSp_{2n},\OO(V))$. Specifically, the Howe duality for $(\GSp_4^+,\GO_2)$ holds.

\subsubsection{Global theta lifting}

Suppose that $k=F$ is a number field. Let $\omega:=\omega_{\psi,V}=\otimes \omega_{\psi_v,V}$ be the Weil representation of $\wtilde{\Sp}_{2n}(\ba)\times \OO(V)_\ba$ on $\ms(V^n_\ba)$, the space of Bruhat-Schwartz functions on $V^n_\ba$. For $\phi\in \ms(V_\ba^n)$, the associated theta kernel function
\[
\theta_\phi(g,h)=\sum_{\xi\in V^n(F)}\omega_{\psi,V}(g,h)\phi(\xi).
\]
is a slowly increasing function on $[\wtilde{\Sp}_{2n}]\times [\OO(V)]$. When $m$ is even, $\omega$ can be extended to $R(\ba)=\GSp_{2n}^+(\ba)\times \OO(V)_\ba$ and the function $\theta_\phi$ is defined on $[R]$.

For rapidly decreasing automorphic forms $\varphi(g)$ on $[\wtilde{\Sp}_{2n}]$ and $F(h)$ on $[\OO(V)]$, we generally define their global theta lifts via $\phi$ by
\[
\Theta_\phi(\varphi)(h)=\int_{[\Sp_{2n}]}\overline{\varphi(g)}\Theta_\phi(g,h)\dd g,\quad \Theta_\phi(F)(h)=\int_{[\OO(V)]}\overline{F(h)}\Theta_\phi(g,h)\dd h.
\]
Specifically, for the pair $(\wtilde{\SL}_2,\OO(V))$ with $m=3$ or $5$, we lift a rapidly decreasing form $F(h)$ on $[\SO(V)]$ to $\wtilde{\SL}_2(\ba)$ by setting $\Theta_\phi(F)(h)=\int_{[\SO(V)]}\overline{F(h)}\Theta_\phi(g,h)\dd h$.

\section{Regularized Local Bessel Period Integral}

Let $k$ be a local field of characteristic zero and $(V,q)$ be a quadratic space over $k$ of dimension $m\geq 3$ and positive Witt index. Recall from Section ~\ref{bgroup} the decomposition of $V$
\[
V=kv_+\oplus U\oplus kv_-=kv_+\oplus (kv_0\oplus V_0)\oplus kv_-,
\]
the Bessel subgroup $\mrr=U\rtimes \SO(V_0)$, and the parabolic subgroup $P=U\rtimes M$ stabilizing $v_0$.

Let $\Pi$ and $\Pi^\prime$ be irreducible unitary representation of $\SO(V)$ and $\SO(V_0)$ respectively. Let $\psi$ a be nontrivial character of $k$ and $\psi_{v_0}(u)=\psi(q(v_0,u))$ be the associated character on $U$. For $f_1,f_2\in \Pi$ and $f_1^\prime,f_2^\prime\in \Pi^\prime$, the local Bessel period functional on $\mrr$ is formally defined as
\[
\mpp(f_1,f_2,f_1^\prime,f_2^\prime)=\int_{U\rtimes \SO(V_0)} (\Pi(uh)f_1,f_2)(\Pi^\prime(h)f_1^\prime,f_2^\prime)\psi_{v_0}(u)\dd u \dd h.
\]
We describe in this section a method to regularize the above integral and condense the expected properties into Conjecture ~\ref{conj_lbp}.

\subsection{Tempered distributions}\label{subsec_td}

\subsubsection{Definition} The space of Bruhat-Schwartz function $\ms(k^m)$ consists of rapidly decreasing functions when $k$ is archimedean and of locally constant compactly supported functions when $k$ is nonarchimedean. It carries a standard complete topology when $k$ is archimedean.

Let $\ms^\prime(k^m)$ denote the space of continuous linear maps $T:\ms(k^m)\rar \bc$ when $k$ is archimedean and the space of linear maps when $k$ is nonarchimedean. It is called the space of tempered distributions on $k^m$. (When $k$ is archimedean, this is the classical notion.)

A functional $T\in \ms^\prime(k^m)$ is said to be represented over an open subset $U\subset k^m$ by a function $f\in L^1_{loc}(U)$ if $T(\phi)=\int_{U}f(X)\phi(X)dX$ for all $\phi\in C^\infty_c(U)$. When $f$ is required to be continuous, it is unique if it exists.

\subsubsection{Application to quadratic spaces}

The Fourier transform on $L^1(V)$ is extended to $\ms^\prime(V)$ by setting $\mf T:=T\circ \mf$. %We need a criterion to tell when a functional $T\in \ms^\prime(V)$ is represented over $V^\circ$ by a smooth function. Consider the natural action of $k^\times \times \SO(V)$ on $V$. 
For $\Psi\in C^\infty_c(k^\times \times \SO(V))$ and $\phi\in \ms(V)$, $T\in \ms^\prime(V)$, set
\begin{align*}
&\Psi\diamond \phi(X)=\int_{k^\times \times \SO(V)}\Psi(a,h)\phi(a^{-1}h^{-1}X)\dd^\times a\dd h,\\
&\Psi\diamond T(\phi)=T(\Psi^\vee\diamond \phi),\quad \text{where}\quad \Psi^\vee(a,h)=\Psi(a^{-1},h^{-1}).
\end{align*}
\begin{proposition}\label{reg_V}
$\Psi\diamond T$ is represented over $V^\circ$ by a smooth function.
\end{proposition}
\begin{proof}
Write $\wtilde{H}=k^\times \times \SO(V)$; for $\tilde{h}=(a,h)$, set $|\tilde{h}|=|a|^m$. It suffices to find an open cover of $V^\circ$ so that $\Psi\diamond T$ is represented over each member of the cover by a smooth function.

(1) $k$ is nonarchimedean. $\Psi$ is left invariant under certain compact open subgroup $W$ of $\wtilde{H}$. Write $\mathrm{Supp}\Psi=\sqcup_{i=1}^\ell W\tilde{h}_i$ and $V^\circ=\sqcup_j Wv_j$ as the disjoint union of $W$-orbits. Observe that each $Wv_j$ is open in $V^\circ$ and they form an open cover of $V^\circ$. By definition,
\[
\Psi^\vee \diamond \phi(v)
=\int_{k^\times \times \SO(V)} \Psi(\tilde{h})\phi(\tilde{h}v)d\tilde{h}
=\sum_{i=1}^\ell \Psi(\tilde{h}_i)\int_W \phi(w\tilde{h}_iv)dw
\]
For $\phi\in C^\infty_c(Wv_j)$, there is $\int_W \phi(w\tilde{h}_iv)dw=\big[\int_W \phi(wv_j)dw\big]\cdot 1_{\tilde{h}_i^{-1}Wv_j}(v)$, whence
\[
\Psi\diamond T(\phi)=T(\Psi^\vee \diamond \phi)=\big[\sum_{i=1}^\ell \Psi(\tilde{h}_i)T(1_{\tilde{h_i}^{-1}Wv_j})\big]\int_W \phi(wv_j)dv
\]
Therefore $\Psi\diamond T$ is represented over $Wv_j$ by the constant $\frac{\mathrm{Vol}(W)}{\mathrm{Vol}(Wv_j)}\sum_{i=1}^\ell \Psi(\tilde{h}_i)T(1_{\tilde{h_i}^{-1}Wv_j})$.

(2) $k$ is archimedean. Given $v_0\in V^\circ$, we will find an open neighborhood $\muu_0$ of $v_0$ in $V^\circ$ such that $\Psi\diamond T$ is represented over $U_0$ by a smooth function.

Let $V_0^\prime$ be the orthogonal complement of $v_0$ in $V$. Observe that the orbit map $\wtilde{H}\rar V^\circ$, $\wtilde{h}\rar \tilde{h}v_0$ has a Jacobian of rank $m$ at the identity element of $\wtilde{H}$. By implicit function theorem, there is an open neighborhood $\muu$ of $v_0$ in $V^\circ$, an open neighborhood $\muu^\prime$ of $1_{V_0^\prime}$ in $\SO(V_0^\prime)$, and a smooth map $\varphi:\muu\times W^\prime \rar \tilde{H}$ such that (i) $W:=\varphi(\muu\times W^\prime)$ is open in $\wtilde{H}$ and $\varphi:\muu\times W^\prime \rar W$ is a diffeomorphism, (ii) $\varphi(v,h^\prime)v_0=v$ for $u\in \muu, h^\prime\in W^\prime$. Choose an open neighborhood $\muu_0$ of $v_0$ in $\muu$ and an open neighborhood $W_0$ of $1_{\tilde{H}}$ in $W$ such that the closure of $W_0^{-1}\muu_0$ is contained in $\muu$.

We first show that there exists a function $\eta\in C^\infty_c(\tilde{H})$ such that $\Psi\diamond T(\phi)=\int_{\tilde{H}}\eta(\tilde{h})\phi(\tilde{h}v_0)d\tilde{h}$ for $\phi\in C^\infty_c(\muu_0)$. Using partition of unity, one can write $\Psi=\sum_{i=1}^\ell \Psi_i$, with each $\Psi_i$ being smooth and supported in $W_0\tilde{h}_i$ for some $\tilde{h}_i$. Suppose $\phi\in C^\infty_c(\muu_0)$, then $\Psi_i^\vee \diamond \phi(v)=\int_{\tilde{H}} \Psi_i(w\tilde{h}_i)\phi(w\tilde{h}_iv)dw$ is supported in $\tilde{h}_i^{-1}W_0^{-1}\muu_0\subset \tilde{h}_i^{-1}\muu_1$. Let $\lambda(v)$ be a smooth function that takes value $1$ on $W_0^{-1}\muu_0$ and vanishes outside $U$; set $\varphi_1(v)=\varphi(1_{V^\prime_0}, v)$ for $v\in U$, then
\begin{align*}
\Psi_i^\vee \diamond \phi(v)=& \lambda(\tilde{h}_iv)\int_{\tilde{H}}\Psi_i(w\tilde{h}_i)\phi(w\tilde{h}_iv)dw\\
=&\lambda (\tilde{h}_iv)\int_{\tilde{H}}\Psi_i(w\tilde{h}_i)\phi\big(w \varphi_1(\tilde{h}_iv)v_0\big)dw\\
=&\int_{\tilde{H}}\lambda (\tilde{h}_iv)\Psi_i(w[\varphi_1(\tilde{h}_iv)]^{-1}\tilde{h}_i)\phi\big(w v_0\big)dw.
\end{align*}
Set $\Phi_i(w,v)=\lambda (\tilde{h}_iv)\Psi_i(w[\varphi_1(\tilde{h}_iv)]^{-1}\tilde{h}_i) \in C^\infty(\tilde{H}\times V^\circ)$, then $\Psi_i^\vee \diamond \phi=\int_{\tilde{H}}\Phi_i(\tilde{h},v)\phi(\tilde{h}v_0)d\tilde{h_0}$. By the continuity of $T$ on $\ms(V)$ or by the Schwartz representation theorem for tempered distributions, one gets
\[
\Psi_i\diamond T(\phi)=T(\Psi_i^\vee\diamond \phi)=T\big(\int_{\tilde{H}}\Phi_i(\tilde{h},v)\phi(\tilde{h}v_0)d\tilde{h}\big)=\int_{\tilde{H}}\eta_i(w)\phi(wv_0)dw,
\]
where $\eta_i(\tilde{h})=T(\Phi_i(\tilde{h},\cdot))\in C^\infty_c(\tilde{H})$. Hence the function $\eta=\sum_{i=1}^\ell \eta_i$ meets the requirement.

We next show the existence of a function $\phi_0\in C^\infty(\muu_0)$ such that $\Psi\diamond T(\phi)=\int_{V^\circ}\phi_0(v)\phi(v)$ for $\phi\in C^\infty_c(\muu_0)$. Consider an open subset $W_1=\varphi(\muu_1\times W_1^\prime)$ of $W$, with $\muu_1, W_1^\prime$ being open in $V^\circ, \SO(V^\prime_0)$ respectively and $\overline{\muu_1}\subset U$, $\overline{W_1^\prime}\subset W^\prime$. Let $\lambda_1\in C^\infty_c(U), \lambda_2\in C^\infty_c(W^\prime)$ be functions such that $\lambda_1|_{\muu_1}=1$ and $\lambda_2|_{W_1^\prime}=1$. With respect to the diffeomorphism $\varphi:U\times W^\prime \rar W$, there is a smooth non-zero function $J(v,h^\prime)$ such that the Haar measure $d\tilde{h}$ is pulled back to be $J(v,h^\prime)\dd v\dd h^\prime$.

Using partition of unity, one can write $\eta=\sum_{j=1}^{\ell^\prime} \wtilde{\eta}_j$ as a finite sum, where each $\wtilde{\eta}_j\in C^\infty_c(\tilde{H})$ is supported in $t_jW_1$ for certain $t_j$. So $\Psi\diamond T(\phi)=\sum_j \int_{\tilde{H}} \wtilde{\eta}_j(\tilde{h})\phi(\tilde{h}v_0)d\tilde{h}$. Observes that
\begin{align*}
\int_{\tilde{H}}\wtilde{\eta}_j(\tilde{h})\phi(\tilde{h}v_0)d\tilde{h}=&\int_{W_1}\wtilde{\eta}_j(t_j\tilde{h})\phi(t_j\tilde{h}v_0)d\tilde{h}\\
=&\int_{V^\circ \times W^\prime}\lambda_1(v)\lambda_2(h^\prime)\wtilde{\eta}_j(t_j\varphi(v,h^\prime))\phi(t_j v)dvdh^\prime,
\end{align*}
where one makes a change of variable $\wtilde{h}=\varphi(v,h^\prime)$ at the second equality and uses the fact $\varphi(v,u^\prime)v_0=v$. Note that $\tilde{\phi}_j(v,h^\prime):=\lambda_1(v)\lambda_2(h^\prime)\wtilde{\eta}_j(t_j\varphi(v,h^\prime))\in C^\infty_c(V^\circ\times W^\prime)$. Hence
\[
\int_{\tilde{H}}\wtilde{\eta}_j(\tilde{h})\phi(\tilde{h}v_0)d\tilde{h}=\int_{V^\circ}\phi_j(v)\phi(v)dv,
\]
with $\phi_j(v)=\int_{W^\prime}\tilde{\phi}_j(t_j^{-1}v,h^\prime)|t_j|^{-1}dh^\prime \in C^\infty_c(V^\circ)$. $\phi_0=\sum_{j=1}^{\ell^\prime} \phi_j$ meets the requirement.
\end{proof}

\subsection{Local height function on $\OO(V)$}\label{subsec_lhf}

Let $(\wtilde{V},\wtilde{q})=(V,q)\oplus (V,-q)$ be the doubled quadratic space and consider the natural embedding $i: \OO(V)\times \OO(V)\hrar \OO(\wtilde{V})$. Let $\wtilde{P}$ be the parabolic subgroup of $\OO(\wtilde{V})$ stabilizing the isotropic subspace $V^d:=\{(v,v):v\in V\}$. Let $K$ and $\wtilde{K}$ be maximal compact subgroups of $\OO(V)$ and $\OO(\wtilde{V})$ respectively such that $i(K\times K)\subset \wtilde{K}$. Let $\Ind_{\wtilde{P}}^{\OO(\wtilde{V})}|\cdot|^s$ denote the representation of $\OO(\wtilde{V})$ (unitarily) induced from the quasi-character $\tilde{p}\rar |\det \tilde{p}|_{V^d}|^s$. Write the modulus character of $\wtilde{P}$ as $\delta_{\wtilde{P}}(\tilde{p})=|\det \tilde{p}|_{V^d}|^{2\rho_{\wtilde{P}}}$.

\begin{definition}\cite{qiu14}
Choose a spherical section $f_{s}\in \Ind_{\wtilde{P}}^{\OO(\wtilde{V})}|\cdot|^s$ with value $1$ on $\wtilde{K}$. Define
\[
\Delta_{\OO(V)}(h)=f_{-\rho_{\wtilde{P}}+1}(h,1).
\]
\end{definition}
The local  height function $\Delta_{\OO(V)}(h)$ is biinvariant by $K$ and tends to zero when $h\rar \infty$. When the group is apparent, we omit the subscript $\OO(V)$.

Now suppose $\dim V_0=2$ and we explicitly describe the height functionon on $\SO(V_0)$. When $V_0$ is anisotropic, $\Delta(h)$ takes constant value $1$ on $\SO(V_0)$. When $V_0$ is isotropic, there is $\SO(V_0)\cong k^\times$ and, for $a\in k^\times$,
\begin{equation*}
\Delta(a)=\begin{cases}
\min\{|a|,|a|^{-1}\}, &\text{$k$ is nonarchimedean},\\
\big[(1+|a|^2)(1+|a|^{-2})\big]^{-1/2}, &\text{$k$ is archimedean}.
\end{cases}
\end{equation*}

\subsection{Local Bessel period integral}\label{subsec_lbp}

We consider an iterated integral with parameter $s$:
\begin{equation}\label{lbpi_s}
\mpp(s;f_1,f_2,f_1^\prime,f_2^\prime)=\int_{\SO(V_0)}\big[\int_U (\Pi(uh)f_1,f_2)\psi_{v_0}(u)du\big] (\Pi^\prime(h)f_1^\prime,f_2^\prime)\Delta(h)^sdh,
\end{equation}
where $f_1, f_2\in \Pi$, $f_1^\prime, f_2^\prime\in \Pi^\prime$ and $\Delta(h)$ is the local height function on $\SO(V_0)$. The inner $U$-integral is interpreted according to Definition ~\ref{def_ui} below.

\subsubsection{The $U$-integration}\label{ss_uintegration}

Because $\Pi$ is unitary, the matrix coefficient $(\Pi(h)f_1,f_2)$ is bounded and smooth on $\SO(V)$, whence the associated functional $T_{(\Pi(u)f_1,f_2)}$ on $\ms(U)$ belongs to $\ms^\prime(V)$,
\[
T_{(\Pi(u)f_1,f_2)}(\phi):=\int_U (\Pi(u)f_1,f_2)\phi(u\dd u.
\]
The integral $\int_U (\Pi(u)f_1,f_2)\psi_{v_0}(u)du$ is the formal expression for the value of $\mf (\Pi(u)f_1,f_2)$ at $v_0$. We shall interpret $\mf (\Pi(u)f_1,f_2)$ from the view of $\mf T _{(\Pi(u)f_1,f_2)}$.

\begin{proposition}
$\mf T_{(\Pi(u)f_1,f_2)}$ is represented over $U^\circ$ by a unique smooth function $W_{f_1,f_2,\psi}$.
\end{proposition}
\begin{proof}
(ii) $k$ is nonarchimedean. $f_1, f_2$ are invariant under certain compact open subgroup $\wtilde{W}$ of $\SO(V)$. Note that $M^+=k^\times \times \SO(U)$ and set $W=\wtilde{W}\cap (k^\times\times \SO(U))$. Then $(\Pi(u)f_1,f_2)=\frac{1}{\mathrm{Vol}(W)}1_W \diamond (\Pi(u)f_{1i},f_{2i})$ and $\mf T_{(\Pi(u)f_1,f_2)}=\frac{1}{\mathrm{Vol}(W)}1_W \diamond \mf T_{(\Pi(u)f_{1i},f_{2i})}$. The assertion then follows from Proposition ~\ref{reg_V}.

(ii) $k$ is archimedean. Applying Dixmier-Malliavin theorem to the smooth action of $k^\times \times \SO(U)$ to $\Pi\otimes \Pi$, one sees that there exist smooth vector $f_{1i}, f_{2i}\in \Pi$ and functions $\Psi_i\in C^\infty_c(k^\times \times \SO(U))$ such that  
\[
f_1\otimes f_2=\sum_{i=1}^\ell \int_{k^\times \times \SO(U)} \Psi_i(a,h) \cdot \big[\Pi(a,h)f_{1i}\otimes \Pi(a,h)f_{2i}\big] \dd^\times a \dd h.
\]
Hence $(\Pi(u)f_1,f_2)=\sum_{i=1}^\ell \Psi_i\diamond (\Pi(u)f_{1i},f_{2i})$ and therefore
\[
\mf T_{(\Pi(u)f_1,f_2)}=\sum_{i=1}^\ell \Psi_i^\prime\diamond \mf T_{(\Pi(u)f_{1i},f_{2i})},
\]
with $\Psi^\prime_i(a,h)=\Psi_i(a^{-1},h)$. The assertion then follows from Proposition ~\ref{reg_V}.
\end{proof}

\begin{definition}\label{def_ui}
For $v_0\in U^\circ$, define $\int_{U}(\Pi(u)f_1,f_2)\psi_{v_0}(u)du\overset{\triangle}{=}W_{f_1,f_2,\psi}(v_0)$.
\end{definition}

\subsubsection{The local conjecture}\label{ss_lc}

Set $L(s,\Pi,\Pi^\prime):=\frac{L(s+\frac{1}{2},\Pi\times \Pi^\prime)}{L(s+1,\Pi,\ad)L(s+1,\Pi^\prime,\ad)}$ and define
\begin{equation*}
\Delta_V=
\begin{cases}
\prod_{i=1}^{(m-1)/2}\zeta_k(2i), &\text{$m$ is odd},\\
L(\frac{m}{2},\chi_{\disc(V)})\prod_{i=1}^{m/2-1}\zeta_k(2i), &\text{$m$ is even.}
\end{cases}
\end{equation*}
Recall that $\disc(V)=(-1)^{\frac{m}{2}}\det V$ when $m$ is even.

\begin{conjecture}\label{conj_lbp}
%Suppose that $\Pi, \Pi^\prime$ are irreducible unitary representations of $\SO(V), \SO(V_0)$.
\begin{itemize}
\item[(i)]  $\mpp(s;f_1,f_2,f_1^\prime,f_2^\prime)$ absolutely converges when $\re(s)>>0$ and has meromorphic continuation to $\bc$.

\item[(ii)] When $\mathrm{ord}_{s=0}L(s,\Pi,\Pi^\prime)\leq 0$, there is $\ord_{s=0}\mpp(s;f_1,f_2,f_1^\prime,f_2^\prime)\geq \ord_{s=0}L(s,\Pi,\Pi^\prime)$. The (possibly zero) functional $\mpp^\sharp(-):=\frac{\mpp(s;-)}{\Delta_V L(s,\Pi,\Pi^\prime)}\big|_{s=0}$ respects the action of $\mrr\times \mrr$: for $u_1, u_2\in U$ and $h_1, h_2 \in \SO(V_0)$, there is
\[
\quad\quad\quad P^\sharp(u_1h_1\circ f_1,u_2h_2\circ f_2, h_1\circ f_1^\prime, u_2h_2\circ f_2^\prime)=P^\sharp(f_1,f_2,f_1^\prime,f_2^\prime)\psi_{v_0}(u_2-u_1).
\]

\item[(iii)] $\mathrm{Hom}_\mrr \big(\Pi\otimes (\psi_{v_0}\boxtimes \Pi^\prime)\big)$ is nonzero if and only if $\mathrm{ord}_{s=0}L(s,\Pi,\Pi^\prime)\leq 0$ and $P^\sharp\neq 0$.
\end{itemize}
\end{conjecture}

Recall from \cite[Coro. 15.3]{gan_gross_prasad2012} and \cite{jsz10} that $\mathrm{Hom}_\mrr \big(\Pi\otimes (\psi_{v_0}\boxtimes \Pi^\prime)\big)$ is at most $1$-dimensional. Based on Conjecture ~\ref{conj_lbp}, when $\mathrm{ord}_{s=0}L(s,\Pi,\Pi^\prime)\leq 0$, we call $P^\sharp$ the regularized local Bessel period functional on $\Pi\otimes \Pi\otimes \Pi^\prime\otimes \Pi^\prime$. Furthermore, when $\mathrm{Hom}_\mrr \big(\Pi\otimes (\psi_{v_0}\boxtimes \Pi^\prime)\big)$ is nonzero and $\Pi, \Pi^\prime$ are unitary spherical with $f\in \Pi, f^\prime\in \Pi^\prime$ being spherical of norm $1$, we expect  $P^\sharp(f,f,f^\prime,f^\prime)=2^{-\beta}$, where $\beta$ is an integer depending on the Arthur packets of $\Pi,\Pi^\prime$.

\subsection{An $s$-parameterized $\SO_2$-integral}\label{s_so2i}

When $\dim V=5$, the outer integral in $\mpp(s;-)$ is a special $\SO_2$ integral and we discuss it her for later use. Let $k^\prime$ be a quadratic algebra over $k$ and $\chi$ be a character of ${k^\prime_1}^\times={k^\prime}^\times/k^\times$. Define
\begin{equation}\label{so2i}
\mpp(s,\chi):=\int_{{k^\prime_1}^\times}\chi(h)\Delta(h)^s\dd h,\quad \mpp^\sharp(\chi):=\frac{\zeta_k(2s)}{L(s,\chi)}\mpp(s,\chi)\big|_{s=0}.
\end{equation}
%The related $L$-function is $\frac{L(s,\chi)}{\zeta_k(2s)}$, where $\chi$ is regarded a character of ${k^\prime}^\times$ via $\pr:{k^\prime}^\times \rar {k^\prime_1}^\times$.

\begin{lemma}\label{integralsov0}
$\mpp(s,\chi)$ is absolutely convergent when $\re(s)>>0$ and has meromorphic continuation to $\bc$ with $\ord_{s=0}\mpp(s,\chi)\geq \ord_{s=0}\frac{L(s,\chi)}{\zeta_k(2s)}$. The space $\Hom_{\SO(V_0)}(\chi,1)$ is nonzero if and only if $\chi=1$, which happens if and only if $\ord_{s=0}\frac{L(s,\chi)}{\zeta_k(2s)}\leq 0$.
\end{lemma}
Lemma ~\ref{integralsov0} is easy to verify and we omit the proof. When $\chi=1$, we define a local constant $c^\prime_{k^\prime}=\mpp^\sharp(1_{{k^\prime_1}^\times})$, whose value depends on the local measure.

\begin{lemma}\label{cep}
Let $E$ be a quadratic algebra over a number field $F$ and $\dd h=\prod_v \dd h_v$ be the Tamagawa measure on $\ba_{E,1}^\times$, then $c_{E_v}^\prime=1$ for almost all $v$ with $\prod_v c^\prime_{E_v}=c_E^{-1}$.
\end{lemma}
\begin{proof}
For a local field $k$ of characteristic zero, let $\dd a$ be the Lebesgue measure when $k=\br$ and twice the Lebesgue measure when $k=\bc$; let $\dd a$ be the Haar measure with $\mathrm{Vol}(\mo_{k})=1$ when $k$ is nonarchimedean. Let $d f_v$ and $d e_v$ be the local measures on $F_v$ and $E_v$ thus chosen, then $\prod_v d f_v$ and $\prod_v d e_v$ are the Tamagawa measures on $\ba$ and  $\ba_E$. The Tamagawa measure on $\ba_{E,1}^\times$ is $c^{-1}_E\prod_v \dd^\ast h_v$, where $\dd^\ast h_v:=L(1,\chi_{E_v/F_v}){\frac{d  e_v}{|e_v|}}/\frac{d f_v}{|f_v|}$ on $E_{v,1}^\times\cong E_v^\times/F_v^\times$. Hence
\[
\prod_v c_{E_v}=c_E^{-1}\prod_v \big[\frac{\zeta_{F_v}(2s)}{\zeta_{E_v}(s)}\int_{E_{v,1}^\times}\Delta(h_v)^s\dd^\ast h_v\big]_{s=0}.
\]
One computes that $\big[\frac{\zeta_{F_v}(2s)}{\zeta_{E_v}(s)}\int_{E_{v,1}^\times}\Delta(h_v)^s\dd^\ast h_v\big]_{s=0}=1$ for all $v$, whence the lemma follows. %One directly computes that $\int_{E_{v,1}^\times}\Delta_v(h_{v})^sd^\ast h_v$ is equal to (i) $\frac{\zeta_{F_v}(s)^2}{\zeta_{F_v}(2s)}$ when $v$ is finite and split, (ii) $4s^{-1}2^{-s}$ when $v$ is archimedean and split, (iii) $1$ when $v$ is finite and innert, (iv) $2$, when $v$ is finite ramified or $E_v=\bc, F_v=\br$. It follows that  $\frac{\zeta_{F_v}(2s)}{\zeta_{E_v}(s)}\int_{E_{v,1}^\times}\Delta_v(h_{v})^sd^\ast h_v$ takes value $1$ at $s=0$ for all $v$. Therefore, $\prod_v c_{E_v}=c_E^{-1}$.
\end{proof}

%When $\Pi$ is the local component of a global representation $\Pi$ in $\SK(\pi,\chi)$, $\So(\eta)$ or $\HPS(\eta)$, integrals of type (\ref{so2i}) occur in the computation of the regularized local Bessel period.

\section{Quadratic Fourier Transform}\label{qft}

We discuss the isometry property of quadratic Fourier transform, which is one key tool for computing the local Bessel period integrals of Saito-Kurokawa representations.

Suppose that $k$ is a local field of characteristic zero and let $\dd r$ be the self-dual Haar measure on $k$. Define the standard Fourier transform $\mf$ and quadratic Fourier transform $\mf_2$ by
\[
\mf \varphi(t)=\int_k \varphi(t)\psi(rt)\dd r,\quad \mf_2 \varphi(t)=\int_k \varphi(r)\psi(r^2t)\dd r,\quad \varphi\in L^1(k).
\]

For a quadratic space $(V,q)$ over $k$, let $\dd X$ be the self-dual Haar measure on $V$ and define the Fourier transform $\mf$ and quadratic Fourier transform $\mf_q$ on $L^1(V)$ as
\[
\mf f(X)=\int_V f(Y)\psi(q(X,Y))dY,\quad \mf_{q}f(t)=\int_V f(X)\psi(t\cdot q(X))dX,\quad f\in L^1(V).
\]
By the measure decomposition in (\ref{md1}), there is $\mf_q f(t)=\mf \wtilde{f}(t)$.

We will derive the isometry property of $\mf_q$ from the isometry property of $\mf$ by establishing certain asymptotic estimates for $\wtilde{f}$. The estimates are in two steps: first for bounded compactly supported functions and then for rapidly decreasing functions that vanish near $0$.

\begin{lemma}\label{estimate1}
If $f\in L^1(V)$ is bounded and vanishes outside a compact set $D$, then there exists $\delta, c$ such that \emph{(i)} $\wtilde{f}(r)$ vanishes when $|r|>\delta$, \emph{(ii)} $|\wtilde{f}(r)|$ is bounded by $c|r|^{\frac{m-2}{2}}$ when $V$ is anisotropic and bounded by $c(|\ln |r||+1)$ when $V$ is isotropic.
\end{lemma}
\begin{proof}
For (i), the number $\delta=\max_{X\in D}|q(X)|$ meets the requirement. For (ii), we Put $A=\sup_{X\in D}|f(X)|$ and distinguish two cases. 

(iia) $V$ is anisotropic. Each $S_r$ ($r\in k^\times$) in $V$ is compact with a finite measure dependent on the coset $r{k^\times}^2$. Set $a=\max\{m(S_r):r\in {k^\times}^2/{k^\times}\}$ and $c=Aa$, then $|\wtilde{f}(r)|\leq c|r|^{\frac{m-2}{2}}$.

(iib) $V$ is isotropic. We prove the estimate for $|\wtilde{f}|$ by induction on $\dim V$.

When $\dim V=2$, choose an isomorphism $V\cong k^2$ so that $q(X)=x_1x_2$, then $\wtilde{f}(r)=\int_{k^\times}f(x_1,rx_1^{-1})\dd^\times x_1$. Choose $\epsilon$ such that $D$ is contained in the set $\{|x_1|,|x_2|\leq \epsilon\}$, then
\[
|\wtilde{f}|\leq A\int_{|r|^{-1}\epsilon\leq |x_1|\leq \epsilon} \dd^\times x_1\leq Aa_1(|\ln |r||+1)
\]
for some constant $a_1$ depending on $\epsilon$. Thus one can choose $c=Aa_1$.

If $\dim V>2$, write $V=V_0\oplus V_1$, where $V_1$ is isotropic and $V_0$ is a non-isotropic line orthogonal to $V_1$. Choose non-negative compactly supported continuous functions $f_0, f_1$ on $V_0, V_1$ so that $|f|\leq f_0\otimes f_1$. Then
\[
|\wtilde{f}(r)|\leq \int_k \wtilde{f}_0(r-r_1)\wtilde{f}_1(r_1)dr_1
\]
As argued, there exist $\delta, \delta_1$ such that $\wtilde{f}(r)$ vanishes when $|r|>\delta$ and $\wtilde{f}_1(r_1)$ vanishes when $|r_1|>\delta_1$; because $V_0$ is anisotropic, there exists $c_0$ so that $|\wtilde{f}_0(r_0)|\leq c_0|r_0|^{-\frac{1}{2}}$; $V^\prime$ is isotropic, by induction, there exists $c_1$ such that  $|\wtilde{f}_1(r_1)|\leq c_1(|\ln |r_1||+1)$. Thus
\[
|\wtilde{f}(r)|\leq c_0c_1\int_{|r_1|\leq \delta_1} |r-r_1|^{-1/2}(|\ln |r_1||+1)dr_1.
\]
When $|r|<\delta$, thre exists a constant $c^\prime$ depending on $\delta_1, \delta$ such that
\[
\int_{|r_1|\leq \delta_1} |r-r_1|^{-1/2}(|\ln |r_1||+1)dr_1\leq c^\prime (|\ln |r||+1)
\]
Set $c=c_0c_1c^\prime$, then $|\wtilde{f}(r)|\leq c(|\ln |r||+1)$.
\end{proof}

\begin{lemma}
Suppose that $k$ is archimedean and write $V=k^m$, $q(X)=\sum_{i=1}^m c_i x_i^2$ with $|c_i|=1$. For $\epsilon\in \br_+$, $n\in \bn$, let $f_{\epsilon,n}(X)$ be the function taking vlaue $|x_1\cdots x_m|^{-n}$ when $|x_1|,\cdots, |x_m|\geq \epsilon$ and value $0$ elsewhere. There exist constants $B_, b, \delta\in \br_+$ depending on $\epsilon, n, c_i (1\leq i\leq m)$ such that $\wtilde{f_{\epsilon,n}}(r)$ is bounded by $B$ and, furthermore, $\wtilde{f_{\epsilon,n}}(r)\leq b |r|^{-\frac{n+1}{2}}$ when $|r|\geq \delta$.
\end{lemma}
\begin{proof}
We prove it by induction on $m$. When $m=1$, the statement obviously holds. Now suppose $m>1$. Let $e=\mathrm{deg}(k/\br)$ and $|x|^\prime=|x|^{2/e}$ be the adjusted norm on $k$. Write $V=V_i\oplus V_i^\prime=k\oplus k^{m-1}$, where $V_i$ is the $i$-th copy of $k$ and $V^\prime$ the orthogonal complement of $V_i$. One observes that if $q(X)=r$, then there must be some $i$ so that $|x_i^2|^\prime\leq \frac{1}{m}|r|^\prime$; for this $i$, there is $|q(x_i)|\leq m^{-e}|r|$ and $|q(X)-q(x_i)|\geq (1-\frac{1}{m})^e|r|$. This observation means that $S_{V,r}=\cup_i S_{V,r}\cap U_{i,r}$, where $S_{V,r}$ is the $r$-sphere in $V$ and $U_{i,r}=\{X\in V: |q(x_i)|\leq m^{-e}|r|\}$. Hence
\[
\wtilde{f_{\epsilon,n}}(r)=|r|^{\frac{m-2}{2}}\int_{S_{V,r}}f_{\epsilon,n}(X)d\sigma_{V,r}\leq \sum_i |r|^{\frac{m-2}{2}}\int_{S_{V,r}\cap U_{i,r}}f_{\epsilon,n}(X)d\sigma_{V,r}.
\]

With respect to the projection $S_{V,r}\rar x_i$, there is measure decomposition $|r|^{\frac{m-2}{2}}d\sigma_{V,r}=|r|^{\frac{m-3}{2}}d\sigma_{V_i^\prime, r-q(x_i)}dx_i$. Also, $f_{\epsilon,n}(X)=f_{\epsilon,n, V_i}(x_i)f_{\epsilon,n,V_i^\prime}(X_i^\prime)$ for $X=(x_i,X_i^\prime)$, where we add subscripts of underlying spaces to distinguish the functions. Hence
\begin{align*}
g_i(r):=|r|^{\frac{m-2}{2}}\int_{S_{V,r}\cap U_{i,r}}f_{\epsilon,n}(X)d\sigma_{V,r}=&\int_{|q(x_i)|\leq m^{-e}|r|}f_{\epsilon, n, V_i}(x_i) \wtilde{f_{\epsilon,n,V_i^\prime}}(r-q(x_i))dx_i\\
=&\int_{\epsilon^2\leq |q(x_i)|\leq m^{-e}|r|}|x_i|^{-n} \wtilde{f_{\epsilon,n,V_i^\prime}}(r-q(x_i))dx_i.
\end{align*}
By induction, there exist constants $B_i, b_i, \delta_i\in \br_+$ such that $\wtilde{f_{\epsilon,n,V_i^\prime}}(r)$ is bounded by $B_i$ and $\wtilde{f_{\epsilon,V_i^\prime}}(r)\leq b_i |r|^{-\frac{n+1}{2}}$ when $|r|\geq \delta_i$.

Observe that the integral $\int_{\epsilon^2\leq |q(x_i)|\leq m^{-e}|r|}|x_i|^{-n} dx_i$ is bounded by a constant $a$ independent of $i$. Hence $g_i(r)$ is bounded by $aB_i$. Recall that when $|q(x_i)|\leq m^{-e}|r|$, there is $|q(X)-q(x_i)|\geq (1-\frac{1}{m})^e|r|$. Hence when $|r|\geq (1-\frac{1}{m})^{-e}\delta_i$,
\[
g_i(r)\leq b_i\big[(1-\frac{1}{m})^e|r|\big]^{-\frac{n+1}{2}} \int_{\epsilon^2\leq |q(x_i)|\leq m^{-e}|r|}|x_i|^{-n} dx_i\leq ab_{\epsilon,i}(1-m^{-1})^{-\frac{(n+1)e}{2}}|r|^{-\frac{n+1}{2}}.
\]

Set $B=a\sum_i B_i$, $\delta=(1-\frac{1}{m})^{-e}\max\{\delta_i:i\}$, $b=a(1-m^{-1})^{-\frac{(n+1)e}{2}}\sum_i b_i$, then $\wtilde{f_{\epsilon.n}}(r)\leq \sum_i g_i(r)$ is bounded by $B$ and, if $|r|\geq \delta$, further bounded by $b|r|^{-\frac{n+1}{2}}$.
\end{proof}

\begin{lemma}\label{estimate2}
Suppose that $k$ is archimedean and $f\in \ms(V)$ vanishes in an open neighborhood of $0$, then $\wtilde{f}(r)$ is bounded on $k^\times$ and decreases faster near $\infty$ than $|r|^{-n}$ for all $n\in \bn$.
\end{lemma}
\begin{proof}
Choose an isomorphism $V=k^m$ so that $q(X)=\sum_{i=1}^m c_i x_i^2$ with $|c_i|=1$; let $k_{(i)}$ be the $i$-th copy of $k$. For a subset $I$ of $\{1,\cdots,m\}$, put $V_I=\oplus_{i\in I} k_{(i)}$ and $V_I^\prime=\oplus_{j\not\in I} k_{(j)}$.

Because $f$ vanishes in an open neighborhood of $0$, there exists $\epsilon\in \br^+$ so that $f(X)=0$ when all $|x_i|\leq \epsilon$. Put $D_{I,\epsilon}=\{X_I\in V_I: |x_{(i)}|\leq \epsilon, i\in I\}$, then the condition $f\in \ms(V)$ implies that given $n\in \bn$, there exist constants $\{c_{I,\epsilon,n}:I\subset\{1,\cdots,m\}\}$ so that
\[
|f|\leq \sum_I c_{I,\epsilon, n} 1_{D_{I,\epsilon}}(X_I)f_{\epsilon.n,V_I^\prime}(X_I^\prime)
\]
Thus,
\[
|\wtilde{f}(r)|\leq \wtilde{|f|}(r)\leq \sum_I c_I \int_{D_{I,\epsilon}} \wtilde{f_{\epsilon,n,V_I^\prime}}(r-q(X_I))dX_I.
\]
Applying the estimates for $\wtilde{f_{\epsilon,n,V_I^\prime}}(r)$, one easily sees that $|\wtilde{f}(r)|$ is bounded and that there exist $\delta_0, b_0\in \br_+$ so that $|\wtilde{f}(r)|\leq b_0|r|^{-\frac{n+1}{2}}$ when $|r|\geq \delta_0$. Because this works for all $n\in \bn$, $\wtilde{f}(r)$ decreases faster near $\infty$ than $|r|^{-n}$ for all $n\in \bn$.
\end{proof}

\begin{proposition}\label{qft_V}
If $\phi\in \ms(V)$ and $\dim V\geq 2$, then $\wtilde{\phi}(r)\in L^2(k)$.
\end{proposition}
\begin{proof}
When $k$ is $p$-adic, it follows from Lemma ~\ref{estimate1}. When $k$ is archimedean, one writes $\phi=\phi_1+\phi_2$, with $\phi_1\in C^\infty_c(V)$ and $\phi_2$ vanishing in an open neighborhood of $0$, then $\wtilde{\phi}_1(r)$ is square-integrable by Lemma ~\ref{estimate1} and $\wtilde{\phi}_2(r)$ is square-integrable by Lemma ~\ref{estimate2}, whence $\wtilde{\phi}(r)=\wtilde{\phi}_1(r)+\wtilde{\phi}_2(r)$ is also square-integrable.
\end{proof}

\begin{proposition}\label{qft_V2}
Suppose that $W(x)\in L^1(k)$ is an even function and $\phi\in \ms(V)$. For $\delta\in k^\times$, set $V_{[\delta]}=\{X\in V:q(X)\in \delta {k^\times}^2\}$. If $W^2(x)|x|^{-1}\in L^1(k)$ and $\dim V\geq 2$, then
\begin{equation}\label{eq_qfti}
\int_k \mf_2W(\delta t)\mf_q\phi(-t)dt=2|2\delta|^{-1}\int_{V_{[\delta]}} W(\sqrt{\delta^{-1}q(X)})\phi(X)|q(X)|^{-\frac{1}{2}}dX.
\end{equation}
\end{proposition}
\begin{proof}
One has $\int_k \mf_2W(\delta t)\mf_q\phi(-t)dt=\int_k \mf\wtilde{W}(\delta t)\mf\wtilde{\phi}(-t)dt$. Because $\dim V\geq 2$, the function $\wtilde{\phi}(r)\in L^2(k)$ by Proposition ~\ref{qft_V}. On the other hand, the condition $W^2(x)|x|^{-1}\in L^1(k)$ means that $\wtilde{W}(r)\in L^2(k)$. Since $\mf\wtilde{W}(\delta \cdot)=|\delta|^{-1}\mf(\wtilde{W}(\delta^{-1}\cdot))$, we apply the $L^2$-isometry property of $\mf$ to get
\[
\int_k \mf\wtilde{W}(\delta t)\mf\wtilde{\phi}(-t)dt=\int_k |\delta|^{-1}\wtilde{W}(\delta^{-1}r)\wtilde{\phi}(r)dr=|\delta|^{-1}\int_V \wtilde{W}(\delta^{-1}q(X))\phi(X)dX.
\]
Note that $\wtilde{W}(\delta^{-1}q(X))$ takes value $2|2|^{-1}W(\sqrt{\delta^{-1}q(X)})|q(X)|^{-\frac{1}{2}}$ when $q(X)\in V_{[\delta]}$ and vanishes when $X\in V^\circ-V_{[\delta]}$. Hence the proposition follows.
\end{proof}

\begin{remark}
When $\dim V=1$, the function $\wtilde{\phi}(r)$ associated to $\phi\in \ms(V)$ is square-integrable if and only if $\phi(0)=0$. Thus, extra conditions are needed to make Equality (\ref{eq_qfti}) hold. It is proved in \cite{qiuy2019} that $W^2(x)|x|^{-1}, W(x)|x|^{-1}, \mf_2W(t)|t|^{-\frac{1}{2}}\in L^1(k)$ are sufficient conditions.
\end{remark}

\section{Saito-Kurokawa Packets}

Now let $(V,q)$ be a $5$-dimensional quadratic space over a number field $F$ with positive Witt index and determinant $1\cdot{F^\times}^2$ as in the introduction. Recall from Section ~\ref{bgroup} the decomposition 
\[
V=Fv_+\oplus U\oplus Fv_-=Fv_+\oplus (Fv_0\oplus V_0)\oplus Fv_-
\]
the Bessel subgroup $\mrr$, and the parabolic subgroup $P$ stabilizing $Fv_+$. %The modulus character on $P(\ba)$ is $\delta_P(u(a,h_0))=|a|^{2\rho_P}$ with $\rho_P=\frac{3}{2}$.

We will use the Weil representation of $\wtilde{\SL}_2(\ba)\times \SO(V)_\ba$ on the mixed model $\ms(U_\ba\oplus \ba^2)$. The intertwining map between the Schr\"{o}dinger model $\ms(V_\ba)$ and the mixed model is the partial Fourier transform given by
\[
\ms(V_\ba)\rar \ms(U_\ba\oplus \ba^2),\quad \phi(X)\lrar \widehat{\phi}(u;x,y)=\int_{\ba}\phi(zv_+ +u+ xv_-)\psi(zy)dz.
\]

\subsection{Saito-Kurokawa packets of $\SO(V)$}\label{subsec_skpackets}

Suppose that $\pi\subset \ma_{cusp}(\PGL_2)$ is irreducible. The Saito-Kurokawa packet $\SK(\pi)$ of $\SO(V)_\ba$ is constructed from the Waldspurger packet $Wd_\psi(\pi)$ of the metaplectic group $\wtilde{SL}_2(\ba)$.
\begin{itemize}
\item[(i)] Identify $\PGL_2$ with $\SO(U^\prime,q^\prime)$, where $U^\prime=\{X\in M_{2\times 2}(F):\Tr(X)=0\}$, $q^\prime(X)=-\det X$ and $\PGL_2$ acts on $U^\prime$ by $g\circ X=g^{-1}Xg$. Consider the Weil representation $\omega_{\psi,U^\prime}$ of $\wtilde{\SL}_2(\ba)\times \PGL_2(\ba)$ on $\ms(U^\prime_\ba)$ and the associated global theta lifting. Define
\[
Wd_\psi(\pi)=\{\Theta_{\wtilde{\SL}_2\times \PGL_2}(\pi\otimes \chi_a,\psi_a):a\in F^\times\}.
\]

\item[(ii)] Consider the Weil representation $\omega_{\psi,V}$ of $\wtilde{\SL}_2(\ba)\times \SO(V)_\ba$ on $\ms(V_\ba)$. Define
\[
\SK(\pi)=\{\Theta_{\wtilde{\SL}_2\times \SO(V)}(\sigma,\psi):\sigma\in Wd_\psi(\pi)\}.
\]
\end{itemize}

\begin{remark}
For $\sigma\in Wd_\psi(\pi)$, the lift $\Theta_{\wtilde{\SL}_2\times \SO(V)}(\sigma,\psi)$ is irreducible and in the discrete spectrum of $\SO(V)$. It is cuspidal when $\Theta_{\wtilde{\SL}_2\times \SO(U)}(\sigma,\psi)= 0$; it is the residue representation of the Eisenstein series associated to $\Ind_P^{\SO(V)}(|\cdot|^s\boxtimes \pi)$ at $s=\frac{3}{2}$ when $\Theta_{\wtilde{\SL}_2\times \SO(U)}(\sigma,\psi)\neq 0$.

Note that by the inner product formula (See Proposition ~\ref{IPF_SO5}) and the knowledge of degenerate principle series, one knows that $\Theta_{\wtilde{\SL}_2\times \SO(V)}(\sigma,\psi)\neq 0$ if and only if $\theta_{\wtilde{\SL}_2\times \SO(V)}(\sigma_v,\psi_v)\neq 0$ for all $v$. Let $S_U$ be the set of places $v$ such that $U(F_v)$ is anisotropic. For $v\not\in S_U$, there is always $\theta_{\wtilde{\SL}_2\times \SO(V)}(\sigma_v,\psi_v)\neq 0$. For $v\in S_U$, $\theta_{\wtilde{\SL}_2\times \SO(V)}(\sigma_v,\psi_v)\neq 0$ happens when (a) $\sigma_v$ is not the odd or even Weil representation,  if $v$ is finite, (b) $\sigma_v$ is not the odd or even Weil representation or local theta lifts of discrete series on $\PGL_2(F_v)$, if $v$ is real (c.f. \cite[Prop. 6.2]{wgan2008}).
\end{remark}

In general, for a quadratic character $\mu$ of $\ba^\times/F^\times$, define the packet
\[
\SK(\pi,\mu):=\{\Pi\otimes \mu: \Pi\in \SK(\pi\otimes \mu)\}.
\]
For $\Pi\in \SK(\pi,\mu)$, its partial $L$-function is $L^S(\Pi)=L(s,\pi)L(s+\frac{1}{2},\mu)L(s-\frac{1}{2},\mu)$. We will study the Bessel period functional on the packet $\SK(\pi)=\SK(\pi,1)$ in details; the according properties of $\SK(\pi,\mu)$ follow quickly and we skip the according discussion.

\subsubsection{Inner product formula}

Define the inner product pairing on $L^2([\SO(V)])$ by $(F_1,F_2)=\int_{[\SO(V)]}F_1(h)\overline{F_2(h)}\dd h$ and the pairing on $L^2([\wtilde{\SL}_2])$ by  $(\varphi_1,\varphi_2)=\int_{[\SL_2]}\varphi_1(g)\overline{\varphi}_2(g)dg$.

\begin{proposition}\label{IPF_SO5} Suppose that $\sigma\in Wd_\psi(\pi)$ and $\Pi=\Theta_{\wtilde{SL}_2\times \SO(V)}(\sigma,\psi)$ is cuspidal. For $\varphi_i\in \sigma$, $\phi_i\in \ms(V_\ba)$ (i=1,2), there is
\[
\big(\Theta(\phi_1,\varphi_1),\Theta(\phi_2,\varphi_2)\big)=\frac{2L(\frac{3}{2},\pi)}{\zeta_F(4)}\prod_v \frac{\zeta_{F_v}(4)}{L(\frac{3}{2},\pi_v)}\int_{\SL_2(F_v)}\overline{(\sigma(g_v)\varphi_{1,v},\varphi_{2,v}) }(\omega_v(g_v)\phi_{1,v},\phi_{2,v})dg_v.
\]
\end{proposition}
When $V$ is of Witt index $2$, the formula is proved in \cite{qiu14} using regularized Siegel-Weil formula. When $V$ is of Witt index $1$, it follows from the Siegel-Weil formula proved in \cite{kudla_rallis88-2}.

One may interpret Proposition ~\ref{IPF_SO5} in an alternative way. For $\Pi=\Theta_{\wtilde{\SL}_2\times \SO(V)}(\sigma,\psi)\in \SK(\pi)$ with $\sigma\in Wd_\psi(\pi)$, one chooses a local theta lifting $\theta_v:\ms(V(F_v))\otimes \sigma_v\rar \Pi_v$ at each place so that $\Theta(\phi,\varphi)=\otimes_v \theta_v(\phi_v,\varphi_v)$. Then for $f_{i,v}=\theta_v(\phi_{i,v},\varphi_{i,v})\in \Pi_v$, there is
\begin{equation}\label{lipfso5}
(f_{1,v},f_{2,v})=C_{\sigma_v,\Pi_v}\int_{\SL_2(F_v)}\overline{(\sigma_v(g_v)\varphi_{1,v},\varphi_{2,v})}(\omega_v(g_v)\phi_{1,v},\phi_{2,v})dg_v,
\end{equation}
where $C_{\sigma_v,\Pi_v}=\frac{\zeta_{F_v}(4)}{L(\frac{3}{2},\pi_v)}$ almost everywhere. The Proposition ~\ref{IPF_SO5} says $\prod_v \frac{C_{\sigma_v,\Pi_v}L(\frac{3}{2},\pi_v)}{\zeta_{F_v}(4)}=\frac{2L(\frac{3}{2},\pi)}{\zeta_F(4)}$.

\subsubsection{Whittaker functionals on $\wtilde{\SL}_2$}\label{wfsl2}

Let $B\subset \SL_2$ be the Borel subgroup of upper triangular matrices and $N$ be its unipotent radical. $N(\ba)$ can be lifted to $\wtilde{\SL}_2(\ba)$ and we denote its lift by the same notation. For $a\in \ba^\times$, write $\underline{a}$ for $(\smalltwomatrix{a}{}{}{a^{-1}},1)\in \wtilde{\SL}_2(\ba)$.

Let $\bk=\prod_v \bk_v$ be a maximal compact subgroup of $\SL_2(\ba)$. With respect to the Iwasawa decomposition $g=\smalltwomatrix{1}{n}{}{1}\smalltwomatrix{a}{}{}{a^{-1}}k$, we do the according measure decomposition $\dd g=|a|^{-2}\dd n \dd^\times a \dd k$, where $\dd k=\prod_v \dd k_v$ with $\int_{\bk_v}1\dd k_v=1$ for almost all $v$.

Define the $\psi$-Whittaker functional on $\ma_{cusp}(\wtilde{\SL}_2)$ by
\[
\ell_{\psi}(\varphi)=\int_{N(\ba)/N(F)}\varphi(n)\psi(-x)dx
\]

Suppose that $\sigma\subset \ma_{cusp}(\wtilde{\SL}_2)$ be an irreducible $\wtilde{SL}_2(\ba)$-representation. Write $\sigma=\otimes \sigma_v$ as a restricted tensor product and, for spherical $\sigma_v$, let $\varphi_{v,0}$ denote the spherical vector in $\sigma_v$ used to construct the restricted tensor product. There is decomposition $(,)_\sigma=\prod_v (,)_{\sigma_v}$, where $(,)_{\sigma_v}$ are local inner product pairings on $\sigma_v$ and $(\varphi_{v,0},\varphi_{v,0})_{\sigma_v}=1$ for almost all $v$. If $\ell_{\psi}|_\sigma\neq 0$, there is decomposition $\ell_\psi|_\sigma=\prod_v \ell_{\psi_v}$, where $\ell_{\psi_v}$ are local $\psi_v$-Whittaker functionals on $\sigma_v$ and $\ell_{\psi_v}(\varphi_{v,0})=1$ for almost all $v$.

For each local place $v$, choose a set of representatives $\{\delta_{v,i}\}$ of $F_v^\times/{F_v^\times}^2$ with $\delta_{v,1}=1$. For each $\delta_{v,i}$, choose a non-zero Whittaker functional $\ell_{\psi_{v,\delta_{v,i}}}$ on $\sigma_v$ with respect to $\psi_{v,\delta_{v,i}}$ if such functionals exist, or set $\ell_{\psi_{v,\delta_{v,i}}}=0$. For $\varphi_v\in \sigma_v$, put $W_{\varphi_v,\psi_{v,\delta_{v,i}}}(g_v)=\ell_{\psi_{v,\delta_{v,i}}}(\sigma_v(g_v)\varphi_v)$. It is known from \cite{baruch_mao03, baruch_mao05, qiuy2019} that there are real constants $c_{\sigma_v,\psi_{v,\delta_{v,i}}}$ such that
\begin{equation}\label{ipwfsl2}
(\varphi_{1,v},\varphi_{2,v})_{\sigma_v}=\sum_i 2^{-1}|2|_v c_{\sigma_v,\psi_{v,\delta_{v,i}}}\int_{F_v^\times} W_{\varphi_{1,v},\psi_{v,\delta_{v,i}}}\overline{W_{\varphi_{1,v},\psi_{v,\delta_{v,i}}}}(\underline{a_v})\dd^\times a_v.
\end{equation}
The constant $c_{\sigma_v,\psi_{v,\delta_{v,i}}}$ is determined by $\ell_{\psi_{v,\delta_{v,i}}}$ and $(,)_{\sigma_v}$;  it is zero when $\ell_{\psi_{v,\delta_{v,i}}}=0$ and positive when $\ell_{\psi_{v,\delta_{v,i}}}\neq 0$. By Propositions 3.1 and 4.2 of \cite{qiuy2019}, we have the following local-global relation.
\begin{proposition}\label{wfdsl2}\cite{qiuy2019}
Suppose that $\sigma\in Wd_\psi(\pi)$ and $\ell_\psi$ is non-zero on $\psi$. Then $c_{\sigma_v,\psi_v}=\frac{L(\frac{1}{2},\pi_v)\zeta_{F_v}(2)}{L(1,\pi_v,\ad)}$ for almost all $v$ and $\prod_v \frac{c_{\sigma_v,\psi_v}L(1,\pi_v,\ad)}{L(\frac{1}{2},\pi_v)\zeta_{F_v}(2)}=2\cdot \frac{L(1,\pi,\ad)}{L(\frac{1}{2},\pi)\zeta_F(2)}$
\end{proposition}

\subsection{Global Bessel periods}\label{subsec_gbp}

Suppose $\Pi\in \SK(\pi)$ is cuspidal. Write $\sigma=\Theta_{\wtilde{\SL}_2\times \SO(V)}(\Pi,\psi)$, then $\Theta_{\wtilde{\SL}_2\times \SO(U)}(\sigma,\psi)=0$ and $\ell_{\psi}|_\sigma=0$ because $\Pi$ is cuspidal. Recall the global Bessel period functional in (\ref{gbp}).
\begin{proposition}\label{gperiod}
$P(\cdot)$ vanishes on $\Pi$ if $q(v_0)\in {F^\times}^2$ or $\chi\neq 1$. When $q(v_0)\not\in {F^\times}^2$ and $\chi=1$, for $f=\Theta(\phi,\varphi)\in \Pi$ with $\phi\in \ms(V_\ba)$ and $\varphi\in \sigma$, there is
\begin{equation}\label{gp_sk}
P(f)=2\int_{N(\ba)\backslash \SL_2(\ba)}\overline{W_{\varphi,q(v_0)}(g)}\omega(g)\widehat{\phi}(v_0;0,1)dg.
\end{equation}
\end{proposition}
\begin{proof}
Suppose $f=\Theta(\phi,\varphi)\in \Pi$. Observe that the theta kernel $\Theta_\phi(g,h)$ is a sum,
\begin{equation}\label{thetakernel}
\Theta_\phi(g,h)
=\sum_{\gamma\in N(F)\backslash \SL_2(F)}\sum_{\xi^\prime\in U(F)}\omega(\gamma g,h)\widehat{\phi}(\xi^\prime;0,1)+\sum_{\xi^\prime\in U(F)}\omega(g,h)\widehat{\phi}(\xi^\prime;0,0).
\end{equation}
Because $\Theta_{\wtilde{\SL}_2\times \SO(U)}(\sigma,\psi)=0$, $\varphi(g)$ is orthogonal to the second summand, whence
\[
f(h)=\int_{[\SL_2]}\overline{\varphi(g)}\Theta_\phi(g,h)dg=\int_{N(F)\backslash \SL_2(\ba)}\overline{\varphi(g)}\sum_{\xi^\prime\in U(F)}\omega(\gamma g,h)\widehat{\phi}(\xi^\prime;0,1)dg.
\]

It follows that
\begin{align*}
&\int_{[U]}f(uh)\psi_{v_0}(u)dn\\
=&\int_{[U]}\int_{N(F)\backslash \SL_2(\ba)}\overline{\varphi(g)}\sum_{\xi^\prime\in U(F)}\omega(\gamma g,h)\widehat{\phi}(\xi^\prime;0,1)\psi(q(u,v_0-\xi^\prime)) \dd g \dd u\\
=&\int_{N(F)\backslash \SL_2(\ba)}\overline{\varphi(g)}\omega(\gamma g,h)\widehat{\phi}(v_0;0,1)dg\\
=&\int_{N(\ba)\backslash \SL_2(\ba)}\overline{W_{\varphi,\psi_{q(v_0)}}(g)}\omega(g,h)\widehat{\phi}(v_0;0,1)dg.
\end{align*}

When $q(v_0)\in {F^\times}^2$, the above integral is zero because $W_{\varphi,\psi_{q(v_0)}}(g)$ is zero, whence $P_{\Pi,\chi}=0$. When $q(v_0)\not\in {F^\times}^2$, the $h$-function $\int_{[U]}f(uh)\psi_{v_0}(u)du$ is constant on $[SO(V_0)]$, whence $P(f)$ is always zero when $\chi\neq 1$ and is equal to the following expression when $\chi=1$,
\begin{align*}
P(f)
=&\int_{[SO(V_0)]}\int_{[U]}f(uh)\psi_{v_0}(u)dudh\\
=&2\int_{N(\ba)\backslash \SL_2(\ba)}\overline{W_{\varphi,q(v_0)}(g)}\omega(g)\widehat{\phi}(v_0;0,1)dg.
\end{align*}
\end{proof}

We define a local functional $I_v$ on $\ms(V(F_v))\otimes \sigma_v$:
\[
I_v(\phi_v,\varphi_v):=\int_{N(F_v)\backslash \SL_2(F_v)}\overline{W_{\varphi_v,\psi_{v,q(v_0)}}(g_v)}\omega_v(g_v)\widehat{\phi}_v(v_0;0,1)dg_v.
\]

\subsection{The Local Bessel period}\label{local_bessel}

Suppose $\Pi\in \SK(\pi)$ is cuspidal. For $f_{1,v}, f_{2,v}\in \Pi_v$, we follow the regularization method in section ~\ref{subsec_lbp} and consider
\[
\mpp_v(s; f_{1,v},f_{2,v}):=\int_{\SO(V_0)_{F_v}}\big[\int_{U(F_v)}(\Pi_v(u_vh_{v})f_{1,v},f_{2,v})\psi_{v,v_0}(u_v)du_v\big]\chi_v(h_{v})\Delta_v(h_{v})^s \dd h_{v},
\]
where the $U(F_v)$-integration is interpreted in terms of distribution as in section ~\ref{ss_uintegration}.

In the current situation, the $L$-function $L(s,\Pi,\chi)$ is equal to
\begin{equation}\label{lf_skp}
\frac{L(s+\frac{1}{2},\pi\boxtimes \chi)L(s+1,\chi)L(s,\chi)}{L(s+1,\pi,\ad)L(s+\frac{3}{2},\pi)L(s+\frac{1}{2},\pi)\zeta_F(s+2)\zeta_F(s+1)\zeta_F(s)L(s+1,\chi_{E/F})}
\end{equation}
Accordingly, the regularized local Bessel period integral is
\begin{align*}
\mpp_v^\sharp(-):=&\frac{\zeta_{E_v}(1)L(\frac{1}{2},\pi_v)L(\frac{3}{2},\pi_v)L(1,\pi_v,\ad)}{\zeta_{F_v}(4)L(1,\chi_v)L(\frac{1}{2},\pi_v\boxtimes \chi_v)}\cdot \frac{\zeta_{F_v}(s)P_{v}(s,-)}{L(s,\chi_v)}\big|_{s=0}.
\end{align*}

\subsubsection{The $U(F_v)$-integration}\label{ss_ui}

To compute the $U(F_v)$-integration, one needs to find the smooth function $W_{f_{1,v},f_{2,v},\psi_v}(x_v)$ on $U(F_v)^\circ$ such that for all $f\in C^\infty_c(U(F_v)^\circ)$,
\begin{equation}\label{i_d}
\int_{U(F_v)}(\Pi_v(u_v)f_{1,v},f_{2,v})\widehat{f}(u_v)du_v=\int_{U(F_v)}W_{f_{1,v},f_{2,v},\psi_v}(x_v)f(x_v)dx_v.
\end{equation}

For $\phi_v\in \ms(V(F_v))$ and $\varphi_v\in \sigma_v$, define
\begin{equation}\label{func_I}
I_v(x_v;\phi_v,\varphi_v):=\int_{N(F_v)\backslash \SL_2(F_v)}\overline{W_{\varphi_v,\psi_{v,q(x_v)}}(g_v)}\omega_v(g_v)\widehat{\phi}_v(x_v;0,1)dg_v,\quad x_v\in U(F_v)^\circ.
\end{equation}

\begin{proposition}\label{prop_ftdistribution}
If $f_{i,v}=\theta_v(\phi_{i,v},\varphi_{i,v})$, the following function satisfies equation (\ref{i_d}),
\begin{equation}\label{ui_ft}
W_{f_{1,v},f_{2,v},\psi_v}(x_v)=C_{\sigma_v,\Pi_v}c_{\sigma_v,\psi_{v,q(x_v)}}c_v|q(x_v)|^{-1}I_v(x_v;\phi_{1,v},\varphi_{1,v})\overline{I_v(x_v;\phi_{2,v},\varphi_{2,v})}.
\end{equation}
As a consequence,
\begin{equation}\label{ui_skp_1}
\int_{U(F_v)}(\Pi_v(u_v)f_{1,v},f_{2,v})\psi_{v_0}(u_v)du_v=C_{\sigma_v,\Pi_v}c_{\sigma_v,\psi_{v,q(v_0)}}c_vI_v(\phi_{1,v},\varphi_{1,v})\overline{I_v(\phi_{2,v},\varphi_{2,v})}.
\end{equation}
\end{proposition}

We will prove Proposition ~\ref{prop_ftdistribution} after preparing two lemmas. Recall the local constants $C_{\sigma_v,\Pi_v}$ in (\ref{lipfso5}) and $c_{\sigma_v,\psi_{v,{\delta_{v,i}}}}$ in (\ref{ipwfsl2}). With respect to the right translation action of $\SL_2(F_v)$ on $F_v\oplus F_v$ and the quotient measure $d\dot{g}_v$ on $N(F_v)\backslash \SL_2(F_v)$, there is a constant $c_v$ such that for all $\lambda\in L^1(F_v^2)$,
\begin{equation}\label{vc}
\int_{F_v^2}\lambda(y_v,z_v)dy_vdz_v=c_v \int_{N(F_v)\backslash \SL_2(F_v)}\lambda\big((0,1)\dot{g}_v\big)d\dot{g}_v.
\end{equation}
One has $\prod_v c_v=1$ with $c_v=\frac{1}{\zeta_{F_v}(2)}$ for almost all $v$.

For each $r_v\in {F_v^\times}^2$, choose one of its two square-roots and denote it by $(r_v)^{1/2}$. For $r_v\in F_v^\times$, define the $\psi_{v,r_v}$-Whittaker functional $\sigma_v$ as $\ell_{\psi_{v,r_v}}(\varphi_v)=\ell_{\psi_{v,\delta_{v,i}}}\big(\sigma_v(\underline{a}_v)\varphi_v\big)$ if $r_v\in\delta_{v,i}{F_v^\times}^2$ and $a_v=(\delta_{v,i}^{-1}r_v)^{1/2}$. Set $W_{\varphi,\psi_{v,r_v}}(g_v)=\ell_{\psi_{v,r_v}}\big(\sigma_v(g_v)\varphi_v\big)$.

\begin{lemma}\label{lippi}
For $f_{i,v}=\theta_v(\phi_{i,v},\varphi_{i,v})$, there is
\begin{align*}
(f_{1,v},f_{2,v})=&C_{\sigma_v,\Pi_v}c_v\int_{\SL_2(F_v)\times K_v}\int_{U(F_v)\times F_v^\times} \overline{(\sigma_v(g_v)\varphi_{1,v},\leftup{k_v}\varphi_{2,v})}\\
&\quad\quad \omega_v(g_v)\widehat{\phi_{1,v}}(x_v;0,a_v^{-1})\leftup{k_v}\overline{\widehat{\phi}_{2,v}(x_v;0,a_v^{-1})}|a_v|^{-2}dx_vd^\ast a_vdg_vdk_v.
\end{align*}
\end{lemma}
\begin{proof}
By (\ref{vc}), we have
\begin{align*}
(\phi_{1,v},\phi_{2,v})=&\int_{U(F_v)\times F_v\times F_v}\hat{\phi}_{1,v}(x_v;y_v,z_v)\overline{\hat{\phi}_{2,v}(x_v;y_v,z_v)}dx_vdy_vdz_v\\
=&c_v \int_{U(F_v)\times (N(F_v)\backslash \SL_2(F_v))}\hat{\phi}_{1,v}(x_v;(0,1)g_v^\prime)\overline{\hat{\phi}_{2,v}(x_v;(0,1)g_v^\prime)}dx_v d\dot{g}_v^\prime\\
=&c_v \int_{K_v}\int_{U(F_v)\times F_v^\times}\leftup{k_v}{\hat{\phi}}_{1,v}(x_v;0,a_v^{-1})\overline{\leftup{k_v}{\hat{\phi}}_{2,v}(x_v;0,a_v^{-1})}|a_v|^{-2}dx_v d^\ast a_vdk_v.
\end{align*}
Here one writes $\dot{g}_v^\prime=\smalltwomatrix{a_{v}}{}{}{a_{v}^{-1}}k_{v}$ with $k_v\in K_v$ and  $d\dot{g}_v^\prime=|a_{v}|^{-2}d^\ast a_{v} dk_{v}$ in the last step.

Using the above to rewrite the local inner product formula in equation (\ref{lipfso5}), we get
\[
(f_{1,v},f_{2,v})=C_{\sigma_v,\Pi_v}c_v\int_{\SL_2(F_v)}\int_{K_v}A(g_v,k_v)dg_vdk_v,
\]
where
\[
A(g_v,k_v)=\int_{U(F_v)\times F_v^\times}\overline{(\omega_v(g_v)\varphi_{1,v},\varphi_{2,v})}\omega_v(k_vg_v){\hat{\phi}}_{1,v}(x_v;0,a_v^{-1})\overline{\leftup{k_v}{\hat{\phi}}_{2,v}(x_v;0,a_v^{-1})}|a_v|^{-2}dx_v d^\ast a_v.
\]
Because $\sigma_v$ is unitary, one can use the asymptotic estimate of the matrix coefficient of $\sigma_v$ to show that $\iint_{\SL_2(F_v)\times K_v}A(k_v^{-1}g_v,k_v)dg_vdk_v$ absolutely converges, whence
\[
(f_{1,v},f_{2,v})=C_{\sigma_v,\Pi_v}c_v\int_{\SL_2(F_v)\times K_v}A(k_v^{-1}g_v,k_v)dg_vdk_v.
\]
This is exactly the equation in the lemma.
\end{proof}

\begin{lemma}\label{sl2iv}
For $\varphi_{1,v},\varphi_{2,v}\in \sigma_v$ and $\Phi\in \ms(U(F_v))$, there is
\begin{align*}
&\int_{N(F_v)}\int_{U(F_v)}\overline{(\sigma(n)\varphi_1,\varphi_2)}\Phi(x_v)\psi_v(n_vq(x_v))dx_vdn_v\\
=&\sum_i c_{\sigma_v,\psi_{\delta_{v,i}}}\int_{U(F_v)_{[\delta_{v,i}]}} \overline{\ell_{\psi_{v,q(x_v)}}(\varphi_{1,v})}\ell_{\psi_{v,q(x_v)}}(\varphi_{2,v})\Phi(x_v)|q(x_v)|^{-1}dx_v,
\end{align*}
where $U(F_v)_{\delta_{v,i}}=\{x_v\in U(F_v):q(x_v)\in \delta_{v,i}{F_v^\times}^2\}$.
\end{lemma}
\begin{proof}
Set $W_i(a_v)=\overline{W_{\varphi_{1,v},\psi_{v,\delta_{v,i}}}}W_{\varphi_{1,v},\psi_{v,\delta_{v,i}}}(a_v)|a_v|^{-1}$, $a_v\in F_v^\times$. By equation (\ref{ipwfsl2}),
\begin{align*}
\overline{(\sigma_v(n_v)\varphi_{1,v},\varphi_{2,v})}
=&\sum_i 2^{-1}|2|c_{\sigma_v,\psi_{v,{\delta_i}}}\int_k \overline{W_{\varphi_{1,v},\psi_{v,{\delta_{v,i}}}}}W_{\varphi_{2,v},\psi_{v,\delta_{v,i}}}(a_v)\psi(-\delta_{i,v} n_va_v^2)\dd^\times a_v\\
=&\sum_{i} 2^{-1}|2|c_{\sigma_v,\psi_{v,\delta_{v,i}}} \mf_2 W_i(-\delta_{v,i} n_v),
\end{align*}
Hence the LHS of the equation in lemma is
\[
\sum_{i} 2^{-1}|2|c_{\sigma_v,\psi_{v,\delta_{v,i}}}\int_{F_v}\mf_2W_i(-\delta_{v,i}n_v)\mf_q\Phi(n_v)dn_v.
\]

Because $\sigma_v$ is unitary, one can uses the asymptotic estimate of the Whittaker functions $W_{\varphi_v,\psi_{v,\delta_i}}(a)$ (c.f. \cite{qiuy2019}, remark 5) to show that $W_i^2(a)|a|^{-1}\in L^1(F_v)$. Hence, by Proposition ~\ref{qft_V2},
\begin{align*}
&\int_{N(F_v)}\mf_2 W_i(-\delta_{v,i} n_v)\mf_{q}\Phi(n_v)dn_v\\
=&2|2\delta_{v,i}|^{-\frac{1}{2}}\int_{U(F_v)_{[\delta_{v,i}]}} W_i(\sqrt{\delta_{v,i}^{-1}q(x_v)})\Phi(x_v)|q(x_v)|^{-\frac{1}{2}}dx_v\\
=&2|2|^{-1}\int_{V_{[\delta_{v,i}]}} \overline{\ell_{\psi_{v,q(x_v)}}(\varphi_1)}\ell_{\psi_{v,q(x_v)}}(\varphi_2)\Phi(x_v)|q(x_v)|^{-1}dx_v.
\end{align*}
Here we use the observation that when $q(x_v)=\delta_{v,i} a_v^2$, the function $W_i(\sqrt{\delta_{v,i}^{-1}q(x_v)})$ takes value $W_i(a_v)=|\delta_{v,i}|^{\frac{1}{2}}\overline{\ell_{\psi_{v,q(x_v)}}(\varphi_{1,v})}\ell_{\psi_{v,q(x_v)}}(\varphi_{2,v})|q(x_v)|^{-\frac{1}{2}}$. Now the equation in the lemma follows.
\end{proof}

\begin{proof}[Proof of Proposition ~\ref{prop_ftdistribution}] By Lemma ~\ref{lippi}, the LHS of equation (\ref{i_d}) is equal to
\begin{align*}
&\text{LHS of (\ref{i_d})}=C_{\sigma_v,\Pi_v}c_v\int_{U(F_v)}\int_{\SL_2(F_v)\times K_v}\int_{U(F_v)\times F_v^\times} \overline{(\sigma_v(g_v)\varphi_{1,v},\leftup{k_v}\varphi_{2,v})}\\
& \cdot \omega_v(g_v)\widehat{\phi_{1,v}}(x_v;0,a_v^{-1})\overline{\leftup{k_v}\widehat{\phi}_{2,v}(x_v;0,a_v^{-1})}\hat{f}(u_v)|a_v|^{-2}\psi_v(-q(u_v,a_v^{-1}x_v))dx_vd^\ast a_vdg_vdk_vdu_v.
\end{align*}

We observe that the inner integration with respect to $dx_vd^\ast a_v$ on $U(F_v)\times F_v^\times$ yields a function of $(u_v,g_v,k_v)$ that is absolutely integrable on $U(F_v)\times \SL_2(F_v)\times K_v$, whence the outer iterated integral $\int_{U(F_v)}\int_{\SL_2(F_v)\times K_v}$ can be replaced by $\int_{\SL_2(F_v)\times K_v}\int_{U(F_v)}$ and the full order of integration is $\int_{\SL_2(F_v)\times K_v}\int_{U(F_v)}\int_{U(F_v)\times F_v^\times}$. Furthermore, the iterated integral $\int_{U(F_v)}\int_{U(F_v)\times F_v^\times}$ is equal to the absolutley convergent multiple integral $\int_{U(F_v)\times U(F_v)\times F_v^\times}$. Thus, one can integrate with respect to $du_v$ on $U(F_v)$ first; this yields
\begin{align*}
\text{LHS of (\ref{i_d})}=&C_{\sigma_v,\Pi_v}c_v\int_{\SL_2(F_v)\times K_v}\int_{U(F_v)\times F_v^\times} \overline{(\sigma_v(g_v)\varphi_{1,v},\leftup{k_v}\varphi_{2,v})}\\
&\cdot \omega_v(g_v)\widehat{\phi_{1,v}}(x_v;0,a_v^{-1})\overline{\leftup{k_v}\widehat{\phi}_{2,v}(x_v;0,a_v^{-1})}f(a_v^{-1}x_v)|a_v|^{-2}dx_vd^\ast a_vdg_vdk_v.
\end{align*}

We use the decomposition $g_v=n_v\smalltwomatrix{b_v}{}{}{b_v^{-1}}k_{v}^\prime$, $dg_v=|b_v|^{-2}dn_vd^\ast b_v dk^\prime_{v}$ to turn the outer integral $\int_{\SL_2(F_v)\times K_v}$ into an iterated integral $\int_{K_v\times K_v}\int_{F_v^\times}\int_{N(F_v)}$. We also rewrites the inner integral $\int_{U(F_v)\times F_v^\times}$ as an iterated integral $\int_{U(F_v)}\int_{F_v^\times}$. Then
\begin{align*}
\text{LHS of (\ref{i_d})}&=C_{\sigma_v,\Pi_v}c_v \int_{K_v\times K_v}\int_{F_v^\times}\int_{N(F_v)}\int_{U(F_v)}\int_{F_v^\times} \overline{(\sigma_v(n_v\underline{b_v})^\leftup{k_v^\prime}{\varphi}_{1,v},\leftup{k_v}\varphi_{2,v})}\\
&\cdot \omega_v(n_v\underline{b_v})\leftup{k_v^\prime}{\widehat{\phi_{1,v}}}(x_v;0,a_v^{-1})\overline{\leftup{k_v}\widehat{\phi}_{2,v}(x_v;0,a_v^{-1})}f(a_v^{-1}x_v)|a_vb_v|^{-2}dx_vd^\ast a_vdn_vd^\ast b_v dk_v^\prime dk_v.
\end{align*}

We now analyze inner part $\int_{N(F_v)}\int_{U(F_v)}\int_{F_v^\times}$ in the above expression. For convenience, set $\wtilde{\varphi}_{1,v}=\sigma_v(\underline{b_v}k_v^\prime)\varphi_{1,v}$, $\wtilde{\varphi}_{2,v}=\sigma_v(k_v)\varphi_{2,v}$, $\Phi_{1,v}=\omega_v(\underline{b_v}k_v^\prime)\phi_{1,v}$, $\Phi_{2,v}=\omega_v(k_v)\phi_{2,v}$
and
\[
\Phi_v(x_v)=\int_{F_v^\times}\Phi_{1,v}(x_v;0,a_v^{-1})\overline{\Phi_{2,v}(x_v;0,a_v^{-1})}f(a_v^{-1}x_v)|a_v|^{-2}d^\ast a_v.
\]
Then the inner integral $\int_{N(F_v)}\int_{U(F_v)}\int_{F_v^\times}$ with respect to $n_v, x_v, a_v$ is turned into
\[
|b_v|^{-2}\int_{N(F_v)}\int_{U(F_v)}\overline{(\sigma_v(n_v)\wtilde{\varphi}_{1,v},\wtilde{\varphi}_{2,v})}\Phi_v(x_v)\psi_v(q(x_v)n_v)dx_vdn_v,
\]
Which is further equal to the following by Lemma ~\ref{sl2iv},
\[
c_{\sigma_v,\psi_{v,q(x_v)}}|b_v|^{-2}\int_{U(F_v)}\overline{\ell_{\psi_{v,q(x_v)}}(\wtilde{\varphi}_{1,v})}\ell_{\psi_{v,q(x_v)}}(\wtilde{\varphi}_2)\Phi_v(x_v)|q(x_v)|^{-1}dx_v
\]
It follows that
\begin{align*}
&\text{LHS of (\ref{i_d})}\\
=&C_{\sigma_v,\Pi_v}c_v c_{\sigma_v,\psi_{v,q(x_v)}}\int_{K_v\times K_v}\int_{F_v^\times}\int_{U(F_v)}\int_{F_v^\times}\overline{\ell_{\psi_{v,q(b_vx_v)}}(\leftup{k_v^\prime}{\varphi}_{1,v})}\ell_{\psi_{v,q(x_v)}}(\leftup{k_v}{\varphi}_{2,v})\\
&\quad \cdot \omega_v(\underline{b_v})\leftup{k_v^\prime}{\widehat{\phi}}_{1,v}(x_v;0,a_v^{-1})\overline{\leftup{k_v}\widehat{\phi}_{2,v}(x_v;0,a_v^{-1})}f(a_v^{-1}x_v)|a_v^{-1}b_v^{-1}|^2|q(x_v)|^{-1}d^\ast a_vdx_vd^\ast b_v dk_vdk_v^\prime\\
=&C_{\sigma_v,\Pi_v}c_v c_{\sigma_v,\psi_{v,q(x_v)}}\int_{K_v\times K_v}\int_{F_v^\times}\int_{U(F_v)}\int_{F_v^\times}\overline{\ell_{\psi_{v,q(a_vb_vx_v)}}(\leftup{k_v^\prime}{\varphi}_{1,v})}\ell_{\psi_{v,q(a_vx_v)}}(\leftup{k_v}{\varphi}_{2,v})\\
&\quad \cdot \omega_v(\underline{a_v}\underline{b_v})\leftup{k_v^\prime}{\widehat{\phi}}_{1,v}(x_v;0,1)\overline{\omega_v(\underline{a_v})\leftup{k_v}\widehat{\phi}_{2,v}(x_v;0,1)}f(x_v)|a_v^4b_v^2|^{-1}|q(x_v)|^{-1}d^\ast a_vdx_vd^\ast b_v dk_vdk_v^\prime,
\end{align*}
Here we make a change of variable $x_v=a_vx_v$ for the second ``$=$".

The above integral is absolutely convergent and the change of variable $b_v=a_v^{-1}b_v$ turns it into
\begin{align*}
&C_{\sigma_v,\Pi_v}c_vc_{\sigma_v,\psi_{v,q(x_{v})}}\int_{U(F_v)}\int_{K_v\times K_v}\int_{F_v^\times\times F_v^\times}\overline{W_{\varphi_{1,v},\psi_{v,q(x_{v})}}(\underline{b_{v}}k_v^\prime)}W_{\varphi_{2,v},\psi_{v,q(x_{v})}}(\underline{a_{v}}k_{v})\cdot\\
&\quad \cdot \omega_v(\underline{b_{v}}k^\prime_{v}))\widehat{\phi_{1,v}}(x_{v};0,1)\overline{\omega_v(\underline{a_{v}}k_{v})\widehat{\phi_{2,v}}(x_{v};0,1)}|b_{v}a_{v}|^{-2} f(x_{v})|q(x_v)|^{-1}d^\ast b_{v} d^\ast a_{v} dk_{v}d k^\prime_{v}dx_{v}.
\end{align*}
This is equal to the RHS of equation (\ref{i_d}) with $W_{f_{1,v},f_{2,v},\psi_v}(x_v)$ given in Proposition ~\ref{prop_ftdistribution}.
\end{proof}

\subsubsection{The regularized local period integral}\label{rlp_sk}

The $h_v$-function $\int_{U(F_v)}(\Pi_v(u_vh_v)f_{1,v},f_{2,v})\psi_{v_0}(u_v)du_v$ is trivial on $\SO(V_0)_{F_v}$ by Proposition ~\ref{prop_ftdistribution}, whence
\begin{align}\label{lrp_sk}
\nonumber \mpp_v^\sharp(f_{1,v},f_{2,v})=&\frac{\zeta_{E_v}(1)L(\frac{1}{2},\pi_v)L(\frac{3}{2},\pi_v)L(1,\pi_v,\ad)}{\zeta_{F_v}(4)L(1,\chi_v)L(\frac{1}{2},\pi_v\boxtimes \chi_v)}\cdot \frac{\zeta_{F_v}(s)\mpp(s,\chi_v)}{L(s,\chi_v)}\big|_{s=0}\\
&\quad \cdot C_{\sigma_v,\Pi_v}c_{\sigma_v,\psi_{v,q(v_0)}}c_vI_v(\phi_{1,v},\varphi_{1,v})\overline{I_v(\phi_{2,v},\varphi_{2,v})}.
\end{align}
Here $\mpp(s,\chi_v)=\int_{\SO(V_0)_{F_v}}\chi_v(h_v)\Delta(h_v)^s \dd h_v$ is as defined in Section ~\ref{s_so2i}. We observe that $\frac{\zeta_{F_v}(s)\mpp(s,\chi_v)}{L(s,\chi_v)}=\frac{\zeta_{F_v}(s)}{\zeta_{F_v}(2s)}\cdot\mpp^\sharp(s,\chi_v)$ is holomorphic at $s=0$.

\begin{proposition}\label{lcriterion}
The following statements are equivalent:
\begin{itemize}
\item[(i)] $\Hom_{R_v(F_v)}(\Pi_v\otimes \big(\psi_{v,v_0}\boxtimes \chi_v\big),\bc)\neq 0$,
\item[(ii)] $\chi_v=1$ and $\sigma_v$ has a non-zero $\psi_{v,q(v_0)}$-Whittaker functional,
\item[(iii)] $\mathrm{ord}_{s=0}L(s,\Pi_v,\chi_v)\leq 0$ and $\mpp^\sharp_v$ is non-vanishing on $\Pi_v\otimes \Pi_v$.
\end{itemize}
\end{proposition}
\begin{proof}
The equivalence of (i) and (ii) follows from \cite[Thm. 9.1]{gan_gurevich09}. We now argue for the equivalence of (ii) and (iii). By the expression of $L(s,\Pi,\chi)$ in (\ref{lf_skp}), there is $\mathrm{ord}_{s=0}L(s,\Pi_v,\chi_v)=\mathrm{ord}_{s=0}\frac{L(s,\chi_v)}{\zeta_{F_v}(2s)}$, whence $\mathrm{ord}_{s=0}L(s,\Pi_v,\chi_v)\leq 0$ if and only $\chi_v=1$. With equation (\ref{lrp_sk}) and Lemma ~\ref{integralsov0}, one sees that $P^\sharp$ is nonzero if and only $\chi_v=1$, $c_{\sigma_v,\psi_{v,q(v_0)}}\neq 0$, and $I_v$ is non-zero on $\ms(V(F_v))\otimes \sigma_v$. Because $c_{\sigma_v,\psi_{v,q(v_0)}}\neq 0$ if and only if $\sigma_v$ has a non-zero $\psi_{v,q(v_0)}$-Whittaker function, it suffices to show $I_v$ is nonzero on $\ms(V(F_v))\otimes \sigma_v$ when $\sigma_v$ has a non-zero $\psi_{v,q(v_0)}$-Whittaker functional.

Recall the decomposition $V(F_v)=F_vv_+\oplus U(F_v)\oplus F_vv_-$. Consider $\phi_v\in \ms(V(F_v))$ of the shape $\phi_v=\phi_v^\prime\otimes \phi^\pprime_{v}$, with $\phi_v^\prime\in \ms(U(F_v))$ and $\phi^\pprime_{v}\in \ms(F_vv_+\oplus F_vv_-)$. There is
\[
I_v(\phi_v,\varphi_v)=\int_{N(F_v)\backslash \SL_2(F_v)}\overline{W_{\varphi_{v},\psi_{q(v_0)}}(g_v)}\omega_v(g_v){\phi}^\prime_v(v_0)\widehat{\phi^\pprime_v}((0,1)g_v)dg_v.
\]
Identify $N(F_v)\backslash \SL_2(F_v)$ with $\{(y_v,x_v)\in F_v^2: y_v\neq 0\}$ via the map $\beta: \dot{g}_v\rar (0,1)\smalltwomatrix{}{-1}{1}{}g_v$ and recall $d\dot{g}_v=\frac{1}{c_v}dy_vdx_v$ (c.f. Section ~\ref{ss_ui} (iii)), then
\begin{equation}\label{iintegral_sk}
I_v(\phi_v,\varphi_v)=\frac{1}{c_v}\int_{F_v^\times \times F_v}\overline{W_{\varphi_{v},\psi_{q(v_0)}}(\beta^{-1}(y_v,x_v))}\omega_v(\beta^{-1}(y_v,x_v))\phi^\prime_v(v_0)\widehat{\phi^\pprime_v}(y_v,x_v)dy_vdx_v.
\end{equation}
Because $\sigma_v$ has a non-zero $\psi_{v,q(v_0)}$-Whittaker functional, the function $W_{\varphi_{v},\psi_{q(v_0)}}(\beta^{-1}(y_v,x_v))$ is nonzero. One can choose $\phi_v^\prime\in \ms(U(F_v))$ so that $\overline{W_{\varphi_{v},\psi_{q(v_0)}}(\beta^{-1}(y_v,x_v))}\omega_v(\beta^{-1}(y_v,x_v))\phi^\prime_v(v_0)$ is not identically zero, then there exists a suitable $\widehat{\phi^\pprime_v}\in \ms(F_v^2)$ so that the integral in (\ref{iintegral_sk}) is nonzero. This shows that $I_v$ is nonzero on $\sigma_v\otimes \ms(V(F_v))$.
\end{proof}

\begin{theorem}\label{lc_skp}
When $\Pi\in \SK(\pi)$ is cuspidal, Conjecture ~\ref{conj_lbp} holds for $\Pi_v$.
\end{theorem}
\begin{proof}
Part (i) of Conjecture ~\ref{conj_lbp} follows from Lemma ~\ref{integralsov0}. When $\mathrm{ord}_{s=0}L(s,\Pi_v,\chi_v)\leq 0$, there is $\chi_v=1$, whence $\mpp_v^\sharp(f_{1,v},f_{2,v})$ is a nonzero scalar multiple of $c_{\sigma_v,\psi_{v,q(v_0)}}I_v(\phi_{1,v},\varphi_{1,v})\overline{I_v(\phi_{2,v},\varphi_{2,v})}$ for $f_{i,v}=\theta(\phi_{i,v},\varphi_{i,v})$ and the invariance property in Part (ii) of Conjecture ~\ref{conj_lbp} obviously holds. Part (iii) of Conjecture ~\ref{conj_lbp} is exactly Proposition ~\ref{lcriterion}.
\end{proof}

\subsection{The Global formula}

Suppose that $\Pi\in \SK(\pi)$ is cuspidal with $\sigma=\Theta_{\wtilde{\SL}_2\times \SO(V)}(\Pi,\psi)$.

\begin{proposition}\label{gcriterion}
$P$ is non-vanishing on $\Pi\boxtimes \chi$ if and only if \emph{(ia)} $\Hom_{\mrr(F_v)}(\Pi_v\otimes \big(\psi_{v,v_0}\boxtimes \chi_v\big),\bc)\neq 0$ for all places $v$ and \emph{(ib)} $L(0,\Pi\otimes \chi)\neq 0$. Equivalently, the conditions can be stated as \emph{(iia)} $q(v_0)\not\in {F^\times}^2$, \emph{(iib)} $\chi=1$, \emph{(iic)} $L(\frac{1}{2},\pi\otimes \chi_{E/F})\neq 0$ and \emph{(iid)} each $\sigma_v$ has a non-zero $\psi_{v,q(v_0)}$-Whittaker functional. As a consequence, there exists at most one cuspidal member $\Pi\in \SK(\pi)$ such that $P$ is nonzero on $\Pi$
\end{proposition}
\begin{proof}
We explain the equivalance of the conditions. By \cite[Thm. 9.1]{gan_gurevich09}, (ia) is equivalent to (iib)$+$ (iid). Supposing (iib) and (iid), we show that (ib) is equivalent to (iia)$+$(iic). Actually, when $\chi=1$, (ib) means $L(\frac{1}{2},\pi\otimes \chi_{q(v_0)})L(0,\chi_{q(v_0)})\neq 0$. So (iia)$+$(iic) implies (ib). For the inverse direction, it suffices to prove (iia)$\Rightarrow$(ib). If (iia) does not hold, then $U$ is split and $L(s,\chi_{q(v_0)})=\zeta_F(s)$ has a simple pole at $s=0$. In this situation, (ib) forces that $\ord_{s=1/2}L(s,\pi\otimes \chi_{q(v_0)})=0$ or $1$. Because (iid) implies $\epsilon(\frac{1}{2},\pi\otimes \chi_{q(v_0)})=1$, there is $\ord_{s=1/2}L(s,\pi\otimes \chi_{q(v_0))})\neq 1$, whence $L(\frac{1}{2},\pi\otimes \chi_{q(v_0)})\neq 0$. Combining this with (iid), we see that that $\sigma$ has non-zero $\psi$-Whittaker coefficient and hence has non-zero theta lifting to $\PGL_2(\ba)\cong \SO(U)_\ba$, a contradiction to the assumption that $\Pi$ is cuspidal.

Now suppose that $P$ is non-vanishing on $\Pi$. By Proposition ~\ref{gperiod},  we have (iia), (iib) and that $\sigma$ has non-vanishing $\psi_{q(v_0)}$-Whittaker functional. The latter impliees (iic) and (iid).

Conversely, supposing (iia), (iib), (iic), and (iid), then Proposition ~\ref{gperiod} tells that $P(F)=2\prod_v I_v(\phi_v,\varphi_v)$ for $F=\Theta(\phi,\varphi)$. As argued in the proof of Proposition ~\ref{lcriterion}, $I_v(-)$ is non-vanishing on $\ms(V(F_v))\otimes \sigma_v$, whence $P$ is non-vanishing on $\Pi$.
\end{proof}
\begin{remark}
With Conditions (iia), (iib), (iic), the only member of $Wd_\psi(\sigma)$ satisfying Considtion (iid) is  $\sigma=\Theta_{\wtilde{\SL}_2\times \PGL_2}(\pi\otimes \chi_{q(v_0)},\psi_{q(v_0)})$. So $P$ is nonzero on certain cuspidal member of $\SK(\pi)$ if and only if $\Theta_{\wtilde{\SL}_2\times \SO(V)}(\sigma,\psi)\neq 0$ and $\Theta_{\wtilde{\SL}_2\times \SO(U)}(\sigma,\psi)=0$.
\end{remark}

\begin{theorem}\label{gformula}
Suppose that $\Pi\in \SK(\pi)$ is cuspidal and that $q(v_0)\not\in {F^\times}^2$, $\chi=1$, then
\[
P\otimes \overline{P}=\zeta_F(2)\zeta_F(4)L(0,\Pi,\chi)\prod_v 2^{-1}\mpp_v^\sharp.
\]
In other words, for decomposable $f_1, f_2\in \Pi$,  there is $\mpp_v^\sharp(f_{1,v},f_{2,v})=2$ for almost all $v$ and
\begin{equation}\label{eq_skp}
P(f_1)\overline{P(f_2)}=\frac{2\zeta_F(4)L(\frac{1}{2},\pi\otimes \chi_{E/F})L(0,\chi_{E/F})}{L(\frac{3}{2},\pi)L(1,\pi,\ad)}\prod_v \mpp^\sharp_{v}(f_{1,v},f_{2,v}),
\end{equation}
\end{theorem}
\begin{proof}
Write $f_i=\Theta(\phi_i,\varphi_i)$. When $\chi=1$, there is
\[
\mpp^\sharp_{v}(f_{1,v},f_{2,v})=\frac{2C_{\sigma_v,\Pi_v}c_{\sigma_v,\psi_{v,q(v_0)}}c_v c_{E_v}^\prime L(\frac{3}{2},\pi_v)L(1,\pi_v,\ad)}{\zeta_{F_v}(4)L(\frac{1}{2},\pi_v\otimes \chi_{E_v/F_v})}\cdot I_v(\phi_{1,v},\varphi_{1,v})\overline{I_v(\phi_{1,v},\varphi_{1,v})}.
\]

For almost all places, one has $I_v(\phi_{i,v},\varphi_{i,v})=1$ and
\[
C_{\sigma_v,\Pi_v}=\frac{\zeta_{F_v}(4)}{L(\frac{3}{2},\pi_v)},\, c_{\sigma_v,\psi_{v,q(v_0)}}=\frac{L(\frac{1}{2},\pi_v\otimes \chi_{E_v/F_v})\zeta_{F_v}(2)}{L(1,\pi_v,\ad)},\, c_v=\frac{1}{\zeta_{F_v}(2)},\, c_{E_v}^\prime=1,
\]
whence $\mpp_v^\sharp(f_{1,v},f_{2,v})=2$ for almost all places.

Furthermore, because $\prod_v \frac{C_{\sigma_v,\Pi_v}L(\frac{3}{2},\pi_v)}{\zeta_{F_v}(4)}=\frac{2L(\frac{3}{2},\pi)}{\zeta_F(4)}$, $\prod_v \frac{c_{\sigma_v,\psi_{v,q(v_0)}}L(1,\pi_v,\ad)}{L(\frac{1}{2},\pi_v\otimes_{\chi_{E_v/F_v}})\zeta_{F_v}(2)}=2\cdot \frac{L(1,\pi,\ad)}{L(\frac{1}{2},\pi\otimes \chi_{E/F})\zeta_F(2)}$, $\prod_v c_v=1$, and $\prod_v c_{E_v}^\prime=\frac{1}{L(1,\chi_{E/F})}$, we have
\begin{equation}\label{eq_skp1}
\prod_v 2^{-1}\mpp_v^\sharp(f_{1,v},f_{2,v})=\frac{4L(\frac{3}{2},\pi)L(1,\pi,\ad)}{\zeta_F(4)L(\frac{1}{2},\pi\otimes \chi_{E/F})L(1,\chi_{E/F})}\prod_v I_v(\phi_{1,v},\varphi_{1,v})\overline{I_v(\phi_{1,v},\varphi_{1,v})}.
\end{equation}
On the other hand, Proposition ~\ref{gperiod} tells that
\begin{equation}\label{eq_skp2}
P(F_1)\overline{P(F_2)}=4\prod_v I_v(\phi_{1,v},\varphi_{1,v})\overline{I_v(\phi_{1,v},\varphi_{1,v})}.
\end{equation}
Comparing equations (\ref{eq_skp1}), (\ref{eq_skp2}) and noticing $L(1,\chi_{E/F})=L(0,\chi_{E/F})$, we obtain (\ref{eq_skp}).
\end{proof}

From Theorem ~\ref{lc_skp}, Proposition ~\ref{gcriterion} and Theorem ~\ref{gformula}, one sees easily that the main theorem in the introduction holds for a cuspidal member $\Pi$ of a general Saito-Kurokawa packet $\SK(\pi,\mu)=\SK(\pi\otimes \mu)\otimes \mu$.

\section{Soudry Packets on $\SO(3,2)$}

In this section, we prove the main theorem for Soudry packets on $\SO(3,2)$. Recall the explicit quadratic space $(V,q)$ of Witt index $2$ and determinant $1\cdot {F^\times}^2$ in Section ~\ref{so32} and the isomorhism $\PGSp_4\cong \SO(V)$. We adopt the following notations:
\begin{itemize}
\item[(i)] Write $V=Fv_+\oplus U\oplus Fv_-$, where
\[
v_+=\smalltwomatrix{0}{J_1}{0}{0},\quad v_-=\smalltwomatrix{0}{0}{-J_1}{0},\quad U=\{X\in M_{2\times 2}(F): \Tr(X)=0\}
\]
$U$ is embedded into $V$ by $X\rar \smalltwomatrix{X}{}{}{\leftup{t}{X}}$ and $q(x^\prime  v_++X+x^\pprime v_-)=x^\prime x^\pprime-\det X$.

\item[(ii)] Let $\mrp$ denote the Siegel parabolic subgroup of $\PGSp_4$, $\mrn$ its unipotent radical, and $\mrm$ its Levi subgroup consisting of block diagonal matrices. Recall from Section ~\ref{bgroup} the parabolic subgroup $P=U\rtimes M$ of $\SO(V)$ that stabilizes the line $Fv_+$. With respect to $\PGSp_4\cong \SO(V)$, $\smalltwomatrix{1_2}{N}{}{1_2}\in \mathrm{N}$ corresponds to $u=-NJ_1\in U$ and $[\smalltwomatrix{A}{}{}{x\leftup{t}{A^{-1}}}]\in \mathrm{M}$ corresponds to the following element $m(A,x)\in M$
\[
v_+\rar x^{-1}\det A\cdot v_+,\quad   X\in U\rar AXA^{-1}\in U,\quad  v_-\rar x(\det A)^{-1}\cdot v_-.
\]

\item[(iii)] Identify $\mrn$ with $\Sym_2$. For $T\in \Sym_2^\circ(F)$, write $v_0=J_1T\in U^\circ(F)$. The character $\psi_T(N):=\psi(\Tr(TN))$ on $[\mathrm{N}]$ corresponds to the character $\psi_{v_0}(u)$ on $[U]$.

\item[(iv)] Write $U=Fv_0\oplus V_0$. $\SO(V_0)$ is identified with $\bt$ via the map $[A]\rar [\smalltwomatrix{A}{}{}{(\det A)\leftup{t}{A^{-1}}}]$,
\begin{align}\label{groupT}
\SO(V_0)=&\{[A]\in \PGL_2\cong \SO(U): A^{-1}v_0A=v_0\}\subset M,\\
\nonumber \bt:=&\left\{[\smalltwomatrix{A}{}{}{(\det A)\leftup{t}{A^{-1}}}]: \leftup{t}{A}TA=(\det A) T\right\}\subset \mrm
\end{align}
When $T=\smalltwomatrix{c}{}{}{c\delta}$ is diagonal, there is  $\SO(V_0)=\{\smalltwomatrix{a}{-b\delta}{b}{a}\in \GL_2\}/F^\times 1_2$.
\end{itemize}

The Soudry packets on $\SO(V)$ are of the form $\So(\eta)$, where $\eta$ is a character of $\ba^\times_K/K^\times$ for a quadratic extension $K/F$ and is trivial on $\ba^\times$. Consider $(V_K,q_K)=(K,\Nm_{K/F})$ and the Weil representation $\omega_{\psi}$ of $\GSp_4^+(\ba)\times \OO(V_K)_\ba$ on $\ms(V_K^2(\ba))$. %We will use $\omega_\psi$ to define packets $\So(\eta)$ on $\GSp_4$ for all $\eta$. When studying the Bessel periods, we then restrict to those packets with $\eta|_{\ba^\times}=1$ for these packets have trivial central character and descend to $\PGSp_4\cong \SO(V)$.
%\[ \omega_{\psi,V_0}(g,h)\phi(X)=|\lambda(h)|^{-1}\omega_{\psi,V_0}\big(g\smalltwomatrix{1}{}{}{\nu(g)^{-1}}\big)\phi(h^{-1}X),\quad (h,g)\in R.\]

\subsection{Packets on $\GO(V_K)$}\label{p_go2}

Write $K=F\oplus Fe$ with $\Tr(e)=0$ and $\Nm(e)=d_K$. Set $\Gal(K/F)=\{1,\iota\}$, then $\GO(V_0)=\GSO(V_0)\rtimes \mu_2$, with $\GSO(V_K)= K^\times$ and  $\mu_2=\Gal(K/F)$.

Let $S_F$ denote the set of places of $F$. For a character $\eta$ of $\ba_K^\times/K^\times$, let $S_\eta$ denote the set of places of $F$ such that $\eta_v^{\iota_v}\neq \eta_v$, where $\iota_v$ is the image of $\iota$ in $\Gal(K_v/F_v)$ and $\eta_v^{\iota_v}(x_v)=\eta_v(x_v^{\iota_v})$. $S_\eta=\emptyset$ happens only when $\eta^\iota=\eta$, that is, when $\eta|_{\ba_{K,1}^\times}=1$. Set
\[
\beta=
\begin{cases}
1, &\eta^\iota\neq \eta,\\
2. &\eta^\iota=\eta.
\end{cases},\quad 
\beta_v=
\begin{cases}
1, &\eta_v^{\iota_v}\neq \eta_v,\\
2. &\eta_v^{\iota_v}=\eta_v.
\end{cases}
\]

(i) For $v\in S_\eta$, set $\eta_v^+=\Ind_{\GSO(V_K)_{F_v}}^{\GO(V_K)_{F_v}}\eta_v$, which is irreducible. When $v\not\in S_\eta$, $\eta_v$ has two irreducible extensions to $\GO(V_K)_{F_v}$; let $\eta_v^+$ denote the extension with value $1$ at  $\iota_v$ and set $\eta_v^-=\eta_v^+ \sgn$, where $\sgn$ is the sign character factoring through $\mu_2$. Define the local packet 
\begin{equation*}
\Pa(\eta_v)=
\begin{cases}
\{\eta_v^+\}, &v\in S_\eta,\\
\{\eta_v^+, \eta_v^-\}, &v\not\in S_\eta.
\end{cases}
\end{equation*}

(ii) For $\epsilon=(\epsilon_v)$ with almost all $\epsilon_v=1$, set $\eta^\epsilon=\prod_v \eta_v^{\epsilon_v}$. Define the Eisenstein series  $E(\Phi):=\Phi(h)+\Phi(\iota h)$ for $\Phi\in \eta^\epsilon$ and set $E(\eta^\epsilon):=\{E(\Phi): \Phi\in \eta^\epsilon\}$. $E(\eta^\epsilon)$ is zero only when $S_\eta=S_F$ with $\sgn(\epsilon)=-1$. Define the global packet
\[
\Pa(\eta)=
\begin{cases}
\{\eta^\epsilon:\epsilon \}, &\eta\neq \eta^\iota,\\
\{\eta^\epsilon: \sgn(\epsilon)=1\}, &\eta=\eta^\iota.
\end{cases}
\]

(iii) $L^2(K^\times\ba^\times\backslash \ba_K^\times)$ is the Hilbert space direct sum of characters $\eta$ satisfying $\eta|_{\ba^\times}=1$. As a consequence, $L^2(\GO(V_K)_F\ba^\times\backslash \GO(V_K)_\ba)$ is the Hilbert space direct sum of $E(\eta^\epsilon)$, with $\eta^\epsilon\in \Pa(\eta)$ and $\eta|_{\ba^\times}=1$.

(iv) Given $\mu_2(\ba)$ the probability measure. For $\Phi_1,\Phi_2\in \eta^\epsilon$ and $\varphi_1,\varphi_2\in E(\eta^\epsilon)$, define
\[
(\Phi_1,\Phi_2)=\int_{\mu_2(\ba)}\Phi_1(h_0)\overline{\Phi_2(h_0)}\dd h_0,\quad (\varphi_1,\varphi_2)_{\GO(V_K)}=\int_{[\PGO(V_K)]}\varphi_1(h)\overline{\varphi_2(h)}\dd h.
\]
When $\varphi_i=E(\Phi_i)$ with $\Phi_i\in \eta^\epsilon$, there is $(\varphi_1,\varphi_2)=2^{\beta}(\Phi_1,\Phi_2)$.

\subsection{Soudry packets on $\GSp_4$}

Set $\GSp_4^+=\{g\in \GSp_4:\nu(g)\in \Nm(K^\times)\}$. Recall from (\ref{wr_s}) the expression for the Weil representation of $R(\ba):=\{(g,h)\in \GSp_4^+(\ba)\times \GO(V_K)_\ba: \nu(g)=\nu(h)\}$ on $\ms(V^2_K(\ba))$,
\begin{equation}\label{wr_s2}
\omega_\psi(g,h)\phi(X)=|\nu(h)|^{-1}\omega_{\psi}\big(g\smalltwomatrix{1}{}{}{\nu(g)^{-1}}\big)\phi(h^{-1}X),\quad (g,h)\in R(\ba).
\end{equation}
The action of $\Sp_4(\ba)\times \OO(V_K)_\ba$ is according to Section ~\ref{ltc}.

For an automorphic form $\varphi(h)$ on $[\GO(V_K)]$ and $\phi\in \ms(V_K^2(\ba))$, define
\[
\Theta_\psi(\phi,\varphi)(g)=\int_{[\OO(V_K)]}\Theta_{\phi,\psi}(g,h_1h^\prime)\overline{\varphi(h_1h^\prime)}\dd h_1,\quad g\in \GSp_4^+(\ba),
\]
where $h^\prime\in \GO(V_K)_\ba$ is an element satisfying $\nu(h^\prime)=\nu(g)$. %We also write $\Theta(\phi,\varphi)$ for $\Theta_\phi(\varphi)$.

For $\eta^\epsilon\in \Pa(\eta)$, define the global theta lift of $E(\eta^\epsilon)$ with respect to $\psi$,
\[
\Theta_\psi(\eta^\epsilon)=\Theta(\eta^\epsilon,\psi):=\{\Theta_\psi(\phi,\varphi): \phi\in \ms(V_K^2(\ba)), \varphi\in E(\eta^\epsilon)\}.
\]
It has central character $\eta|_{\ba^\times}$ and is a nonzero irreducible $\GSp_4^+(\ba)$-representation in the discrete spectrum of $\GSp_4^+$. There is $\Theta_\psi(\eta^\epsilon)\cong\otimes_v \theta(\eta_v^{\epsilon_v},\psi_v)$ by local Howe duality and it is cuspidal unless $S_\eta=S_F$ and $\epsilon=(1)$.

We view functions in $\Theta_\psi(\eta^\epsilon)$ as functions on $[\GSp_4]$ according to two observations:
\begin{itemize}
\item[(i)] $\GSp_4^\circ(\ba):=\{g\in \GSp_4(\ba):\nu(g)\in \Nm_{K/F}(\ba_K^\times)\cdot F^\times\}$ is open and of index $2$ in $\GSp_4(\ba)$. Write $\GSp_4(\ba)=\GSp_4^\circ(\ba)\sqcup \GSp_4^\circ(\ba)g_\iota$. A function $F$ on $\GSp_4^\circ(\ba)$ is regarded as a function on $\GSp_4(\ba)$ by setting $F|_{\GSp_4^\circ(\ba)g_\iota}=0$.
.
\item[(ii)] $\GSp_4(F)\backslash \GSp_4^\circ(\ba)=\GSp_4^+(F)\backslash \GSp_4^+(\ba)$, whence a function on $[\GSp_4^+]$ is naturally a function on $\GSp_4(F)\backslash \GSp_4^\circ(\ba)$.
\end{itemize}
Let $\Pi_\eta^\epsilon$ be the $\GSp_4(\ba)$-representation generated by functions $f\in \Theta_\psi(\eta^\epsilon)$. It does not depend on the choice of $\psi$, following Lemma ~\ref{l_thetacom} and ~\ref{l_decomp} below. Define the global packet
\[
\So(\eta)=\{\Pi^\epsilon_\eta: \eta^\epsilon\in \Pa(\eta)\}.
\]
%\begin{equation*} \begin{cases} \So(\eta):=\{\Pi_\eta^\epsilon: \epsilon\}, &\eta|_{\ba_{K,1}^\times}\neq 1,\\ \HPS(\eta):=\{\Pi_\eta^\epsilon: \sgn \epsilon=1\}, &\eta|_{\ba_{K,1}^\times}= 1.\end{cases}\end{equation*}
When $\eta|_{\ba^\times}=1$, $\Pi^\epsilon_\eta$ has trivial central character and the packet $\So(\eta)$ is on $\PGSp_4\cong \SO(V)$.

\begin{remark}
When $\eta=\eta^\iota$, $\eta$ is trivial on $\ba^\times_{K,1}$ and it leads to a character of $\Nm(\ba_E^\times)/\Nm(E^\times)F^\times$, which is of index $2$ in $\ba^\times/F^\times$. So there exists a unique couple $(\chi_1,\chi_2)$ of characters on $\ba^\times/F^\times$ with $\chi_1\chi_2=\chi_{E/F}$ such that $\eta=\chi_i\circ \Nm$. The packet $\So(\eta)$ is the Howe-Piatetski-Shapiro packet $\HPS(\chi_1,\chi_2)$ that consists of irreducible representations $\Pi$ in the discrete spectrum with $L^S(\Pi)=L^S(s+\frac{1}{2},\chi_1)L^S(s+\frac{1}{2},\chi_2)L^S(s-\frac{1}{2},\chi_1)L^S(s-\frac{1}{2},\chi_2)$. There is $\HPS(\chi_1,\chi_2)=\HPS(1,\chi_{E/F})\otimes \chi_1$ and $\HPS(1,\chi_{E/F})$ can be alternatively constructed by lifting elementary theta series from $\wtilde{\SL}_2(\ba)$ to $\SO(V)_\ba$ (c.f. \cite{PS83}).

\end{remark}

Now we discuss the structure of $\Pi^\epsilon_\eta$. First is the local aspect. At a nonsplit place $v$ of $E/F$, $\GSp_4^+(F_v)$ is of index $2$ in $\GSp_4(F_v)$ and we comment that $\Ind_{\GSp_4^+(F_v)}^{\GSp_4(F_v)}\theta(\eta_v^{\epsilon_v},\psi_v)$ is irreducible. Actually, choose $a_v\in F_v^\times- \Nm(K_v^\times)$ and $g_{a_v}\in \GSp_4(F_v)$ with $\nu(g_{a_v})=a_v$, one can show that the local theta lift $\theta(\eta_v^{\epsilon_v},\psi_v)$ and its twist $\theta(\eta_v^{\epsilon_v},\psi_v)^{g_{a_v}}$ are non-isomorphic. This is because the local theta lift of $\theta(\eta_v^{\epsilon_v},\psi_v)^{g_{a_v}}$to $\GO(V_K(F_v),a_v q)$ with respect to $\psi_v$ is $\eta_v^{\epsilon_v}$, while the local theta lift of $\theta(\eta_v^{\epsilon_v},\psi_v)$ to $\GO(V_0(F_v),a_v q)$ is zero by the conservation principle.

Second is the global aspect. For a space $\mh$ of functions on $\GSp_4(\ba)$ and $g^\prime\in \GSp_4(\ba)$, define $\mh^{g^\prime}:=\{f(\cdot g^\prime): f\in \mh\}$. Let $\Pi_\eta^{\epsilon^\circ}$ be the the $\GSp_4^\circ(\ba)$-representation spanned by functions in $\Theta(\eta^\epsilon)$. Observe that $\GSp_4^+(\ba)\backslash \GSp_4^\circ (\ba)= F^\times/\Nm(K^\times)$.
\begin{lemma}\label{l_thetacom}
For each $a\in F^\times/\Nm(K^\times)$, choose $g_a\in \GSp_4(F)$ with $\nu(g_a)=a$, then
\begin{equation}\label{thetacom}
\Theta_\psi(\eta^\epsilon)^{g_a}=\Theta_{\psi_a}(\eta^\epsilon).
\end{equation}
\end{lemma}
\begin{proof}
We do a sketch of the proof.  First, suppose $\Theta_\psi(\eta^\epsilon)$ is cuspidal, then $\Theta_\psi(\eta^\epsilon)$is also cuspidal. Because the local representations $\theta(\eta_v^{\epsilon_v},\psi_v)^{g_a}$, $\theta(\eta_v^{\epsilon_v},\psi_{v,a})$ are isomorphic and $\Theta_{\psi_a}(\Theta_{\psi_a}(\eta^\epsilon))=E(\eta^\epsilon)$, one knows from inner product formula and the nonvanishing of local zeta integrals that $\Theta_{\psi_a}(\Theta_\psi(\eta^\epsilon)^{g_a})$ is nonzero. There must be $\Theta_{\psi_a}(\Theta_\psi(\eta^\epsilon)^{g_a})=E(\eta^\epsilon)$ because their local representations are isomorphic by local Howe duality and there is multiplicity one on $\GO(V_K)$. Therefore, (\ref{thetacom}) holds.

Second, when $\Theta_\psi(\eta^\epsilon)$ is noncuspidal, then it belongs to the case $\eta^\iota=\eta$ with $\epsilon=(1)$. $\Theta_\psi(\eta^\epsilon)$ is the residue representation of an explicit Borel-type Eisenstein series. Both $\Theta_\psi(\eta^\epsilon)^{g_a}$ and $\Theta_{\psi_a}(\eta^\epsilon)$ can be seen explicitly and compared to be equal.
\end{proof}

\begin{lemma}\label{l_decomp}
\begin{itemize}
\item[(i)] $\Pi_\eta^{\epsilon^\circ}= \oplus_{a\in F^\times/\Nm(E^\times)} \Theta_{\psi_a}(\eta^\epsilon)$ as a $\GSp_4^+(\ba)$-representation.
\item[(ii)] $\Pi^\epsilon_\eta=\Ind_{\GSp_4^\circ(\ba)}^{\GSp_4(\ba)}\Pi_\eta^{\epsilon^\circ}=\otimes_v \Pi_{\eta_v}^{\epsilon_v}$ with $\Pi_{\eta_v}^{\epsilon_v}=\Ind_{\GSp_4^+(F_v)}^{\GSp_4(F_v)} \theta(\eta_v^{\epsilon_v},\psi_v)$. $\Pi^\epsilon_\eta$ is irreducible.
\end{itemize}
\end{lemma}
\begin{proof}
First, $\Pi_\eta^{\epsilon^\circ}= \sum_{a\in \Nm(E^\times)\backslash F^\times} \Theta(\eta^\epsilon)^{g_a}$. It is a direct sum because $\Theta(\eta^\epsilon)^{g_a}$ are distinctive $\GSp_4^+(\ba)$-representations. Actually, if $a_1\neq a_2$ in $\Nm(E^\times)\backslash F^\times$, then there is a nonsplit place $v$ of $E/F$ such that $a_1\neq a_2$ in $F_v^\times/\Nm(E_v^\times)$, whence the local representations $\theta(\eta_v^{\epsilon_v},\psi_v)^{g_{a_1}}$ and $\theta(\eta_v^{\epsilon_v},\psi_v)^{g_{a_2}}$ are distinctive. This proves (i).

Recall that $\GSp_4(\ba)=\GSp_4^\circ(\ba)\sqcup \GSp_4^\circ(\ba)g_\iota$. The same argument as for (i) shows that
\begin{align*}
\Pi^\epsilon_\eta=&\big(\oplus_{a\in F^\times/\Nm(K^\times)} \Theta(\eta^\epsilon)^{g_a}\big)\oplus \big(\oplus_{a\in F^\times/\Nm(K^\times)} \Theta(\eta^\epsilon)^{g_\iota g_a}\big)=\Pi_\eta^{\epsilon^\circ}\oplus (\Pi_\eta^{\epsilon^\circ})^{g_\iota}.
\end{align*}
So $\Pi^\epsilon_\eta=\Ind_{\GSp_4^\circ(\ba)}^{\GSp_4(\ba)}\Pi_\eta^{\epsilon^\circ}$. Combining this with (i), one sees that $\Pi^\epsilon_\eta\cong \otimes_v \Ind_{\GSp_4^+(F_v)}^{\GSp_4(F_v)} \theta(\eta_v^{\epsilon_v},\psi_v)$ because the elements $g_a, g_ag_\iota$ form a complete set of representatives of $\GSp_4^+(\ba)\backslash \GSp_4(\ba)$.  $\Pi^\epsilon_\eta$ is irreducible because each local representation $\Ind_{\GSp_4^+(F_v)}^{\GSp_4(F_v)} \theta(\eta_v^{\epsilon_v},\psi_v)$ is irreducible.
\end{proof}

Third is about the inner product formula. We present the inner product formula for the lift $E(\eta^\epsilon)\rar \Pi^\epsilon_\eta$. The inner product formula for the inverse direction is not used in the computation of Bessel periods and skipped. For irreducible unitary representations $\Pi$ (resp. $\Pi^+$) of $\GSp_4(\ba)$ (resp. $\GSp_4^+(\ba)$) in the discrete spectrum and $f_i\in \Pi, f_i^+\in \Pi^+$, define

\[
(f_1,f_2)_{\GSp_4}=\int_{[\PGSp_4]}f_1(g)\overline{f_2(g)}\dd g,\quad (f_1^+,f_2^+)_{\GSp_4^+}=\int_{[\PGSp_4^+]}f_1(g)\overline{f_2(g)}\dd g,
\]

%Let $\eta\circ \pr$ be the character $\ba^\times/K^\times\overset{\pr}{\rar}\ba_{K,1}^\times/K_1^\times \hrar \ba_K^\times/K^\times\overset{\eta}{\rar} \bc^\times$.
\begin{proposition}\label{ipf_sp}
Suppose $\eta^\epsilon\in \Pa(\eta)$ and $\Pi^\epsilon_\eta$ is cuspidal. For $\varphi_i\in E(\eta^\epsilon)$, $\phi_i\in \ms(V_K^2(\ba))$ and $f_i=\Theta_\psi(\phi_i,\varphi_i)$, there is
\[
(f_1,f_2)_{\GSp_4}=\frac{L(2, \eta\circ \pr)}{\zeta_F(4)}\prod_v \frac{\zeta_{F_v}(4)}{L(2,\eta_v\circ \pr)}\underset{\OO(V_K)_{F_v}}{\int} \overline{(h_v\circ \varphi_{1,v},\varphi_{2,v})}(\omega_{\psi_v}(h_{v})\phi_{1,v},\phi_{2,v})\dd h_v.
\]
\end{proposition}
\begin{proof}
$f_1, f_2$ are functions on $[\GSp_4^+]$. They are supported on $\GSp_4^\circ(\ba)$ when regarded as functions on $\GSp_4(\ba)$, whence $(f_1,f_2)_{\GSp_4}=(f_1,f_2)_{\GSp_4^+}$. Equip the compact space $\mc=\Nm(\ba_K^\times)/{\ba^\times}^2\Nm(K^\times)$ with the probability measure $1$, then
\begin{align}\label{gpairing}
&(f_1,f_2)_{\GSp_4^+}=\int_{\mc}\int_{[\Sp_4]}f_1(g_1g_c)\overline{f_2(g_1g_c)}\dd g_1\dd c,\\
\nonumber &(\varphi_1,\varphi_2)_{\GO(V_K)}=\int_{\mc}\int_{[\OO(V_K)]}\varphi_1(h_1h_c)\overline{\varphi_2(h_1h_c)}\dd h_1\dd c,
\end{align}
where $g_c\in \GSp_4^+(\ba)$ and $h_c\in O(V_K)_\ba$ satisfy $\nu(h_c)=\nu(g_c)=c$. Write $\phi_{i,c}=\omega(h_c,g_c)\phi_i$, $\varphi_{i,c}(h_1)=\varphi_i(h_1 h_c)$, and $f_{i,c}=\Theta(\phi_{i,c},\varphi_{i,c})$, then $f_i(g_1g_c)=f_{i,c}(g_1)$, whence $(f_1,f_2)_{\GSp_4^+}=\int_{\mc}(f_{1,c},f_{2,c})_{\Sp_4}\dd c$ and $(\varphi_1,\varphi_2)_{\GO(V_K)}=\int_{\mc}(\varphi_{1,c},\varphi_{2,c})_{\OO(V_K)}\dd c$. 

By \cite[Thm. 11.4]{gqt14}, 
\begin{align*}
(f_{1,c},f_{2,c})_{\Sp_4}=\int_{\OO(V_K)_\ba}\overline{(h_1\circ \varphi_{1,c},\varphi_{2,c})_{\OO(V_K)}}(\omega(h_1)\phi_{1,c},\phi_{2,c})\dd h_1.
\end{align*}
Note that the above integral is absolutely convergent. Also note that $(\omega(h)\phi_{1,c},\phi_{2,c})=(\omega(h)\phi_1,\phi_2)$ because the pairing on $\ms(V_K^2(\ba))$ is unitary with respect to $\omega$. Therefore,
\begin{align*}
\big(f_1,f_2)_{\GSp_4^+}
=&\int_{\mc}\int_{\OO(V_K)_\ba}\overline{(h_1\circ \varphi_{1,c},\varphi_{2,c})_{\OO(V_K)}}(\omega(h_1)\phi_{1},\phi_{2})\dd h_1\dd c\\
=&\int_{\OO(V_K)_\ba}\big[\overline{\int_{\mc}(h_1\circ \varphi_{1,c},\varphi_{2,c})_{\OO(V_K)}\dd c}\big]\ (\omega(h_1)\phi_{1},\phi_{2})\dd h_1\\
=&\int_{\OO(V_K)_\ba}\overline{(h_1\circ \varphi_1,\varphi_2)_{\GO(V_K)}}(\omega(h_1)\phi_1,\phi_2)\dd h_1.
\end{align*}
At almost all places, there is $\int_{\OO(V)_{F_v}}\overline{(h_{1,v}\circ \varphi_{1,v},\varphi_{2,v})}(\omega_v(h_{1,v})\phi_{1,v},\phi_{2,v})dh_{1,v}=\frac{L(2, \eta_v\circ \pr)}{\zeta_{F_v}(4)}$, whence the formula in the proposition follows.
\end{proof}

\begin{remark}\label{lipf_sp}
Choose local theta liftings $\theta_v:\ms(V_K^2(F_v))\otimes \eta_v^{\epsilon_v}\rar \theta(\eta_v^{\epsilon_v},\psi_v)$ such that $\Theta_\psi(\phi,E(\Phi))=\otimes \theta_v(\phi_v,\Phi_v)$. Proposition \ref{ipf_sp} combined with Section ~\ref{p_go2} (iv) then implies that for $f_{i,v}=\theta_{\psi_v}(\phi_v,\Phi_v)$, there are
\[
(f_{1,v},f_{2,v})=c_{\eta_v}\underset{\OO(V_K)_{F_v}}{\int} \overline{(h_v\circ \Phi_{1,v},\Phi_{2,v})}(\omega_{\psi_v}(h_{v})\phi_{1,v},\phi_{2,v})\dd h_v,
\]
where $c_{\eta_v}$ are equal to $\frac{\zeta_{F_v}(4)}{L(2,\eta_v\circ \pr)}$ for almost all $v$ with $\prod_v \frac{c_{\eta_v} L(2,\eta_v\circ \pr)}{\zeta_{F_v}(4)}=\frac{2^{\beta}L(2,\eta\circ \pr)}{\zeta_F(4)}$.
\end{remark}

\subsection{Global Bessel periods}

Suppose $\eta|_{\ba^\times}=1$ and $\Pi^\epsilon_\eta\in \So(\eta)$ is cuspidal. We now compute the global Bessel period functional $P$ on $\Pi^\epsilon_\eta$. In terms of $\PGSp_4$, there is $P(f)=\int_{[\SO(V_0)]}\ell_T(\Pi^\epsilon_\eta(g)f)\dd g$, where $T=-J_1v_0$ and 
\[
\ell_T(f):=\int_{[\mathrm{N}]}f(N)\psi_T(N)dN.%\quad f\in \Pi^\epsilon_\eta.
\]

For $\Phi\in \eta^\epsilon,\phi\in \ms(V_E^2(\ba))$ and $\xi\in V_K^{2,\circ}(F)$, we introduce
\[
P(\phi,\Phi,\xi):=\int_{\OO(V_K)_\ba}\overline{\Phi(h)}\phi(h^{-1}\xi)\dd h,\quad P_{v}(\phi_v,\Phi_v,\xi):=\int_{\OO(V_K)_{F_v}}\overline{\Phi_v(h_v)}\phi_v(h_v^{-1}\xi)\dd h_v.
\]

To compute the global Bessel period functional, we make two observations: first, $\ell_T$ is zero on $\big(\Pi_\eta^{\epsilon^\circ}\big)^{g_\iota}$ because functions in $\big(\Pi_\eta^{\epsilon^\circ}\big)^{g_\iota}$ are supported on $\GSp_4^\circ(\ba)g_\iota$; second, by (\ref{sm}), there is a unique element $a_T\in F^\times/\Nm(E^\times)$ with $T\in -\frac{a_T}{2}q_E(V_E^{2,\circ}(E))$ because $\det T=-q(v_0)$. Let $\delta(\chi,\eta)$ be the Dirac symbol, which takes value $1$ when $\chi=\eta$ and value $0$ when $\chi\neq \eta$.

\begin{proposition}\label{functionalT}
\begin{itemize}
\item[(i)] When $E\neq K$, $\ell_T$ is zero on $\Pi^\epsilon_\eta$ 
\item[(ii)] When $E=K$, $\ell_T$ is zero on $\big(\Pi_\eta^{\epsilon^\circ}\big)^{g_\iota}$ and $\Theta_{\psi_a}(\eta^\epsilon)$ for $a\neq a_T$. 
\item[(iii)] Suppose $E=K$ and $f=\Theta_{\psi_{a_T}}(\phi, E(\Phi))$ with $\phi\in \ms(V^2_K(\ba))$ and $\Phi\in \eta^\epsilon$, choose $\xi_T\in V_K^2$ with $q_K(\xi_T)=-2T/a_T$, then $\ell_T(f)=P(\phi,\Phi,\xi_T)+P(\phi,\Phi,\xi_T^{\iota})$. Furthermore, when $T=-a_T\smalltwomatrix{1}{}{}{\delta}$ and $\xi_T=(1,e)$, there is
\[
P(f)=
\begin{cases}
2^{\beta}\delta(\chi,\eta)P(\phi,\Phi,\xi_T)+2^{\beta}\delta(\chi,\eta^\iota)P(\phi,\Phi,\xi_T^\iota), &\eta\neq \eta^{\iota},\\
2^{\beta}\delta(\chi,\eta)P(\phi,\Phi,\xi_T), &\eta=\eta^\iota.
\end{cases}
\]
\item[(iv)] $P$ is nonzero if and only if $E=K$ and $\chi\in\{\eta, \eta^\iota\}$.
\end{itemize}
\end{proposition}
\begin{proof}
For $a\in F^\times/\Nm(K^\times)$ and $f=\Theta_{\psi_a}(\phi, \varphi)\in \Theta_{\psi_a}(\eta^\epsilon)$ with $\varphi=E(\Phi)$, there is
\begin{align*}
\ell_T(F)=&\int_{[\mrn]}\int_{[\OO(V_K)]}\overline{\varphi(h_1)}\sum_{\xi\in V_K^2(F)} \omega_{\psi_a}(N,h_1)\phi(\xi)\psi_T(N)\dd h_1\dd N\\
=&\int_{[\OO(V_K)]}\overline{\varphi(h_1)}\sum_{\xi\in V_K^2(F)}\phi(h_1^{-1}\xi)\int_{[\mrn]}\psi\big(\Tr[(\frac{bq_K(\xi)}{2}+T)N])\big)\dd N \dd h_1.
\end{align*}
In order that the above expression is nonzero, it is necessary that $T\in  -\frac{b}{2}q(V_K^2(F))$. When $E\neq K$, this can not happen. When $E=K$, this happens only when $a=a_T$ is the class such that $-2T\in aq_K(V_K^\circ(F))$. This proves (i) and (ii).

Now suppose $a=a_T$. The set $\Sigma_{-2T/a}(F)=\{\xi\in V_K^2(F): q_K(\xi)=-2T/a\}$ is homeomorphic to $\OO(V_K)_F$. Choosing one $\xi_T\in \Sigma_{-2T/a_T}(F)$, one can write $\ell_T(f)$ as
\[
\int_{[\OO(V_K)]}\overline{\varphi(h_1)}\sum_{\gamma\in \OO(V_K)_F}\phi(h_1^{-1}\gamma^{-1}\xi_T)\dd h_1=\int_{\OO(V_K)_\ba}\overline{(\Phi(h_1)+\Phi(\iota h_1))}\cdot \phi(h_1^{-1}\xi_T)\dd h_1.
\]
This proves the part about $\ell_T(f)$ in (iii).

%Choose $m_0\in \GL_2(F)$ so that $\leftup{t}{m_0}Tm_0=\smalltwomatrix{a}{}{}{ad}$. Write $\xi_0^\prime=\xi_0 m_0$, then $q(\xi_0)=\smalltwomatrix{2}{}{}{2d}$. Choose $e\in K$ with $\Tr(e)=0$ and $e^2=-d$, then there exists a unique $h_1^\prime\in \OO(V_K)_F$ such that $\xi_0^\prime=h_1^\prime(1,e)$. So one can write $\xi_0=h_1^\prime\xi_0^\prime m_0^{-1}$.

%By equation (\ref{groupT}), we see that $\SO(V_0)$ corresponds to \begin{equation}\label{groupT}\bt:=\left\{[\smalltwomatrix{A}{}{}{(\det A)\leftup{t}{A^{-1}}}]: A=m_0Bm_0^{-1}, B=\smalltwomatrix{\alpha}{-\beta\lambda}{\beta}{\alpha}, \alpha^2+\beta^2\lambda\neq 0\right\}.\end{equation}

Recall (\ref{groupT}). For the given special diagonal $T$ and special $\xi_T$, there is
\[
\SO(V_0)=\{[\smalltwomatrix{a}{-b \delta}{b}{a}]\in \PGL_2\}\cong \bt=\left\{[\smalltwomatrix{A}{}{}{(\det A)\leftup{t}{A^{-1}}}]: A=\smalltwomatrix{a}{-b \delta}{b}{a}\right\}
\]
For $g=[\smalltwomatrix{a}{-b \delta}{b}{a}]\in \SO(V_0)_\ba$, set $h_g=a+b e\in \ba_K^\times=\GSO(V_K)_\ba$, $\varphi_{h_g}(\cdot)=\varphi(\cdot h_g)$, and 
\[
f_g=\Theta(\omega_{\psi_a}(h_g,g)\phi,\varphi_{h_g}).
\]
There is $f(ug)=f_g(u)$. The previous calculation for $\ell_T$ yields
\begin{align*}
P(f)=&\int_{[\SO(V_0)]}\ell_T(f_g)\chi(g)\dd g\\
=&\int_{[\SO(V_0)]}\int_{\OO(V_K)_\ba}\chi(g)\varphi(h_1h_g)\omega_{\psi_{a_T}}(h_g,g)\phi(h_1^{-1}\xi_T)\dd h\dd g.
\end{align*}
Set $A=\smalltwomatrix{a}{-b \delta}{b}{a}$, then $\xi_TA=h_g \xi_T$. By the expression of Weil representation, there is
\[
\omega_{\psi_a}(h_g,g)\phi(h_1^{-1}\xi_0)=\phi({h_g}^{-1}h_1^{-1}\xi_TA)= \phi({h_g}^{-1}h_1^{-1}h_g\xi_T).
\]
Hence
\[
P(f)=\int_{[\SO(V_0)]}\int_{\OO(V_K)_\ba}\chi(g)\overline{\varphi(h_1h_g)}\phi({h_g}^{-1}h_1^{-1}h_g\xi_T).
\]
By making a change of variable $h_1\rar h_g h_1{h_g}^{-1}$, we see that
\begin{align*}
P(f)=&\int_{[\SO(V_0)]}\int_{\OO(V_K)_\ba}\chi(g)\overline{\varphi(h_g h_1)}\phi(h_1^{-1}\xi_T)\dd h_1\dd g\\
=&\int_{\OO(V_K)_\ba}\int_{[\SO(V_0)]}\big[\chi(g)\overline{\eta(h_g)\Phi(h_1)}+\chi(g)\overline{\eta^\iota(h_g)\Phi(\iota h_1)}\big]\phi(h_1^{-1}\xi_T)\dd h_1\dd g.
\end{align*}
When $\eta\neq \eta^\iota$, we observe $\int_{\OO(V_K)_\ba}\overline{\Phi(\iota h_1)}\phi(h_1^{-1}\xi_T)\dd h_1=\int_{\OO(V_K)_\ba}\overline{\Phi( h_1)}\phi(h_1^{-1}\xi_T^{\iota})\dd h_1$ by making a change of variable $h_1\rar \iota h_1$, then the formula for $P(\chi)$ follows. When $\eta=\eta^\iota$, there is $\Phi(\iota h_1)=\Phi(h_1)$ and the formula for $P(f)$ in this case follows. So (iii) is proved.

We now verify (iv). By (i), $E=K$ is a necessary condition for $P\neq 0$. Suppose $E=K$, we show that $\chi=\eta$ or $\eta^\iota$ is a sufficient condition. One may assume that the associated $T=J_1^{-1}v_0$ is of the shape $-a_T\smalltwomatrix{1}{}{}{\delta}$. This is because there always exists $h\in \SO(U)$ such that $J_1^{-1}h\circ v_0$ is of such a shape; also, if the assertion holds for $h\circ v_0$, then it also holds for $v_0$. Now the ``only if" part follows from the formula for $P(f)$ in (iii) .  For the other direction, it suffices to show that for $\xi=(\xi_1,\xi_2)\in V_K^2(F)$ with $\xi_1,\xi_2$ linearly independent, $P(\phi,\Phi,\xi)$ is nonzero for certain $\Phi\in \eta^\epsilon$ and $\phi\in \ms(V_K^2(\ba))$. Look at
\[
P_{v}(\phi_v,\Phi_v,\xi)=\int_{\OO(V_K)_{F_v}} \overline{\Phi_v(h_{1,v})}\phi_v(h_{1,v}^{-1}\xi)\dd h_{1,v},\quad \Phi_v\in \eta_v^{\epsilon_v}, \phi_v\in \ms(V_K^2(F_v)).
\]
It is $1$ almost everywhere.% and one can easily choose $\Phi_v, \phi_v$ at the remaining places so the the local integrals all all nonzero. This proves (iv)
 It suffices to show that $P_v(\cdot,\cdot,\xi)$ is noznero at the remaining finitely many places. When $\eta_v^{\iota_v}\neq \eta_v$, choose $\Phi_v$ with $\Phi_v(1)=1$ and $\Phi_v(\iota_v)=0$, then the local integral is $\int_{\SO(V_K)_{F_v}}\eta_v(h_v)\phi_v(h_v^{-1}\xi)\dd h_v$, which is is obviously nonzero for certain $\phi_v$. When $\eta_v^{\iota_v}=\eta_v$, we use the fact that $\xi, \xi^\iota$ are linearly independent over $K_v$ and choose $\phi_v$ of the shape $\phi_v(x_1\xi+x_2\xi^{\iota})=\phi_1(x_1)\phi_2(x_2)$, $x_1,x_2\in K_v$, $\phi_1,\phi_2\in \ms(K_v)$, then $P_{v}(\phi_v,\Phi_v,\xi)$ is \[ \int_{\SO(V_K)_{F_v}}\eta_v(h_v)\phi_1(h_v^{-1})\dd h \phi_2(0)+\eta_v(\iota_v)\int_{\SO(V_K)_{F_v}}\eta_v(h_v)\phi_2(h^{-1})\dd h \phi_1(0) \] One can certainly make the above expression to be nonzero.
\end{proof}

\subsection{Regularized local Bessel periods }

When $\Pi^\epsilon_\eta\in \So(\eta)$ is cuspidal with $\eta|_{\ba^\times}=1$,
\begin{equation}\label{lf_sp1}
L(s,\Pi^\epsilon_\eta,\chi)=\frac{L(s,\pi_\eta\boxtimes \pi_\chi)L(s+1,\pi_\eta\boxtimes \pi_\chi)}{L(s+1,\chi_{K/F})L(s+1,\chi_{E/F})\prod_{i=0}^2\big(\zeta_F(s+i)L(s+i,\eta\circ \pr)\big)}.
\end{equation}
If $E=K$ and $\chi\in\{\eta, \eta^\iota\}$, it is further reduced to
\begin{equation}\label{lfunction_s}
L(s,\Pi^\eta_\epsilon,\chi)=\frac{L(s,\chi_{K/F})}{\zeta_F(s+2)L(s+1,\chi_{K/F})L(s+2,\eta\circ \pr)}.
\end{equation}
Here $\pi_\eta,\pi_\chi$ refer to the local functorial lifts of $\eta, \chi$ to $\GL_2(\ba)$ and $\eta\circ \pr:\ba^\times_K\overset{\pr}{\rar} \ba^\times_{K,1}\hrar \ba^\times_K\overset{\eta}{\rar} \bc$ is a character of $\ba^\times_K/K^\times$.

Recall the integral $\mpp(s,\chi_v)$ defined in Section ~\ref{s_so2i}. There is
\[
\mpp_v(s;f_{1,v},f_{2,v})=\int_{\SO(V_0)_{F_v}}\ell_{v,T}(h_v\circ f_{1,v},f_{2,v})\chi_v(h_v)\Delta(h_v)^s \dd h_v,
\]
where $\ell_{v,T}(f_{1,v},f_{2,v}):=\int_{\bn(F_v)}(\Pi_v(N_v)f_{1,v},f_{2,v})\psi_{v,T}(N_v)\dd N_v$ for $f_{1,v},f_{2,v}\in \Pi_{\eta_v}^{\epsilon_v}$. Recall that
\[
\Pi^{\epsilon_v}_{\eta_v}=\begin{cases}
\theta(\eta_v^{\epsilon_v},\psi_v), &\text{$K_v$ is split},\\
\oplus_{a_v\in F_v^\times\backslash \Nm(K_v^\times)} \theta(\eta_v^{\epsilon_v},\psi_{v,a_v}), &\text{$K_v$ is nonsplit}.
\end{cases}
\]
For given $T$, there is a unique $a_{v,T}\in F_v^\times/\Nm(K_v^\times)$ such that $T\in -\frac{a_{v,T}}{2}q_E(V_E^{2,\circ}(F_v))$.
\begin{proposition}\label{lbp_sp}
Suppose $f_{i,v}=\theta_{\psi_{v,a_v}}(\phi_{i,v},\Phi_{i,v})\in \theta(\eta_v^{\epsilon_v},\psi_{v,a_v})$. \emph{(i)} When $E_v\neq K_v$ or $E_v=K_v$ with $a_v\neq a_{v,T}$, $\ell_{v,T}$ is zero on $\theta(\eta_v^{\epsilon_v},\psi_{v,a_v})$. \emph{(ii)} When $E_v=K_v$ and $a_v=a_{v,T}$, choose $\xi_{v,T}\in V_K^2(F_v)$ with $q_K(\xi_{v,T})=-2T/a_{v,T}$, then $\ell_{v,T}(f_{1,v},f_{2,v})$ is equal to
\begin{equation*}
2^{{\beta_v}-1}|a_{v,T}||2\det T|^{-\frac{1}{2}}c_{\eta_v}c_{E_v}\cdot
\begin{cases}
\sum_{\tau_v\in \Gal(K_v/F_v)}P_v(\phi_{1,v},\Phi_{1,v},\xi_{v,T}^{\tau_v})\overline{P_v(\phi_{2,v},\Phi_{2,v},\xi_{v,T}^{\tau_v})}, &\eta_v^{\iota_v}\neq \eta_v,\\
P_v(\phi_{1,v},\Phi_{1,v},\xi_{v,T})\overline{P_v(\phi_{2,v},\Phi_{2,v},\xi_{v,T})}, &\eta_v^{\iota}=\eta_v.
\end{cases}
\end{equation*}
Furthermore, when $T=-a_{v,T}\smalltwomatrix{1}{}{}{\delta}$ and $\xi_{v,T}=(1,e)$, $\mpp_v(s;f_{1,v},f_{2,v})$ is equal to
\begin{equation*}
2^{{\beta_v}-1}|a_{v,T}||2\det T|^{-\frac{1}{2}}c_{\eta_v}c_{E_v}\cdot
\begin{cases}
\sum_{\tau_v=1,\iota_v}\mpp(s,\chi_v\overline{\eta_v^{\tau_v}})P_v(\phi_{1,v},\Phi_{1,v},\xi_{v,T}^{\tau_v})\overline{P_v(\phi_{2,v},\Phi_{2,v},\xi_{v,T}^{\tau_v})}, &\eta_v\neq \eta_v^{\iota},\\
\mpp(s,\chi_v\overline{\eta_v})P_v(\phi_{1,v},\Phi_{1,v},\xi_{v,T})\overline{P_v(\phi_{2,v},\Phi_{2,v},\xi_{v,T})}, &\eta_v=\eta_v^\iota.
\end{cases}
\end{equation*}
\end{proposition}
\begin{proof}
To compute $\ell_{v,T}(f_{1,v},f_{2,v})$, 
we need to calculate $\int_{\mrn(F_v)}(\Pi_v(N_v)f_{1,v},f_{2,v})\widehat{\lambda}(N_v)\dd N_v$, where $\lambda\in C^\infty_c (\Sym_2^\circ(F_v))$ and $\widehat{\lambda}(N_v)=\int_{\Sym_2(F_v)}\lambda(N_v^\prime)\psi(\Tr(N_vN_v^\prime))\dd N_v^\prime$.

Temporarily set $\rho(h_v,X_v)=\phi_{1,v}(h_v^{-1}X_v)\overline{\phi_{2,v}(X_v)}$. By Remark ~\ref{lipf_sp},
\[
(f_{1,v},f_{2,v})=c_{\eta_v}\int_{\OO(V_K)_{F_v}}\overline{(h_{1,v}\circ\Phi_{1,v},\Phi_{2,v})}\int_{V_K^2(F_v)}\rho(h_v,X_v)\dd X_v
\]
It follows that $\int_{\mrn(F_v)}(\Pi_v(N_v)f_{1,v},f_{2,v})\widehat{\lambda}(N_v)\dd N_v$ is equal to
\begin{align*}
&c_{\eta_v}\int_{\Sym_2(F_v)}\int_{\OO(V_K)_{F_v}}\overline{(h_{1,v}\circ\Phi_{1,v},\Phi_{2,v})}\widehat{\lambda}(N_v)\int_{V_K^2(F_v)}\rho(h_v,X_v)\psi(a_v\Tr(q(X_v)N_v)/2)\dd X_v\dd h_v\dd N_v.
\end{align*}
After the innermost integration over $V_K^2(F_v)$, the integrand is absolutely convergent on $\Sym_2(F_v)\times \OO(V_K)_{F_v}$, so one can change the order of integration to $\int_{\OO(V_K)_{F_v}}\int_{\Sym_2(F_v)}\int_{V_K^2(F_v)}$. One can further change the order to $\int_{\OO(V_K)_{F_v}}\int_{V_K^2(F_v)}\int_{\Sym_2(F_v)}$ because $\widehat{\lambda}(N_v)\rho(h_v,X_v)$ is absolutely integrable concerning $(N_v,X_v)$. Applying Fourier inversion formula to the $N_v$-integration, we get
\begin{align*}
&\int_{\mrn(F_v)}(\Pi_v(N_v)f_{1,v},f_{2,v})\widehat{\lambda}(N_v)\dd N_v\\
=&c_{\eta_v}\int_{\OO(V_K)}\int_{V_K^2(F_v)}\overline{(h_v\circ\Phi_{1,v},\Phi_{2,v})}\rho(h_v,X_v)\lambda(-a_vq(X_v)/2)\dd X_v \dd h_v\\
=&\frac{c_{\eta_v}}{|2|^{1/2}}\int_{\OO(V_K)}\int_{\Sym_2(F_v)}\int_{\Sigma_{-2N_v}}\overline{(h_v\circ\Phi_{1,v},\Phi_{2,v})}\rho(h_v,X_v)\dd \sigma_{v,-2N_v/a_v}\lambda(N_v)|\det N_v|^{-1/2}\dd N_v \dd h_v.
\end{align*}
Here in the second equality, we use the measure decomposition in Lemma ~\ref{l_md2}. One can move the $N_v$-integration to the outside because $\lambda(N_v)\in C_c^\infty(\Sym_2^\circ(F_v))$. 

According to Definition ~\ref{def_ui}, $\ell_{v,T}(f_{1,v},f_{2,v})=\int_{\mrn(F_v)}(\Pi_v(N_v)f_{1,v},f_{2,v})\psi_{v,T}(N_v)\dd N_v$ is
\[
c_{\eta_v}|a_v||2\det T|^{-1/2}\int_{\OO(V_K)}\int_{\Sigma_{v,-2T}}\overline{(h_v\circ\Phi_{1,v},\Phi_{2,v})}\rho(h_v,X_v)\dd \sigma_{v,-2T/a_v}
\]

When $E_v\neq K_v$ or $E_v=K_v$ with $T\not\in -\frac{a_v}{2}q(V_K^2(F_v))$, the set $\Sigma_{v,-2T/a_v}$ is empty, whence the above expression is zero. Otherwise, choose $\xi_{v,T}\in V_K^2(F_v)$ with $q(\xi_{v,T})=-2T/a_v$, write $X_v={h_v^\prime}^{-1}\xi_{v,T}$ and recall $\dd \sigma_{v,-2T/a_v}=2c_{K_v}\dd h_v^\prime$ from Section ~\ref{ss_md2}. We have
\begin{align*}
\ell_{v,T}(f_{1,v},f_{2,v})=&\frac{2|a_v|c_{\eta_v}c_{E_v}}{|2\det T|^{1/2}}\int_{\OO(V_K)_{F_v}^2}\overline{(h_v\circ\Phi_{1,v},\Phi_{2,v})}\phi_{1,v}({h_v}^{-1}{h_v^\prime}^{-1}\xi_{v,T})\overline{\phi_{2,v}(h_v^{-1}\xi_{v,T})}\dd h_v\dd h_v^\prime\\
=&\frac{2|a_v|c_{\eta_v}c_{E_v}}{|2\det T|^{1/2}}\int_{\OO(V_K)_{F_v}^2}\overline{(h_v\circ\Phi_{1,v},h_v^\prime\Phi_{2,v})}\phi_{1,v}({h_v}^{-1}\xi_{v,T})\overline{\phi_{2,v}({h_v^\prime}^{-1}\xi_{v,T})}\dd h_v\dd h_v^\prime.
\end{align*}

Note that
\[
(h_v\circ \Phi_{1,v},h_v^\prime\Phi_{2,v})=\frac{1}{2}\Phi_{1,v}(h_v)\overline{\Phi_{2,v}(h_v^\prime)}+\frac{1}{2}\Phi_{1,v}(\iota_v h_v)\overline{\Phi_{2,v}(\iota_vh_v^\prime)}.
\]
When $\eta_v=\eta_v^{\iota_v}$, there is $(h_v\circ \Phi_{1,v},h_v^\prime\Phi_{2,v})=\Phi_{1,v}(h_v)\overline{\Phi_{2,v}(h_v^\prime)}$, whence
\[
\ell_{v,T}(f_{1,v},f_{2,v})=2|a_v||2\det T|^{-\frac{1}{2}}c_{\eta_v}c_{E_v}P_v(\phi_{1,v},\Phi_{1,v},\xi_{v,T})\overline{P_v(\phi_{2,v},\Phi_{2,v},\xi_{v,T})}.
\]
When $\eta_v\neq \eta_v^{\iota_v}$, there is $\int_{\OO(V_K)_{F_v}^2}\overline{\Phi_{1,v}(\iota_v h_v)}\Phi_{2,v}(\iota_vh_v^\prime)\phi_{1,v}({h_v}^{-1}\xi_{v,T})\overline{\phi_{2,v}({h_v^\prime}^{-1}\xi_{v,T})}\dd h_v\dd h_v^\prime=P_v(\phi_{1,v},\Phi_{1,v},\xi_{v,T}^{\iota_v})\overline{P_v(\phi_{2,v},\Phi_{2,v},\xi_{v,T}^{\iota_v})}$, whence
\[
\ell_{v,T}(f_{1,v},f_{2,v})=|a_v||2\det T|^{-\frac{1}{2}}c_{\eta_v}c_{E_v}\sum_{\tau_v\in \Gal(K_v/F_v)}P_v(\phi_{1,v},\Phi_{1,v},\xi_{v,T}^{\tau_v})\overline{P_v(\phi_{2,v},\Phi_{2,v},\xi_{v,T}^{\tau_v})}
\]
This proves the assertion about $\ell_{v,T}$. When when $T=-a_{v,T}\smalltwomatrix{1}{}{}{\delta}$ and $\xi_{v,T}=(1,e)$, the formula for $\mpp_v(s,f_{1,v},f_{2,v})$ follows from the fact
\[
P_v(g_v\circ \phi_{1,v},\Phi_{1,v},\xi_{v,T}^{\tau_v})=\overline{\eta_v^{\iota_v}(g_v)}P_v(g_v\circ \phi_{1,v},\Phi_{1,v},\xi_{v,T}^{\tau_v}),\quad g_v\in \SO(V_0)_{F_v}.
\]
\end{proof}

The above proposition tells that $\mpp_v(s,-)$ is zero on $\Pi_v^{\epsilon_v}\otimes \Pi_v^{\epsilon_v}$ when $E_v\neq K_v$ and factors through $\theta(\eta_v^{\epsilon_v},\psi_{v,a_{v,T}})\otimes \theta(\eta_v^{\epsilon_v},\psi_{v,a_{v,T}})$ when $E_v=K_v$. Now consider
\[
\mpp^\sharp_v(f_{1,v},f_{2,v})=\frac{\mpp_v(s; f_{1,v},f_{2,v})}{\zeta_{F_v}(2)\zeta_{F_v}(4)L(s,\Pi^{\epsilon_v}_{\eta_v},\chi_v)}\big|_{s=0},\quad  f_{1,v},f_{2,v}\in \Pi_{\eta_v}^{\epsilon_v}.
\]
When $E_v=K_v$ and $\chi_v\in \{\eta_v,\eta_v^{\iota_v}\}$, we may assume $T=-a_{v,T}\smalltwomatrix{1}{}{}{\delta}$ and $\xi_{v,T}=(1,e)$, then for $f_{i,v}=\theta_{\psi_{v,a_v}}(\phi_{i,v},\Phi_{i,v})\in \theta(\eta_v^{\epsilon_v},\psi_{v,a_v})$, there is
\begin{align}\label{rlbp}
\nonumber \mpp^\sharp_v(f_{1,v},f_{2,v})=&\frac{2^{{\beta_v}}|a_{v,T}|c_{\eta_v}c_{K_v}L(1,\chi_{K_v/F_v})L(2,\eta_v\circ \pr)}{|2\det T|^{\frac{1}{2}}\zeta_{F_v}(4)}\cdot\\
& \cdot
\begin{cases}
\sum_{\tau_v=1,\iota_v}\mpp^\sharp(\chi_v\overline{\eta_v^{\iota_v}})P_v(\phi_{1,v},\Phi_{1,v},\xi_{v,T}^{\tau_v})\overline{P_v(\phi_{2,v},\Phi_{2,v},\xi_{v,T}^{\tau_v})}, &\eta_v\neq \eta_v^{\iota_v},\\
\mpp^\sharp(\chi_v\overline{\eta_v}) P_v(\phi_{1,v},\Phi_{1,v},\xi_{v,T})\overline{P_v(\phi_{2,v},\Phi_{2,v},\xi_{v,T})}, &\eta_v=\eta_v^{\iota_v}.
\end{cases}
\end{align}
For a character $\chi_v^\prime$ of $\SO(V_0)_{F_v}$, the symbol $\mpp^\sharp(\chi_v^\prime)$ is defined in Section ~\ref{s_so2i} and takes a nonzero value only when $\chi_v^\prime=1$. Recall that $\mpp^\sharp(1_v)$ is denoted by $c_{K_v}^\prime$.

\begin{proposition}\label{lnvs_sp}
The following statements are equivalent:
\begin{itemize}
\item[(i)] $\Hom_{\mrr(F_v)}(\Pi_{\eta_v}^{\epsilon_v}\otimes (\chi_v\boxtimes \psi_{v_0,v}),\bc)\neq 0$,

\item[(ii)] $E_v=K_v$ and $\chi_v\in \{\eta_v,\eta_v^{\iota_v}\}$

\item[(iii)] $\ord_{s=0}L(s,\Pi_{\eta_v}^{\epsilon_v},\chi_v)\leq 0$ and $\mpp_v^\sharp$ is nonzero on $\Pi_{\eta_v}^{\epsilon_v}\otimes \Pi_{\eta_v}^{\epsilon_v}$.
\end{itemize}
\end{proposition}
\begin{proof}
The equivalence of (i) and (ii) follows from the structure of the twisted Jacquet module of $\Pi^{\epsilon_v}_{\eta_v}$ with respect to $\psi_{v_0,v}=\psi_{v,T_v}$, as described in Proposition 7.6 of \cite{gs2013}.

We now argue for the equivalence of (ii) and (iii). Observe that $\ord_{s=0}L(s,\Pi_{\eta_v}^{\epsilon_v},\chi_v)=\ord_{s=0}\frac{L(s,\pi_{\eta_v}\boxtimes\pi_{\chi_v})}{\zeta_{F_v}(s)L(s,\eta_v\circ \pr)}$. Suppose (ii), then $\ord_{s=0}L(s,\Pi_{\eta_v}^{\epsilon_v},\chi_v)=\ord_{s=0}L(s,\chi_{E_v/F_v})\leq 0$ and $\mpp_v(s;-)$ factors through $\theta(\eta_v^{\epsilon_v},\psi_{v,a_{v,T}})\otimes \theta(\eta_v^{\epsilon_v},\psi_{v,a_{v,T}})$, on which $\mpp_v^\sharp$ is given by (\ref{rlbp}). Because $P_v(\phi_{1,v},\Phi_{1,v},\xi_{v,T}^{\tau_v})$ ($\tau_v\in \Gal(K_v/F_v)$) is nonzero on $\ms(V^2_K(F_v))\otimes \eta_v^{\epsilon_v}$ (as argued in the end of the proof for Proposition ~\ref{functionalT}), one sees that $\mpp_v^\sharp$ is nonzero. Thus, (iii) holds.

Conversely, we suppose (iii), then $E_v$ must be equal to $K_v$; otherwise, $\mpp_v(s;-)$ is identifically zero and hence $\mpp_v^\sharp$ is zero. With $E_v=K_v$, the condition $\ord_{s=0}L(s,\Pi_{\eta_v}^{\epsilon_v},\chi_v)\leq 0$ implies $\ord_{s=0}L(s,\pi_{\eta_v}\boxtimes\pi_{\chi_v})<0$, which further forces $\chi_v\in \{\eta_v,\eta_v^{\iota_v}\}$. Actually, if $K_v$ is split, then $\chi_v,\eta_v$ as characters of $K_v^\times=F_v^\times\times F_v^\times$ are of the shape $\mu_v\boxtimes \mu_v^{-1}$ and $\mu_v^\prime\boxtimes {\mu_v^\prime}^{-1}$, whence
\[
L(s,\pi_{\chi_v}\boxtimes \pi_{\eta_v})=L(s,\mu_v\mu_v^\prime)L(s,\mu_v^{-1}\mu_v^\prime)L(s,\mu_v{\mu_v^\prime}^{-1})L(s,\mu_v{\mu_v^\prime}^{-1}),
\]
which has a pole at $s=0$ only when $\mu_v^\prime=\mu_v$ or $\mu_v^{-1}$, that is, when $\chi_v=\eta_v$ or $\eta_v^{\iota_v}$. If $K_v$ is nonsplit and nonarchimedean, one can similarly verify the relation when $\pi_{\chi_v},\pi_{\eta_v}$ are induced representations; when one of them is supercuspidal, we apply \cite[Prop. 1.2]{geljac78} to see that $\pi_{\chi_v}\cong \pi_{\eta_v}$, whence $\chi_v\in \{\eta_v,\eta_v^{\iota_v}\}$. Finally, when $K_v$ is nonsplit and archimedean, there must be $K_v=\bc$ and $F_v=\br$, for which the relation can be checked directly because one can numberate all cases of $\pi_{\chi_v},\pi_{\eta_v}$.
\end{proof}

\begin{theorem}\label{lc_sp_split}
When $\Pi\in \So(\eta)$ is cuspidal, Conjecture ~\ref{conj_lbp} holds for $\Pi_v$.
\end{theorem}
\begin{proof}
We verify Part (i) and (ii) of Conjecture ~\ref{conj_lbp}; Part (iii) is exactly Proposition ~\ref{lnvs_sp}. 

When $E_v\neq K_v$, $\mpp_v(s;f_{1,v},f_{2,v})$ is zero for all $f_{1,v},f_{2,v}\in \Pi_{\eta_v}^{\epsilon_v}$, whence the assertions in Part (i) and (ii) of Conjecture ~\ref{conj_lbp} automatically hold. 

Now suppose $E_v=K_v$. We may assume $T=-a_{v,T}\smalltwomatrix{1}{}{}{\delta}$, then $\mpp_v(s;-)$ factors through $\theta(\eta_v^{\epsilon_v},\psi_{v,a_{v,T}})\otimes \theta(\eta_v^{\epsilon_v},\psi_{v,a_{v,T}})$ and is given in Proposition ~\ref{lbp_sp}. It has meromorphic continuation because $\mpp(s;\chi^\prime_v)$ is so for any character $\chi_v^\prime$ of $\SO(V_0)_{F_v}$. When $\ord_{s=0}L(s,\Pi_{\eta_v}^{\epsilon_v},\chi_v)\leq 0$, there is $\chi_v\in \{\eta_v,\eta_v^{\iota_v}\}$ as shown in the proof of Proposition ~\ref{lbp_sp}. Hence $\ord_{s=0}L(s,\Pi_{\eta_v}^{\epsilon_v},\chi_v)=\ord_{s=0} L(s,\chi_{E_v/F_v})$. By Lemma ~\ref{integralsov0}, $\ord_{s=0}\mpp_v(s;f_{1,v},f_{2,v})\geq \ord_{s=0} L(s,\chi_{E_v/F_v})$. Therefore, $\ord_{s=0}\mpp_v(s;f_{1,v},f_{2,v})\geq \ord_{s=0}L(s,\Pi_{\eta_v}^{\epsilon_v},\chi_v)$. It is obvious that the functional $\mpp_v^\sharp$ given in (\ref{rlbp}) respects the actions of $\mrr(F_v)\times \mrr(F_v)$.
\end{proof}

\subsection{The period formula}

In Theorem \ref{bp_sp} below. We see that $P$ is nonzero on a cuspidal $\Pi^\eta_\epsilon$ if and only if the local homomorphism space is nonzero at every local place. This accords with our main theorm in the introduction because in this situation there is $E=K$ and $\chi\in \{\eta,\eta^\iota\}$, whence $L(0,\Pi,\chi)=\frac{\zeta_F(4)}{L(2,\eta\circ \pr)}$ is automatically nonzero.

\begin{theorem}\label{bp_sp}
Suppose $\Pi_\eta^\epsilon\in \So(\eta)$ is cuspidal. \emph{(i)} $P$ is nonzero on $\Pi_\eta^\epsilon$ if and only if $\Hom_{\mrr(F_v)}(\Pi_{\eta_v}^{\epsilon_v}\otimes (\chi_v\boxtimes \psi_{v_0,v}),\bc)$ is nonzero for all $v$, \emph{(ii)} if so, then $P\otimes \overline{P}=2^{\beta}\prod_v 2^{-{\beta_v}}\mpp_v^\sharp$.
\end{theorem}
\begin{proof}
By Proposition ~\ref{functionalT}, $P$ is nonzero if and only $E=K$ and $\chi\in \{\eta,\eta^\iota\}$. By Proposition ~\ref{lnvs_sp}, $\Hom_{\mrr(F_v)}(\Pi_{\eta_v}^{\epsilon_v}\otimes (\chi_v\boxtimes \psi_{v_0,v}),\bc)$ is nonzero for all $v$ if and only if $E=K$ and $\chi_v\in \{\eta_v,\eta_v^{\iota_v}\}$ for all $v$. So for (i), it suffices to show that if a global character $\chi$ of $[\SO(V_0)]$ has $\chi_v\in \{\eta_v,\eta_v^{\iota_v}\}$ for all $v$, then $\chi\in \{\eta,\eta^{\iota}\}$. This is easy to see, say for example, from Lemma 10.3 of \cite{gs2013}. 

Supposing $E=K$ and $\chi=\{\eta, \eta^\iota\}$, we now prove (ii). We may assume $\chi=\eta$ and $T=J_1^{-1}v_0$ is $-\smalltwomatrix{1}{}{}{\delta}$. Put $\xi_T=(1,e)$. Consider two vectors $f_i=\Theta_\psi(\phi_i,E(\Phi_i))\in \Pi^\epsilon_\eta$, $\phi_i\in \ms(V^2_K(\ba))$ and $\Phi_i\in \eta^\epsilon$. By Proposition ~\ref{functionalT}, there is
\[
P(f_1)\overline{P(f_2)}=2^{2\beta}\prod_v P_v(\phi_{1,v},\Phi_{1,v},\xi_T)\overline{P_v(\phi_{1,v},\Phi_{1,v},\xi_T)}.
\]
On the other hand, by (\ref{rlbp}), there is
\[
2^{-{\beta_v}}\mpp_v^\sharp(f_{1,v},f_{2,v})=\frac{c_{\eta_v}c_{K_v}c_{K_v}^\prime L(1,\chi_{K_v/F_v})L(2,\eta_v\circ \pr) }{\zeta_{F_v}(4)}P_v(\phi_{1,v},\Phi_{1,v},\xi_T)\overline{P_v(\phi_{1,v},\Phi_{1,v},\xi_T)}
\]
For almost all $v$, there are $c_{\eta_v}=\frac{\zeta_{F_v}(4)}{L(2,\eta_v\circ \pr)}$, $c_{K_v}=L(1,\chi_{K_v/F_v})$, $c_{K_v}^\prime=1$. So $2^{-{\beta_v}}\mpp_v^\sharp(f_{1,v},f_{2,v})=1$ almost everywhere. By Lemma ~\ref{md2_r}, Lemma ~\ref{cep}, and Remark ~\ref{lipf_sp}, we further have

\begin{align}\label{pcomp}
\prod_v 2^{-{\beta_v}}\mpp_v^\sharp(f_{1,v},f_{2,v})
=&\frac{2^{\beta}L(2,\eta\circ\pr)}{\zeta_F(4)}\prod_v P_v(\phi_{1,v},\Phi_{1,v},\xi_T)\overline{P_v(\phi_{1,v},\Phi_{1,v},\xi_T)}.
\end{align}
Therefore $P(f_1)\overline{P(f_2)}=\frac{2^\beta\zeta_F(4)}{L(2,\eta\circ\pr)} \prod_v 2^{-\beta_v}\mpp_v^\sharp(f_{1,v},f_{2,v})$.
\end{proof}

\section{Soudry Packets on $\SO(4,1)$}\label{sp_so41}

The inner forms $\SO(4,1)$ over $F$ are classified as $\SO(V_D)\cong \PGU_2(D)$ in Section ~\ref{so41}, with $D$ running over nonsplit quaternion algebras over $F$. The Soudry packets on $\GU_2(D)_\ba$ are constructed with the theta correspondence between $\GU_1(D)$ of skew-Hermitian type and $\GU_2(D)$. The study of the Bessel period functional on these Soudry packets is essentially the same as for the Soudry packets on $\SO(3,2)_\ba$. The major difference is that when $D$ is nonsplit, $\U_1(D)$ is non-connected and has the same set of $F$-rational points as its connected component $\U_1(D)^c$. This affects the choice of the Tamagawa measure on $\U_1(D)_\ba$ (See Section ~\ref{s_tmeasure}) and the construction of the packet slightly (See Sect. ~\ref{ss_spd}), but the qualitative assertions about Bessel period functionals and the global period formula remain the same. We briefly treat $\SO(4,1)$ with emphasis on the issues for which $\SO(4,1)$ differs from $\SO(3,2)$. The Soudry packet $\So(\eta)$ on $\GU_2(D)$ will be put as $\So_D(\eta)$ in this section, to emphasize the role of $D$.

\subsection{The groups}\label{sod}

Let $D$ be a nonsplit quaternion algebra over $F$. For a nonzero element $e\in D_0$, recall from Section ~\ref{gu1dgroup} that $\mv_e$ denotes the space $D$ with the skew-Hermitian pairing satisfying $(1,1)=e$. Set $G=\GU_2(D)$, $H=\GU(\mv_e)$, $G_1=\{g\in G:\nu(g)=1\}$, $H_1=\{h\in H:\nu(h)=1\}$, $PG=G/F^\times$, $PH=H/F^\times$, and $K=F\oplus Fe$. Let $S_D$ denote the set of places $v$ such that $D(F_v)$ is nonsplit. Let $H^c$ and $H_1^c$ be the identity components of $H$ and $H_1$ respectively.

(i) Write $D=K\oplus Ke^\prime$ with $e^\prime\in D_0(F)$ satisfying $ee^\prime=-e^\prime e$. Put $e^2=-\delta$ and $\ep^2=\delta^\prime$, then $\delta\not\in {F^\times}^2$ and $\delta^\prime\not\in \Nm(K^\times)$ because $D$ is nonsplit. There are $H(F)=K^\times\sqcup K^\times e^\prime$, $H_1(F)=K_1^\times$, and $\nu(H(F))=\Nm(K^\times)\sqcup \delta^\prime \Nm(K^\times)$.

(ii) Locally, $\delta^\prime\in \Nm(K_v^\times)$ if and only if $v\not\in S_D$; for $v\not\in S_D$, choose $h_v\in K_v^\times$ satisfying $\delta^\prime=\Nm(h_v)$ and set $\mu_2(F_v)=\{1,h_ve^\prime\}$. There are $H_v(F_v)=K_v^\times\sqcup K_v^\times e^\prime$ and
\[
H_1(F_v)=
\begin{cases}
K_{v,1}^\times\rtimes \mu_2(F_v), &v\not\in S_D,\\
K_{v,1}^\times, &v\in S_D.
\end{cases}
,\quad \nu(H(F_v)=
\begin{cases}
\Nm(K_v^\times), &v\not\in S_D,\\
F_v^\times, &v\in S_D.
\end{cases}
\]

(iii) Set $\ba^\times_{K,D}:=\{(t_v)\in \ba^\times: \text{$t_v\in \Nm(K_v^\times)$ when $v\not\in S_D$}\}$, then $\nu(H(\ba))=\ba^\times_{K,D}$. Accordingly, set $G^+=\{g\in G:\nu(g)\in \nu(H)\}$. There is $G^+(\ba)=\prod_v G^+(F_v)$ with
\[
G^+(F_v)=
\begin{cases}
\big\{g_v\in G(F_v): \nu(g_v)\in \Nm(K_v^\times)\big\}, &v\not\in S_D,\\
G(F_v), &v\in S_D,
\end{cases}
\]
and $G^+(F)=\{g\in G(F):\nu(g)\in \nu(H(F))\}$. Additionally, set $G^\dagger(F)=\{\gamma\in G(F): \nu(\gamma)\in \Nm(K^\times)\}$ and $G^\ddagger(F)=G(F)\cap G^+(\ba)$

\begin{lemma}\label{product}
\emph{(i)} $F^\times \ba^\times_{K,D}=\ba^\times$. Numerate the quaternion algebras $D^\prime$ with $S_{D^\prime}\subseteq S_D$ as $D_i$, $1\leq i\leq 2^{|S_D|-1}$. Write $D_i=K\oplus Ke_i$ with $\Tr(e_i)=0$, $ee_i=-ee_i$ and $e_i^2=\delta_i$, then $F^\ddagger:=F^\times\cap \ba_{K,D}^\times=\sqcup_i \delta_i\Nm(K^\times)$. \emph{(ii)} $G(F)\cdot G^+(\ba)=G(\ba)$.  For each $D_i$, choose $\gamma_{\delta_i}\in G^\ddagger(F)$ with $\nu(\gamma_{\delta_i})=\delta_i$, then $G^+(F)=G^\dagger(F)\sqcup G^\dagger \gamma_{\delta^\prime}$ and $G^\ddagger(F)=\sqcup_{i=1}^{2^{|S_D|-1}} G^\dagger(F)\gamma_{\delta_i}$.
\end{lemma}
\begin{proof}
(i) Observe that the group $F^\times\cdot \Nm(\ba^\times_K)$, as the kernel of the quadratic character $\chi_{K/F}$, is contained in $\ba^\times_{K,D}$ and is of index $2$ in $\ba^\times$. To prove $\ba^\times_{K,D}=\ba^\times$, it suffices to show that $\chi_{K/F}$ is nontrivial on $\ba^\times_{K,D}$. Because $D$ is nonsplit, $S_D$ is nonempty. Choose $v^\prime\in S_D$ and consider the element $t^\prime=(t_v)$ with $t_{v^\prime}=\delta^\prime$ and $t_v=1$ for $v\neq v^\prime$, then $t\in \ba^\times_{K,D}$ and $t_{v^\prime}\not\in \Nm(K_{v^\prime}^\times)$, whence $\chi_{K/F}(t)=-1$.

For the asssertion about $F^\times\cap \ba_{K,D}^\times$, we notice that $\chi_{K_v/F_v}(\delta_i)$ is  $-1$ for $v\not\in S_{D_i}$ and is $1$ for $v\not\in S_{D_i}$. For two $D_i, D_j$, there is $v$ such that $\chi_{K_v/F_v}(\delta_i\delta_j^{-1})=-1$, whence $\delta_i\delta_j^{-1}\not\in \Nm(K^\times)$ ; this shows that $\delta_i\Nm(K^\times)$ are disjoint. Second, for $a\in F^\times\cap \ba^\times_{K,D}$, there is $\chi_{K/F}(a)=1$ and $\chi_{K_v/F_v}(a)=1$ for $v\not\in S_D$, whence $\prod_{v\in S_D}\chi_{K_v/F_v}(a)=1$. The set of places with $\chi_{K_v/F_v}(a)=-1$ is of even cardinality and is equal to $S_{D_i}$ for certain $D_i$. There is $\chi_{K_v/F_v}(a)=\chi_{K_v/F_v}(\delta_i)$ for all $v$, whence $a\in \delta_i \Nm(K^\times)$. Assertions in (ii) directly follow from (i).
\end{proof}

With Lemma ~\ref{product} (ii), one sees that $G(F)\backslash G(\ba)=G^\ddagger(F)\backslash G^+(\ba)$. The Tamagawa measure on $G(\ba)$ induces a measure on $G^+(\ba)$, with respect to which there is $\Vol(G^+(F)\ba^\times\backslash G^+(\ba))=2^{|S_D|-1}$ because $\Vol(G(F)\ba^\times\backslash G(\ba))=2$ and $G^+(F)$ is of index $2^{|S_D|-2}$ in $G^\ddagger(F)$.

\subsection{Weil representation}

Fix a character $\psi$ of $\ba/F$ and let $\omega:=\omega_{\psi,\mv_e}$ be the Weil representation of $G_1(\ba)\times H_1(\ba)$ on $\ms(\mv_e(\ba))$, the space of Bruhat-Schwartz functions on $\mv_e(\ba)$:
\begin{align*}
&\omega(h)\phi(x)=\phi(h^{-1}x),\quad h\in H_1(\ba),\\
&\omega\smalltwomatrix{a}{}{}{\bar{a}^{-1}}\phi(x)=\chi_{K_e/F}(\Nm_{D}(a))|\Nm_{D}(a)|\phi(xa),\quad a\in D^\times(\ba),\\
&\omega\smalltwomatrix{1}{u}{}{1})\phi(x)=\phi(x)\psi\big(\frac{1}{2}\Tr_{D}(u(x,x))\big),\quad u\in D_0(\ba),\\
&\omega\smalltwomatrix{}{1}{1}{}\phi(x)=\int_{\mv_e(\ba)}\phi(y)\psi(\Tr_{D}(x,y))\dd y.
\end{align*}
Set $R(\ba):=\{(g,h)\in G^+(\ba)\times H(\ba): \nu(g)=\nu(h)\}$. The extension of $\omega$ from $G_1(\ba)\times H_1(\ba)$ to $R(\ba)$ is given below and we write $\omega^+=\mathrm{ind}_{R(\ba)}^{G^+(\ba)\times H(\ba)}\omega$.
\begin{equation}\label{wr_s2d}
\omega_\psi(g,h)\phi(X)=|\nu(h)|^{-1}\omega_{\psi}\big(g\smalltwomatrix{1}{}{}{\nu(g)^{-1}}\big)\phi(h^{-1}X),\quad (g,h)\in R(\ba).
\end{equation}

Locally, it is known that the local Howe duality holds for the triplets $(\omega_v, G_1(F_v),H_1(F_v))$ and $(\omega_v^+, G^+(F_v),H(F_v))$ at every place. Globally, for an automorphic form $\varphi$ on $[H]$ and $\phi\in \ms(\mv_e(\ba))$, define
\[
\Theta_\psi(\phi,\varphi)(g)=\int_{[H_1]}\Theta_{\phi,\psi}(g,h_1h^\prime)\overline{\varphi(h_1h^\prime)}\dd h_1,\quad g\in G^+(\ba),
\]
where $h^\prime\in H(\ba)$ is an element satisfying $\nu(h^\prime)=\nu(g)$. %We also write $\Theta(\phi,\varphi)$ for $\Theta_\phi(\varphi)$.

\subsection{The packet $\Pa_D(\eta)$}

Let $\eta$ be a character of $\ba^\times_K/K^\times$ and $S_\eta=\{v:\eta_v^{\iota_v}\neq \eta_v\}$. Set
\[
\beta=
\begin{cases}
1, &\eta^\iota\neq \eta,\\
2. &\eta^\iota=\eta.
\end{cases},\quad 
\beta_v=
\begin{cases}
1, &\eta_v^{\iota_v}\neq \eta_v,\\
2. &\eta_v^{\iota_v}=\eta_v.
\end{cases}
\]

(i) For $v\in S_\eta$, set $\eta_v^+=\Ind^{H(F_v)}_{K_v^\times}\eta_v$. For $v\not\in S_\eta$, $\eta_v$ has two irreducible extensions to $H(F_v)$; let $\eta_v^+$ denote the extension with value $1$ at  $\iota_v$ and set $\eta_v^-=\eta_v^+ \sgn$, where $\sgn:H(F_v)/K_v^\times\rar \{\pm 1\}$ is the sign character. Define the local packet 
\begin{equation*}
\Pa_{D_v}(\eta_v)=
\begin{cases}
\{\eta_v^+\}, &v\in S_\eta,\\
\{\eta_v^+, \eta_v^-\}, &v\not\in S_\eta.
\end{cases}
\end{equation*}

(ii) For $\epsilon=(\epsilon_v)$ with almost all $\epsilon_v=1$, set $\eta^\epsilon=\prod_v \eta_v^{\epsilon_v}$. Define the global packet
\[
\Pa_D(\eta)=
\begin{cases}
\{\eta^\epsilon:\epsilon \}, &\eta\neq \eta^\iota,\\
\{\eta^\epsilon: \sgn(\epsilon)=1\}, &\eta=\eta^\iota.
\end{cases}
\]
The representation $\eta^\epsilon\in \Pa_D(\eta)$ is embedded into the space of automorphic forms on $[H]$ via the Eisenstein series map $\Phi\in \eta^\epsilon \rar E(\Phi):=\Phi(h)+\Phi(e^\prime h)$. Put $E(\eta^\epsilon):=\{E(\Phi): \Phi\in \eta^\epsilon\}$, then $E(\eta^\epsilon)\cong \eta^\epsilon$ as $H(\ba)$-representations for $\eta^\epsilon\in \Pa_D(\epsilon)$. Furthermore, $L^2(H(F)\ba^\times\backslash H(\ba))$ is the Hilbert space direct sum of $E(\eta^\epsilon)$, with $\eta^\epsilon\in \Pa_D(\eta)$ and $\eta|_{\ba^\times}=1$.

(iii) Equip $H^c(\ba)\backslash H(\ba)$ with the probability measure. For $\Phi_i\in \eta^\epsilon$ and $\varphi_i=E(\Phi_i)$, define
\[
(\Phi_1,\Phi_2)=\int_{H^c(\ba)\backslash H(\ba)}\Phi_1(h_1)\overline{\Phi_2(h_0)}\dd h_0,\quad (\varphi_1,\varphi_2)_{H}=\int_{H(F)\ba^\times\backslash H(\ba)}\varphi_1(h)\overline{\varphi_2(h)}\dd h.
\]
There is $(\varphi_1,\varphi_2)_H=2^{\beta}(\Phi_1,\Phi_2)$.

\subsection{The packet $\So_D(\eta)$}\label{ss_spd}

For $\eta^\epsilon\in \Pa(\eta)$, define the global theta lift of $E(\eta^\epsilon)$ as
\[
\Theta_\psi(\eta^\epsilon)=\Theta(\eta^\epsilon,\psi):=\{\Theta_\psi(\phi,\varphi): \phi\in \ms(\mv_e(\ba)), \varphi\in E(\eta^\epsilon)\}.
\]
It has central character $\eta|_{\ba^\times}$ and is a nonzero irreducible $G^+(\ba)$-representation in the discrete spectrum of $G^+(F)\backslash G^+(\ba)$. There is $\Theta_\psi(\eta^\epsilon)=\otimes_v \theta(\eta_v^{\epsilon_v},\psi_v)$.

Now we define a representation $\Pi^\epsilon_\eta$ from $\Theta(\eta^\epsilon,\psi)$. First, recall from Section ~\ref{sod} that $G(F)\backslash G(\ba)=G^\ddagger(F)\backslash G^+(\ba)$. For a function $f$ on $G^+(F)\backslash G^+(\ba)$, we associate a function
\[
\wtilde{f}(g)=\sum_{\gamma\in G^+(F)\backslash G^\ddagger(F)}f(\gamma g),\quad g\in G^+(\ba),
\]
which is defined on $G^\ddagger(F)\backslash G^+(\ba)$ and hence naturally regarded as a function on $G(F)\backslash G(\ba)$. Define
\[
\wtilde{\Theta}_\psi(\eta^\epsilon):=\{\wtilde{f}: f\in \Theta_\psi(\eta^\epsilon)\},
\]
which consists of functions on $G(F)\backslash G(\ba)$ and is isomorphic to $\Theta_\psi(\eta^\epsilon)$ as an $G^+(\ba)$-representation.

Second, let $\Pi^\epsilon_\eta$ be the $G(\ba)$-representation generated by $\wtilde{\Theta}_\psi(\eta^\epsilon)$ and define
\[
\So_D(\eta)=\{\Pi^\epsilon_\eta: \eta^\epsilon\in \Pa_D(\eta)\}.
\]
$\Pi^\epsilon_\eta$ and $\Theta_\psi(\eta^\epsilon)$ are cuspidal unless $\eta=\eta^\iota$ and $\{v: \epsilon_v=-1\}\subseteq S_D$. (see \cite[Thm. 12.1]{gs2013}).

The strcuture of $\Pi^\epsilon$ is similar to the split case. Locally, for $v\in S_D$, $\theta(\eta_v^{\epsilon_v},\psi_v)$ is an irreducible $G(F_v)$-representation; for $v\not\in S_D$, $\Ind_{G^+(F_v)}^{G(F_v)}\theta(\eta_v^{\epsilon_v},\psi_v)$ is irreducible. Globally, for a space $\mh$ of functions on $G(\ba)$ and $g^\prime\in G(\ba)$, set $\mh^{g^\prime}:=\{f(\cdot g^\prime): f\in \mh\}$; notice that $G^+(\ba)\backslash G(\ba)=G^\ddagger(F)\backslash G(F)=F^\times/F^\ddagger$. The lemma below is parallel to Lemma ~\ref{l_thetacom} and ~\ref{l_decomp}.
\begin{lemma}
\emph{(i)} For $a\in F^\times/F^\ddagger$, choose $g_a\in G^\ddagger(F)\backslash G(F)$ with $\nu(g_a)=a$, then $\wtilde{\Theta}_{\psi_a}(\eta^\epsilon)=\wtilde{\Theta}_\psi(\eta^\epsilon)^{g_a}$, $a\in F^\times/F^\ddagger$. \emph{(ii)} $\Pi^\epsilon_\eta=\oplus_{a\in F^\times/F^\ddagger} \wtilde{\Theta}_{\psi_a}(\eta^\epsilon)\cong \otimes_v \Ind_{G^+(F_v)}^{G(F_v)}\theta(\eta_v^{\epsilon_v},\psi_v)$.
\end{lemma}
\begin{proof}
The proof is the same the proofs of Lemma ~\ref{l_thetacom} and ~\ref{l_decomp} in the split case.
\end{proof}

\subsection{The inner product formula}

For functions $f_1, f_2$ on $G^+(F)\backslash G^+(\ba)$ in the discrete spectrum and with the same central characters, define
\begin{equation}\label{gdpairing}
(f_1,f_2)_{G^+}:=\int_{G^+(F)\ba^\times \backslash G^+(\ba)}f_1(g)\overline{f_2(g)}\dd g.
\end{equation}

%\begin{proof} The measure $\dd h$ on $H(F)\ba^\times\backslash H(\ba)$ can be decomposed as $c_1 \dd h_1\dd c$, where $\dd h_1$ is the measure on $[H_1]$. There is $c_1=1/2$ because $\Vol([PH])=1$ and $\Vol([H_1])=2$, whence the expression for $(\varphi_1,\varphi_2)_H$ follows. The measure $\dd g$ on $G^+(F)\ba^\times \backslash G^+(\ba)$ can be decomposed as $\dd g=c_2 \dd g_1\dd c$, where $\dd g_1$ is the measure on $[G_1]$. There is $c_2=2^{|S_D|}$ because $\Vol(G^+(F)\ba^\times \backslash G^+(\ba))=2$ and $\Vol([G_1])=1$. Applying (\ref{gdpairing}), one gets the following \[ (f_1,f_2)_G=2^{1-|S_D|}\int_{G^\dagger(F)\ba^\times \backslash G^+(\ba)}f_1(g)\overline{f_2(g)}\dd g=2\int_\mc \int_{[G_1]}f_1(g_1g_c)\overline{f_2(g_1g_c)}\dd g_1\dd c. \] \end{proof}

\begin{proposition}\label{ipf_spd}
Suppose $\Pi^\epsilon_\eta$ is cuspidal. For $\varphi_i\in E(\eta^\epsilon)$, $\phi_i\in \ms(\mv_e(\ba))$ and $f_i=\Theta_\psi(\phi_i,\varphi_i)$, there is
\[
(f_1,f_2)_{G^+}=\frac{L(2, \eta\circ \pr)}{\zeta_F(4)}\prod_v \frac{\zeta_{F_v}(4)}{L(2,\eta_v\circ \pr)}\underset{H_1(F_v)}{\int} \overline{(h_v\circ \varphi_{1,v},\varphi_{2,v})}(\omega_{\psi_v}(h_{v})\phi_{1,v},\phi_{2,v})\dd h_v.
\]
\end{proposition}
\begin{proof}

For functions $\varphi_1^\prime,\varphi_2^\prime$ on $[H_1]$, define $(\varphi_1^\prime,\varphi_2^\prime)_{H_1}=\int_{[H_1]}\varphi_1^\prime(h_1)\overline{\varphi_2^\prime(h_1)}\dd h_1$. For functions $f_1^\prime, f_2^\prime$ on $[G_1]$, define $(f_1^\prime,f_2^\prime)_{G_1}=\int_{[G_1]}f_1^\prime(g_1)\overline{f_2^\prime(g_1)}\dd g_1$. From the regularized first-term SW-formula in \cite[Thm. 3.3]{yamana13}, one can apply the standard method to deduce an innner product concerning the lift from $H_1$ to $G_1$: for automorphic forms $\varphi_1^\prime, \varphi_2^\prime$ on $[H_1]$, there are
\begin{equation}\label{ipf_sod}
\int_{[G_1]}\Theta_\psi(\phi_1,\varphi_1^\prime)(g_1)\overline{\Theta_\psi(\phi_2,\varphi_2^\prime)(g_1)}\dd g_1=\int_{H_1(\ba)}\overline{(h_1\circ \varphi_1^\prime,\varphi_2^\prime)_{H_1}}(h_1\circ \phi_1,\phi_2)\dd h_1.
\end{equation}

Give the compact space $\mc:=\ba_{K,D}^\times/{\ba^\times}^2\nu(H(F))$ the Haar measure $\dd c$ with total volume $2^{|S_D|-1}$. For each $c\in \mc$, choose $g_c\in G^+(\ba)$ and $h_c\in H(\ba)$ with $\nu(g_c)=\nu(h_c)$, then
\begin{align}\label{gdpairing2}
\nonumber (f_1,f_2)_{G^+}&=\int_\mc \int_{[G_1]}f_1(g_1g_c)\overline{f_2(g_1g_c)}\dd g_1\dd c,\\
\nonumber (\varphi_1,\varphi_2)_H&=\int_\mc \int_{[H_1]}\varphi_1(h_1h_c)\overline{\varphi_2(h_1h_c)}\dd h_1\dd c.
\end{align}
The remaining argument is the same as in the proof of Proposition ~\ref{ipf_sp}. (In the proof of Proposition ~\ref{ipf_sp}, $\mc$ is defined as $\Nm(\ba_K^\times)/{\ba^\times}^2\Nm(K^\times)$ and given the total volume $1$.)
%Write $\phi_{i,c}=\omega(h_c,g_c)\phi_i$, $\varphi_{i,c}(h_1)=\varphi_i(h_1 h_c)$, and $f_{i,c}=\Theta(\phi_{i,c},\varphi_{i,c})$, then $f_i(g_1g_c)=f_{i,c}(g_1)$, whence $(f_1,f_2)_{G^+}=\int_{\mc}(f_{1,c},f_{2,c})_{G_1}\dd c$ and $(\varphi_1,\varphi_2)_{H}=\int_{\mc}(\varphi_{1,c},\varphi_{2,c})_{H_1}\dd c$. Therefore,\begin{align*} \big(f_1,f_2)_{G^+} =&\int_{\mc}\int_{H_1(\ba)}\overline{(h_1\circ \varphi_{1,c},\varphi_{2,c})_{H_1}}(\omega(h)\phi_{1,c},\phi_{2,c})\dd h_1\dd c\\ =&\int_{H_1(\ba)}\big[\overline{\int_{\mc}(h_1\circ \varphi_{1,c},\varphi_{2,c})_{H_1}\dd c}\big]\ (\omega(h_1)\phi_{1},\phi_{2})\dd h_1\\ =&\int_{H_1(\ba)}\overline{(h_1\circ \varphi_1,\varphi_2)_{H}}(\omega(h_1)\phi_1,\phi_2)\dd h_1. \end{align*} Because $\int_{H_1(F_v)}\overline{(h_{1,v}\circ \varphi_{1,v},\varphi_{2,v})}(\omega_v(h_{1,v})\phi_{1,v},\phi_{2,v})dh_{1,v}=\frac{L(2, \eta_v\circ \pr)}{\zeta_{F_v}(4)}$ almost everywhere, the formula in the proposition follows.
\end{proof}

\begin{remark}
Choose local theta liftings $\theta_v:\ms(\mv_e(F_v))\otimes \eta_v^{\epsilon_v}\rar \theta(\eta_v^{\epsilon_v},\psi_v)\subset \Pi_{\eta_v}^{\epsilon}$ such that $\Theta(\phi,E(\Phi))=\otimes \theta_v(\phi_v,\Phi_v)$. Proposition ~\ref{ipf_spd} implies that for $f_{i,v}=\theta_{\psi_v}(\phi_v,\Phi_v)$, there are
\[
(f_{1,v},f_{2,v})=c_{\eta_v}\underset{H_1(F_v)}{\int} \overline{(h_v\circ \Phi_{1,v},\Phi_{2,v})}(\omega_{\psi_v}(h_{v})\phi_{1,v},\phi_{2,v})\dd h_v,
\]
where $c_{\eta_v}$ are equal to $\frac{\zeta_{F_v}(4)}{L(2,\eta_v\circ \pr)}$ for almost all $v$ with $\prod_v \frac{c_{\eta_v} L(2,\eta_v\circ \pr)}{\zeta_{F_v}(4)}=\frac{2^{\beta}L(2,\eta\circ \pr)}{\zeta_F(4)}$.
\end{remark}

For functions $f_1,f_2$ in the discrete spectrum of $G(F)\backslash G(\ba)$ and with the same central character, define
\[
(f_1,f_2)_G=\int_{G(F)\ba^\times\backslash G(\ba)}f_1(g)\overline{f_2(g)}\dd g=\int_{G^\ddagger(F)\ba^\times\backslash G^+(\ba)}f_1(g)\overline{f_2(g)}\dd g.
\]
Recall the map $\Theta_\psi(\eta^\epsilon)\rar \wtilde{\Theta}_\psi(\eta^\epsilon)$, $f\rar \wtilde{f}$.
\begin{lemma}\label{ipcompare}
Suppose $\Theta_\psi(\eta^\epsilon)$ is cuspidal and $f_1,f_2\in \Theta_\psi(\eta^\epsilon)$, then $(\wtilde{f}_1,\wtilde{f}_2)_G=(f_1,f_2)_{G^+}$.
\end{lemma}
\begin{proof}
By unfolding, there is $(\wtilde{f}_1,\wtilde{f}_2)_G=\sum_{\gamma\in G^+(F)\backslash G^\ddagger(F)} (f_1,f_{2,\gamma})_{G^+}$, where $f_{2,\gamma}(g)=f_2(\gamma g)$. It suffices to show that $(f_1,f_2(\gamma\cdot))_{G^+}=0$ when $\gamma\neq [1]$.

Suppose $\gamma\neq [1]$. We may suppose $\gamma=\smalltwomatrix{1}{}{}{a}$, where $a$ is in $F^\ddagger$ but not in $\nu(H(F))$. Choose $h_a\in H(\ba)$ with $\nu(h_a)=a$. Suppose $f_i=\Theta_\psi(\phi_i,E(\Phi_i))$, with $\phi_i\in \ms(\mv_e(\ba))$ and $\varphi_i\in E(\eta^\epsilon)$. The value of $f_i$ on the unipotent subgroup $\{\smalltwomatrix{1}{u}{}{1}\}\subset G(\ba)$ is
\[
f_i\smalltwomatrix{1}{u}{}{1}=\sum_{\xi\in \mv_e(F)}\underset{[H_1]}{\int}\omega_\psi(u)\phi_i(h_1^{-1}\xi)\overline{\varphi_i(h_1)}\dd h_1=\sum_{\xi\in \mv_e^\circ(F)}\underset{[H_1]}{\int}\psi(\frac{1}{2}\Tr_D(u(\xi,\xi)))\phi(h_1^{-1}\xi)\overline{\varphi_i(h_1)}\dd h_1.
\]
Here the summand for $\xi=0$ because $f_i$ are cuspidal. Similarly, with (\ref{wr_s2d}), one computes that
\begin{equation}\label{nvalue}
f_i\big(\gamma\smalltwomatrix{1}{u}{}{1}\big)=\sum_{\xi\in \mv_e^\circ(F)}\int_{[H_1]}\psi(\frac{1}{2}\Tr_D(a^{-1}u(\xi,\xi)))\phi(h_a^{-1}h_1^{-1}\xi)\overline{\varphi_i(h_1h_a)}\dd h_1
\end{equation}

So $f_1\smalltwomatrix{1}{u}{}{1}$ is a sum of characters of shape $\psi(\frac{1}{2}\Tr_D(u(\xi,\xi)))$ and $f_2\big(\gamma\smalltwomatrix{1}{u}{}{1}\big)$ is a sum of characters of the shape $\psi(\frac{1}{2}\Tr_D(ua^{-1}(\xi,\xi)))$, $\xi\in \mv_e^\circ(F)$. Because $a\not\in \nu(H(F))$, these two families of characters have no intersection (c.f. equation (\ref{sm2})), whence
\[
\int_{D_0(F)\backslash D_0(\ba)} f_1\smalltwomatrix{1}{u}{}{1}\overline{f_2\big(\gamma \smalltwomatrix{1}{u}{}{1}\big)}\dd u=0.
\]
It follows that $(f_1,f_{2,\gamma})_{G^+}=0$.
\end{proof}

\subsection{The isomorphism $\SO(V_D)\cong PG$}

One needs to utilize the isomorhism $\SO(V_D)\cong PG$.
Write $V_D=Fv_+\oplus U\oplus Fv_-$, where $v_+=\smalltwomatrix{0}{1_D}{0}{0}$, $v_-=\smalltwomatrix{0}{0}{1_D}{0}$ and $U=D_0$ is embedded into $V_D$ by $X\rar \smalltwomatrix{X}{}{}{\overline{X}}$. There is $q(x^\prime  v_++X+x^\pprime v_-)=x^\prime x^\pprime-\det X$.

Let $P=U\rtimes M$ of $\SO(V_D)$ be the parabolic subgroup that stabilizes the line $Fv_+$. Let $\mrp$ denote the Siegel parabolic subgroup of $PG$, $\mrn$ be its unipotent radical, and $\mrm$ be its Levi subgroup consisting of diagonal matrices. The element $u\in U$ correponds to $\smalltwomatrix{1_2}{u}{}{1_2}\in \mathrm{N}$ and thus we identify $U$ with $\mrn$. The element  $[\smalltwomatrix{A}{}{}{x{\overline{A}^{-1}}}]\in \mathrm{M}$ corresponds to the element $m(A,x)\in M$ that acts by
\[
v_+\rar x^{-1}\Nm(A)\cdot v_+,\quad   X\in U\rar AXA^{-1}\in U,\quad  v_-\rar x\Nm(A)^{-1}\cdot v_-.
\]

For $v_0\in U^\circ(F)$, there is $\psi_{v_0}(u):=\psi(\Tr(v_0u))$ for $u\in [U]$. Write $U=Fv_0\oplus V_0$ and put $E=F\oplus Fv_0\cong F(\sqrt{q(v_0)})$, then there is the following identification,
\begin{align*}\label{groupT}
\SO(V_0)=E^\times/F^\times\subset \SO(U)&\isoto \bt:=\left\{[\smalltwomatrix{A}{}{}{\Nm(A){\bar{A}^{-1}}}]: {\bar{A}}TA=\Nm(A) T\right\}\subset \mrm,\\
[A]&\rar [\smalltwomatrix{A}{}{}{\Nm(A){\bar{A}^{-1}}}]
\end{align*}

\subsection{The global Bessel period functional}

On $\Pi^\epsilon_\eta$, the computation of the global Bessel period functional $P$ with respect to $\psi_{v_0}$ and a character $\chi$ of $[\SO(V_0)]$ is analogous to the split case. Write $E=F\oplus F\cdot {v_0}$.

(i) For a function $f$ on $G^+(F)\backslash G^+(\ba)$, there is $P(f)=\int_{[\SO(V_0)]}\ell_{v_0}(h\circ f)\chi(h)\dd h$, where
\[
\ell_{v_0}(f)=\int_{[D_0]}f\smalltwomatrix{1}{u}{}{1}\psi(\Tr_D(uv_0))\dd u.
\]

(ii) For $\Phi\in \eta^\epsilon,\phi\in \ms(\mv_e(\ba))$ and $\xi\in \mv_e^\circ (F)$, we introduce
\[
P(\phi,\Phi,\xi):=\int_{H_1(\ba)}\overline{\Phi(h)}\phi(h^{-1}\xi)\dd h,\quad P_{v}(\phi_v,\Phi_v,\xi):=\int_{H_1(F_v)}\overline{\Phi_v(h_v)}\phi_v(h_v^{-1}\xi)\dd h_v.
\]

(iii) Recall the map $\wtilde{q}_e:D\rar D_0$, $x\rar \bar{x}ex$ in Section ~\ref{gu1dgroup}. Given $v_0\in U^\circ(F)=D_0^\circ(F)$, there is a unique class $[a_{v_0}]\in F^\times/(\Nm(E^\times)\cup \delta^\prime\Nm(E^\times))$ such that $v_0\in -\frac{a_{v_0}}{2}\wtilde{q}_{v_0}(\mv_{v_0}^\circ)$. Recall that $D=K\oplus Ke^\prime$ and $H(F)=K^\times\sqcup e^\prime K^\times$.

(iv) For $\phi\in \ms(\mv_e(\ba))$, $\Phi\in \eta^\epsilon$ and $h^\prime\in H(\ba)$, write $\leftup{h^\prime}{\phi}(x)=\phi({h^\prime}^{-1}x)$ and $\leftup{h^\prime}{\Phi}(h)=\Phi(h h^\prime)$.

(v) When constructing $\Pi^\epsilon_\eta$, we introduced a map $\Theta_\psi(\eta^\epsilon)\rar \wtilde{\Theta}_\psi(\eta^\epsilon)$ and the latter space generates $\Pi^\epsilon_\eta$. We need a relation between $P(\wtilde{f})$ and $P(f)$ for $f\in \Theta_\psi(\eta^\epsilon)$.

\begin{lemma}\label{pcompare}
Suppose $a\in F^\times$ and $\Theta_{\psi_a}(\eta^\epsilon)$ is cuspidal. When $E\not\cong K$ or $E\cong K$ with $a\neq a_{v_0}$ in $F^\times/F^\ddagger$, then $\ell_{v_0}$ is trivial on $\Theta_{\psi_a}(\eta^\epsilon)$ and $\wtilde{\Theta}_{\psi_a}(\eta^\epsilon)$. When $E\cong K$ and $a=a_{v_0}$ in $F^\times/F^\ddagger$, let $[a^\prime]$ be a class in $F^\ddagger/\nu(H(F))$ satisfying $a=a_{v_0}a^\prime$, then for $f\in \Theta_{\psi_a}(\eta^\epsilon)$, there is $\ell_{v_0}(\wtilde{f})=\ell_{v_0}(f_{\gamma_{a^\prime}})$ and $P(\wtilde{f})=P(f_{\gamma_{a^\prime}})$, where $f_{\gamma_{a^\prime}}(g)=f(\gamma_{a^\prime}g)$ and $\gamma_{a^\prime}=\smalltwomatrix{1}{}{}{a^\prime}$.
\end{lemma}
\begin{proof}
Choose a set of representations $\{a_i: 1\leq i\leq 2^{|S_D|-2}\}$ of $F^\ddagger/\nu(H(F))$ and write $\gamma_{a_i}=\smalltwomatrix{1}{}{}{a_i}$. Choose $h_{a_i}\in H(\ba)$ satisfyng $\nu(h_{a_i})=a_i$. By definition,
\[
\ell_{v_0}(\wtilde{f})=\sum_{i} \int_{[D_0]}f\big(\gamma_{a_i}\smalltwomatrix{1}{u}{}{1}\big)\psi(\Tr_D(v_0u))\dd u=\sum_i \ell_{v_0}(f_{\gamma_i}),
\]
Here $f_{\gamma_{a_i}}(g)=f(g\gamma_{a_i})$. As we derive (\ref{nvalue}) in the proof of Lemma ~\ref{ipcompare}, there is
\[
f\big(\gamma_{a_i}\smalltwomatrix{1}{u}{}{1}\big)=\sum_{\xi\in \mv_e^\circ(F)}\int_{[H_1]}\psi_a(\frac{1}{2}\Tr_D(a_i^{-1}u(\xi,\xi)))\phi(h_{a_i}^{-1}h_1^{-1}\xi)\overline{\varphi_i(h_1h_a)}\dd h_1.
\]
Note that $(\xi,\xi)$ is $\wtilde{q}_e(\xi)$. So $\ell_{v_0}(f_{\gamma_i})$ is zero if $v_0\not\in -\frac{aa_i^{-1}}{2}\wtilde{q}_e(\mv_e^\circ(F))$. From here, one easily deduces the lemma.
\end{proof}

\begin{proposition}\label{functionalT_sod}
Suppose $\Pi^\epsilon_\eta\in \So_D(\eta)$ is cuspidal.
\begin{itemize}
\item[(i)] When $E\not\cong K$, $\ell_{v_0}$ is zero on $\Pi^\epsilon_\eta$ 
\item[(ii)] When $E\cong K$ and $a\in F^\times/F^\ddagger$, $\ell_{v_0}$ is zero on $\wtilde{\Theta}_{\psi_a}(\eta^\epsilon)$ if $a\neq a_{v_0}$ in $F^\times/F^\ddagger$. 
\item[(iii)] Suppose $E\cong K$ and $f\in  \Theta_{\psi_{a_{v_0}}}(\phi, E(\Phi))$ with $\phi\in \ms(\mv_e(\ba))$ and $\Phi\in \eta^\epsilon$, choose $\xi_{v_0}\in \mv_e(F)$ with $\wtilde{q}_e(\xi_{v_0})=-2v_0/a_{v_0}$, then
\[
\ell_{v_0}(\wtilde{f})=\ell_{v_0}(f)=P(\phi,\Phi,\xi_{v_0})+P(\leftup{e^\prime}{\phi},\leftup{e^\prime}{\Phi},e^\prime\xi_{v_0}).
\]
Furthermore, when $v_0=-\frac{a_{v_0}}{2}\cdot e$ and $\xi_{v_0}=1$, there is
\[
P(\wtilde{f})=P(f)=
\begin{cases}
2^\beta\delta(\chi,\eta)P(\phi,\Phi,\xi_{v_0})+2^\beta\delta(\chi,\eta^\iota)P(\leftup{e^\prime}{\phi},\leftup{e^\prime}{\Phi},e^\prime\xi_{v_0}), &\eta\neq \eta^{\iota},\\
2^\beta\delta(\chi,\eta)P(\phi,\Phi,\xi_{v_0}), &\eta=\eta^\iota.
\end{cases}
\]
\item[(iv)] $P$ is nonzero on $\Pi^\epsilon_\eta$ if and only if $E\cong K$ and $\chi\in\{\eta, \eta^\iota\}$.
\end{itemize}
\end{proposition}
\begin{proof}
The proof is completely analogous to the proof of Proposition ~\ref{functionalT}.
\end{proof}

\subsection{The regularized local Bessel period integral}

The computation of the regularized local Bessel period integral is completely analogous to Proposition ~\ref{lbp_sp}; we have the same inner product formulas (c.f. Prop. ~\ref{ipf_sp} and Prop. ~\ref{ipf_spd}) and the same type of measure decompositions (c.f. Section ~\ref{ss_md2} and ~\ref{fiberation2}).

There is $\Pi^{\epsilon_v}_{\eta_v}=\oplus_{a_v\in F_v^\times\backslash \nu(H(F_v))} \theta(\eta_v^{\epsilon_v},\psi_{v,a_v})$ and
\[
\mpp_v(s;f_{1,v},f_{2,v})=\int_{\SO(V_0)_{F_v}}\ell_{v,T}(h_v\circ f_{1,v},f_{2,v})\chi_v(h_v)\Delta(h_v)^s \dd h_v,
\]
where $\ell_{v,v_0}(f_{1,v},f_{2,v}):=\int_{\bn(F_v)}(\Pi_v(N_v)f_{1,v},f_{2,v})\psi_{v,v_0}(N_v)\dd N_v$ for $f_{1,v},f_{2,v}\in \Pi_{\eta_v}^{\epsilon_v}$. For given $v_0$, there is a unique class $[a_{v,v_0}]\in F_v^\times/\nu(H(F_v))$ such that $v_0\in -\frac{a_{v,v_0}}{2}\wtilde{q}_{v_0}(\mv_{v_0}^\circ(F_v))$.
\begin{proposition}\label{lbp_spd}
Suppose $a_v\in F_v^\times$ and $f_{i,v}=\theta_{\psi_{v,a_v}}(\phi_{i,v},\Phi_{i,v})\in \theta(\eta_v^{\epsilon_v},\psi_{v,a_v})$ with $\phi_{i,v}\in \ms(\mv_e(F_v))$, $\Phi_{i,v}\in \eta_v^{\epsilon_v}$.
\begin{itemize}
\item[(i)] When $E_v\neq K_v$ or $E_v=K_v$ with $a_v\neq a_{v,v_0}$, $\ell_{v,v_0}$ is zero on $\theta(\eta_v^{\epsilon_v},\psi_{v,a_v})$. 
\item[(ii)] When $E_v=K_v$ and $a_v=a_{v,v_0}$, choose $\xi_{v,v_0}\in \mv_e(F_v)$ with $\wtilde{q}_e(\xi_{v,v_0})=-2v_0/a_{v,v_0}$ and set $C_v=2^{{\beta_v}-1}|a_{v,v_0}||2\Nm(v_0)|^{-\frac{1}{2}}c_{\eta_v}c_{E_v}$, then $\ell_{v,v_0}(f_{1,v},f_{2,v})$ is equal to
\begin{equation*}
C_v\cdot 
\begin{cases}
\sum_{h_v\in\{1, e^\prime\}}P_v(\leftup{h_v}{\phi_{1,v}},\leftup{h_v}{\Phi_{1,v}},h_v\xi_{v,v_0})\overline{P_v(\leftup{h_v}{\phi_{2,v}},\leftup{h_v}{\Phi_{2,v}},h_v\xi_{v,v_0})}, &\eta_v^{\iota_v}\neq \eta_v,\\
P_v(\phi_{1,v},\Phi_{1,v},\xi_{v,v_0})\overline{P_v(\phi_{2,v},\Phi_{2,v},\xi_{v,v_0})}, &\eta_v^{\iota}=\eta_v.
\end{cases}
\end{equation*}
Furthermore, when $v_0=-\frac{a_{v,v_0}}{2}\cdot e$ and $\xi_{v,v_0}=1_v$, $\mpp_v(s;f_{1,v},f_{2,v})$ is equal to
\begin{equation*}
C_v\cdot
\begin{cases}
\sum_{h_v\in\{1, e^\prime\}}\mpp(s,\chi_v\overline{\eta_v^{\tau_v}})P_v(\leftup{h_v}{\phi_{1,v}},\leftup{h_v}{\Phi_{1,v}},h_v\xi_{v,v_0})\overline{P_v(\leftup{h_v}{\phi_{2,v}},\leftup{h_v}{\Phi_{2,v}},h_v\xi_{v,v_0})}, &\eta_v\neq \eta_v^{\iota},\\
\mpp(s,\chi_v\overline{\eta_v})P_v(\phi_{1,v},\Phi_{1,v},\xi_{v,v_0})\overline{P_v(\phi_{2,v},\Phi_{2,v},\xi_{v,v_0})}, &\eta_v=\eta_v^\iota.
\end{cases}
\end{equation*}
\end{itemize}
\end{proposition}
\begin{proof}
The proof is the same as the proof of Proposition ~\ref{lbp_sp}.
\end{proof}

The above proposition tells that $\mpp_v(s,-)$ is zero on $\Pi_v^{\epsilon_v}\otimes \Pi_v^{\epsilon_v}$ when $E_v\neq K_v$ and factors through $\theta(\eta_v^{\epsilon_v},\psi_{v,a_{v,v_0}})\otimes \theta(\eta_v^{\epsilon_v},\psi_{v,a_{v,v_0}})$ when $E_v=K_v$. We then set
\[
\mpp^\sharp_v(f_{1,v},f_{2,v})=\frac{\mpp_v(s; f_{1,v},f_{2,v})}{\zeta_{F_v}(2)\zeta_{F_v}(4)L(s,\Pi^{\epsilon_v}_{\eta_v},\chi_v)}\big|_{s=0},\quad  f_{1,v},f_{2,v}\in \Pi_{\eta_v}^{\epsilon_v}.
\]

\subsection{The conclusion}

Proposition ~\ref{lnvs_sp}, Theorem ~\ref{lc_sp_split}, and Theorem ~\ref{bp_sp} hold for a cuspidal member $\Pi$ of $\So_D(\eta)$; with Proposition ~\ref{functionalT_sod} and ~\ref{lbp_spd}, the arguments are the same as in the split case and hence are skipped here. We only note that when deducing the Bessel period formula on $\Pi^\epsilon_\eta$, one uses Lemma ~\ref{ipcompare} and ~\ref{pcompare} to reduce the question to a Bessel period formula on $\Theta_\psi(\eta^\epsilon)$ and the remaining argument is the same as in the proof of Theorem ~\ref{bp_sp}. This reduction is necessary because when $D$ is nonsplit, $\Pi^\epsilon_\eta$ is generated by $\wtilde{\Theta}_\psi(\eta^\epsilon)$ but not $\Theta_\psi(\eta^\epsilon)$.

%Suppose $a\in F^\times$. When $E\not\cong K$, $P$ is zero on $\Theta(\eta^\epsilon, \psi_a)$. When $E\cong K$, we may suppose $v_0\in F^\times\cdot e$, then $P$ is zero on $\Theta(\eta^\epsilon, \psi_a)$ unless $v_0\in -\frac{a}{2}\nu(H(F))\cdot e$. When $E\cong K$ and $v_0\in -\frac{a}{2}\nu(H(F))\cdot e$, choose $\xi_{v_0}$ with $(\xi_{v_0},\xi_{v_0})=-\frac{2v_0}{a}$, then for $f=\Theta(\phi,E(\Phi))$ with $\phi\in \ms(\mv_e(\ba))$ and $\Phi\in \eta^\epsilon\in \Pa_D(\eta)$, there is

%\addcontentsline{toc}{chapter}{\numberline{}Bibliography}
%\bibliographystyle{plain}
%\bibliography{bib-representations,bib-cohomology,qiu,preprint}

\def\cprime{$'$} \def\cprime{$'$}

\end{document}